\documentclass[twoside]{article}
\usepackage[letterpaper, left=3.5cm, right=3.5cm, top=4cm, bottom=2.5cm]{geometry}
\usepackage{graphicx}
\usepackage{amsmath,amssymb,amsthm}
\usepackage{authblk}
\usepackage[utf8]{inputenc}
\usepackage{microtype}
\usepackage{mathrsfs} 
\usepackage{enumitem}
\usepackage{calc}
\usepackage{esint}
\usepackage{fancyhdr}
\usepackage{xcolor}
\usepackage[labelfont = bf]{caption}
\usepackage{doi}
\usepackage{subfigure}
\usepackage[numbers]{natbib}
\usepackage{float}
\usepackage{stmaryrd}
\usepackage{mathtools}
\usepackage{booktabs}
\usepackage{colortbl}
\usepackage{tabulary}
\usepackage{diagbox}
\usepackage{caption}

\usepackage{physics}
\usepackage{amsmath}
\usepackage{tikz}
\usepackage{mathdots}
\usepackage{yhmath}
\usepackage{cancel}
\usepackage{color}
\usepackage{siunitx}
\usepackage{array}
\usepackage{multirow}
\usepackage{amssymb}
\usepackage{gensymb}
\usepackage{tabularx}
\usepackage{extarrows}
\usepackage{booktabs}
\usetikzlibrary{fadings}
\usetikzlibrary{patterns}
\usetikzlibrary{shadows.blur}
\usetikzlibrary{shapes}

\newtheorem{theorem}{Theorem}[section]
\numberwithin{equation}{section}
\newtheorem{proposition}[theorem]{Proposition}

\newtheorem{remark}[theorem]{Remark}
\newtheorem{lemma}[theorem]{Lemma}

\newtheorem{algorithm}[theorem]{Algorithm}

\usepackage{titlesec}
\titleformat{\section}{\normalfont\scshape\centering}{\thesection.}{0.5em}{}
\titleformat*{\subsection}{\itshape}
\titleformat*{\subsubsection}{\itshape}

\providecommand{\keywords}[1]
{
	{\small\emph{Keywords:} #1}
}

\providecommand{\MSC}[1]
{
	{\small\emph{AMS MSC (2020):~~} #1}
}

\definecolor{denim}{rgb}{0.08, 0.38, 0.74}
\definecolor{byzantium}{rgb}{0.44, 0.16, 0.39} 
\definecolor{shamrockgreen}{rgb}{0.0, 0.62, 0.38} 

\usepackage{hyperref}
\hypersetup{
	colorlinks=true,
	linkcolor=denim,
	citecolor = shamrockgreen,
	filecolor=magenta, 
	urlcolor=byzantium,
}

\providecommand{\jumptmp}[2]{#1\llbracket{#2}#1\rrbracket}
\providecommand{\jump}[1]{\jumptmp{}{#1}}

\begin{document}
	\setlength{\abovedisplayskip}{5.5pt}
	\setlength{\belowdisplayskip}{5.5pt}
	\setlength{\abovedisplayshortskip}{5.5pt}
	\setlength{\belowdisplayshortskip}{5.5pt}

	\title{\vspace*{-1.5cm}\textit{A priori} and \textit{a posteriori} error identities\\ for the scalar Signorini problem}
	\author[1]{Sören Bartels\thanks{Email: \texttt{bartels@mathematik.uni-freiburg.de}}}
	\author[2]{Thirupathi Gudi\thanks{Email: \texttt{gudi@iisc.ac.in}}}
	\author[3]{Alex Kaltenbach\thanks{Email: \texttt{kaltenbach@math.tu-berlin.de}}}
	\date{\today}
	\affil[1]{\small{Department of Applied Mathematics, University of Freiburg, Hermann--Herder--Str. 10,  79104  Freiburg}}
	\affil[2]{\small{Department of Mathematics, Indian Institute of Science Bangalore, India 560012}}
	\affil[3]{\small{Institute of Mathematics, Technical University of Berlin, Stra\ss e des 17.\ Juni 135, 10623 Berlin}}
	\maketitle

	\pagestyle{fancy}
	\fancyhf{}
	\fancyheadoffset{0cm}
	\addtolength{\headheight}{-0.25cm}
	\renewcommand{\headrulewidth}{0pt} 
	\renewcommand{\footrulewidth}{0pt}
	\fancyhead[CO]{\textsc{Error identities for the scalar Signorini problem}}
	\fancyhead[CE]{\textsc{S. Bartels, T. Gudi and A. Kaltenbach}}
	\fancyhead[R]{\thepage}
	\fancyfoot[R]{}
	
	\begin{abstract}
		In this paper, on the basis of a (Fenchel) duality theory on the continuous level, we derive an \textit{a posteriori} error identity for arbitrary conforming approximations of the primal formulation and the dual formulation of the scalar Signorini problem.
		In addition, on the basis of a (Fenchel) duality theory on the discrete level, we derive an \textit{a priori} error identity that applies
		to the approximation of the primal formulation using the Crouzeix--Raviart element and  to  the approximation of the dual formulation using the Raviart--Thomas element, and leads~to quasi-optimal error decay rates without imposing additional~assumptions~on~the~contact~set and in arbitrary space dimensions.
	\end{abstract}
	
	
	\MSC{35J20; 49J40; 49M29; 65N30; 65N15; 65N50.}
	
	\section{Introduction}\thispagestyle{empty} 
	
	\hspace*{5mm}The \emph{scalar Signorini problem} is a model problem that captures non-trivial~effects~present \hspace*{-0.1mm}in \hspace*{-0.1mm}elastic \hspace*{-0.1mm}contact \hspace*{-0.1mm}problems.
	\hspace*{-0.1mm}It \hspace*{-0.1mm}is \hspace*{-0.1mm}a \hspace*{-0.1mm}non-linear problem \hspace*{-0.1mm}as \hspace*{-0.1mm}it \hspace*{-0.1mm}contains \hspace*{-0.1mm}a \hspace*{-0.1mm}non-linear \hspace*{-0.1mm}boundary~\hspace*{-0.1mm}\mbox{condition}: in a bounded domain $\Omega\subseteq \mathbb{R}^d$, $d\in \mathbb{N}$,
	the solution  $u\colon \overline{\Omega}\to \mathbb{R}$ of the scalar Signorini problem (\textit{i.e.}, the \emph{displacement field}) 
	on a part of the (topological) boundary $\Gamma_C\subseteq \partial\Omega$ (\textit{i.e.}, the \emph{contact set}) 
	is greater or equal to $\chi \colon \Gamma_C\to \mathbb{R}$ (\textit{i.e.}, the \emph{obstacle}) (\textit{cf}.\ \cite{Sigonrini59}). It can be expressed in form of a convex minimization problem with an optimality condition
	given~via~variational~inequality~(\textit{cf}.~\cite{Fichera64}).

	\subsection{Related contributions}\enlargethispage{15mm}

	\hspace*{5mm}Finite element approximation as well as its \textit{a priori} and \textit{a posteriori} error analysis~for~unilateral contact problems is an active area of research for many decades. There is a vast literature on this topic; including conforming, non-conforming, and hybrid finite element methods~(\textit{cf}.~\cite{Bel2000,BelBre01,BelRen03,BelBerAdeVoh09}), mixed (\textit{cf}.\ \cite{Sch11}), and mortar finite element methods (\textit{cf}.\ \cite{BelHilLab97}). These methods typically employ element-wise affine or quadratic polynomial finite elements, due to limited regularity of the solution of these nonlinear contact problems (\textit{cf}.\ Remark \ref{rem:regularity}).
	
	Due to scarcity, we refer to a few articles and references therein on this topic:
	\begin{itemize}[noitemsep,topsep=2pt]
		\item  In the context of \textit{a posteriori} error analyses that provide reliable and efficient error bounds, we refer the reader to the contributions \cite{HilSer05,WeiWoh09,BelBerAdeVoh12,KraVeeWal15}.
		
		\item   In the context of \textit{a priori} error analyses, in \cite{HueWoh05}, assuming that the solution lies in $H^{s+1}(\Omega)$, $s \in (\frac{1}{2},1]$, and that the contact  set $\Gamma_C$ has a certain regularity,
		quasi-optimal \textit{a priori} error estimates are derived; in  \cite{HilRen12}, assuming, again, that the solution lies in $H^{s+1}(\Omega)$, $s \in (\frac{1}{2},1]$, but no additional regularity of the contact set $\Gamma_C$,
		improved  quasi-optimal \textit{a priori} error estimates are derived; recently, in \cite{GuiHil15}, important and interesting results are established to obtain quasi-optimal \textit{a priori} error estimates for conforming finite element methods in two and three space dimensions without additional assumptions on the contact set $\Gamma_C$ when the solution lies in $H^{s+1}(\Omega)$, $s\hspace*{-0.1em} \in\hspace*{-0.1em}(\frac{1}{2},\frac{1}{2}+\frac{k}{2}]$, where~${k\hspace*{-0.1em}\in \hspace*{-0.1em}\{1,2\}}$~is~the~polynomial~degree~\mbox{being}~used;
		
		in \cite{ChoHil13, ChoHilYve15}, Nitsche's type methods with symmetric and non-symmetric variants are proposed and analyzed for the contact problem with  $H^{s+1}(\Omega)$, $s \in (\frac{1}{2},1]$, regular solution and  derived optimal order convergence in $H^1$-norm. A penalty method is formulated and its convergence at continuous and discrete level are studied in \cite{ChoHil13-penalty} for the two dimensional contact problem with $H^{s+1}(\Omega)$, $s \in (\frac{1}{2},1]$,   regular solution but without any assumption~on~the~contact~set~$\Gamma_C$, and further therein, the authors have established optimal convergence rates by deriving necessary relation between penalty parameter and the mesh size.\vspace*{-1mm}
	\end{itemize}

	\subsection{New contributions}\vspace*{-1mm}
	
	\hspace*{5mm}The contributions of the present paper to the existing literature are two-fold:\enlargethispage{5mm}
	
	\begin{itemize}[noitemsep,topsep=2pt]
		\item On the basis of (Fenchel) duality theory on the continuous level (combining approaches  from  \cite{RepinValdman18}, \cite{ChambollePock20}, and \cite{BartelsKaltenbach23}), we derive an \textit{a posteriori} error identity that applies to arbitrary conforming approximations of the primal  formulation and the dual formulation~of~the~scalar Signorini problem. More precisely, denoting by $u\in K$ and $z\in K^*$ the primal and dual solution, respectively,  for admissible approximations $v\in K$~and~${y\in K^*}$,~it~holds~that
		\begin{align}\label{intro:aposteriori}
			\begin{aligned} 
				&\tfrac{1}{2}\|\nabla v-\nabla u\|_{\Omega}^2+\langle z\cdot n,v-\chi\rangle_{\partial\Omega}-\langle g,v-\chi\rangle_{\Gamma_N}\\&\quad+\tfrac{1}{2}\|y-z\|_{\Omega}^2+\langle y\cdot n,u-\chi \rangle_{\partial\Omega}-\langle g,u-\chi\rangle_{\Gamma_N}	\\&= \tfrac{1}{2}\| \nabla v-y\|_{\Omega}^2+\langle y\cdot n, v-\chi\rangle_{\partial\Omega}-\langle g, v-\chi\rangle_{\Gamma_N} \,,
			\end{aligned}
		\end{align}
		In addition, the induced local refinement indicators of the primal-dual gap (\textit{a posteriori}) error estimator (\textit{i.e.}, the right-hand side of) can be employed in adaptive mesh-refinement. 
		
		\item On the basis of (Fenchel) duality theory on the discrete level, 
		analogously to the \textit{a posteriori} error identity  on the continuous level \eqref{intro:aposteriori}, we derive an \textit{a priori} error identity~that applies to the approximation of the primal formulation using the Crouzeix--Raviart element (\textit{cf}.\ \cite{CR73}) and the  approximation of the dual formulation using the Raviart--Thomas element~(\textit{cf}.~\cite{RT77}). More precisely,  denoting by $u_h^{cr}\in K_h^{cr}$ and $z_h^{rt}\in K_h^{rt,*}$ the discrete primal~and~discrete~dual solution, respectively,  for admissible approximations $v_h\hspace*{-0.1em}\in \hspace*{-0.1em} K_h^{cr}$ and $y_h\hspace*{-0.1em}\in \hspace*{-0.1em}K_h^{rt,*}$,~it~holds~that
		\begin{align}\label{intro:apriori}
			\begin{aligned} 
				&\tfrac{1}{2}\|\nabla_h v_h-\nabla_h u_h^{cr}\|_{\Omega}^2+( z_h^{rt}\cdot n,\pi_hv_h-\chi_h)_{\Gamma_C}\\&\quad
				+\tfrac{1}{2}\|\Pi_h y_h-\Pi_h z_h^{rt}\|_{\Omega}^2+(y_h\cdot n,\pi_h u_h^{cr}-\chi_h)_{\Gamma_C}\\&=
				\tfrac{1}{2}\| \nabla_h v_h-y_h\|_{\Omega}^2+(y_h\cdot n, \pi_h v_h-\chi_h)_{\Gamma_C}\,.
			\end{aligned}
		\end{align}
		From \hspace*{-0.1mm}the \hspace*{-0.1mm}\textit{a \hspace*{-0.1mm}priori} \hspace*{-0.1mm}error \hspace*{-0.1mm}identity \hspace*{-0.1mm}\eqref{intro:apriori}, \hspace*{-0.1mm}we \hspace*{-0.1mm}derive \hspace*{-0.1mm}quasi-optimal \hspace*{-0.1mm}error \hspace*{-0.1mm}decay~\hspace*{-0.1mm}rates~\hspace*{-0.1mm}without~\hspace*{-0.1mm}impo-sing additional assumptions on the regularity of the contact set $\Gamma_C$ for arbitrary~dimensions. This improves the existing literature (\textit{cf}.\ \cite{Wang99,HuaWang07,LiHuaLian21}) on \textit{a priori} error analyses for approximations of the scalar Signorini problem using the Crouzeix--Raviart element.\vspace*{-1mm}
	\end{itemize}

	\subsection{Outline}\vspace*{-1mm}
	
	\hspace{5mm}\textit{This article is organized as follows:} In Section \ref{sec:preliminaries}, we introduce the~notation, the~relevant~function spaces and  finite element spaces. In Section \ref{sec:continuous_signorini}, 
	a (Fenchel) duality theory for the continuous scalar Signorini problem is developed. 
	This (Fenchel) duality theory is used in Section \ref{sec:aposteriori} in the derivation of an \textit{a posteriori} error identity.
	In Section \ref{sec:discrete_signorini}, 
	a discrete (Fenchel) duality theory for the discrete scalar Signorini problem is developed. 
	This discrete (Fenchel) duality theory,~in~turn, is used in Section \ref{sec:apriori} in the derivation of an \textit{a priori} error identity, which, in turn, is used to derive error decay rates given only fractional regularity assumptions on the solution and the obstacle.
	In Section \ref{sec:experiments}, we carry out numerical experiments that support the findings~of~Section~\ref{sec:aposteriori}~and~Section~\ref{sec:apriori}.\enlargethispage{12.5mm}\newpage
	
	\section{Preliminaries}\label{sec:preliminaries}
	
	\hspace{5mm}Throughout the  article,  let ${\Omega\subseteq \mathbb{R}^d}$, ${d\in\mathbb{N}}$, be a bounded simplicial Lipschitz domain such that $\partial\Omega$ is divided into three disjoint
	 (relatively) open sets:~a~\mbox{Dirichlet}~part~${\Gamma_D\subseteq \partial\Omega}$~with~$\vert \Gamma_D\vert>0$\footnote{For a (Lebesgue) measurable set $M\subseteq \mathbb{R}^d$, $d\in \mathbb{N}$, we denote by $\vert M\vert $ its $d$-dimensional Lebesgue measure. For a $(d-1)$-dimensional submanifold $M\subseteq \mathbb{R}^d$, $d\in \mathbb{N}$, we denote by $\vert M\vert $ its $(d-1)$-dimensional~Hausdorff~measure.}, a Neumann part $\Gamma_N\subseteq \partial\Omega$, and a contact part $\Gamma_C\subseteq \partial\Omega$ such that $\partial\Omega=\overline{\Gamma}_D\cup\overline{\Gamma}_N\cup\overline{\Gamma}_C$.\enlargethispage{5mm}
	
	\subsection{Standard function spaces}
	
	\hspace{5mm}For a (Lebesgue) measurable set $\omega\subseteq \mathbb{R}^n$, $n\in \mathbb{N}$, and  (Lebesgue) measurable functions or vector fields $v,w\colon \omega\to \mathbb{R}^{\ell}$, $\ell\in \mathbb{N}$, we employ~the~inner~product $(v,w)_{\omega}\coloneqq \int_{\omega}{v\odot w\,\mathrm{d}x}$,
	%	\begin{align*}
		%		(u,v)_{\omega}\coloneqq \int_{\omega}{u\odot v\,\mathrm{d}x}\,,
		%	\end{align*}
	whenever the right-hand side is well-defined, where $\odot\colon \mathbb{R}^{\ell}\times \mathbb{R}^{\ell}\to \mathbb{R}$ either denotes scalar multiplication~or~the Euclidean inner product. The integral mean over a (Lebesgue)~\mbox{measurable}~set~${\omega\subseteq\mathbb{R}^n}$,~${n\in \mathbb{N}}$,~with $\vert \omega\vert>0$ of an integrable function or vector field $v\colon \omega\to \mathbb{R}^{\ell}$, $\ell\in \mathbb{N}$,~is~defined~by~${\langle v\rangle_\omega\coloneqq \frac{1}{\vert \omega\vert}\int_{\omega}{v\,\mathrm{d}x}}$. 
	 
	For $m\in \mathbb{N}$ and an open set $\omega\subseteq \mathbb{R}^n$, $n\in \mathbb{N}$, we define the spaces
	\begin{align*}
		\begin{aligned}
			H^m(\omega)&\coloneqq \big\{v\in L^2(\omega)&&\hspace*{-3.25mm}\mid \mathrm{D}^{\boldsymbol{\alpha}} v\in L^2(\omega)\textup{ for all }\boldsymbol{\alpha}\in (\mathbb{N}_0)^n\text{ with }\vert \boldsymbol{\alpha}\vert \leq m\big\}\,,\\
			H(\textup{div};\omega)&\coloneqq \big\{y\in (L^2(\omega))^n&&\hspace*{-3.25mm}\mid \textup{div}\,y\in L^2(\omega)\big\}\,,
		\end{aligned}
	\end{align*}
	where $\mathrm{D}^{\boldsymbol{\alpha}}\coloneqq \frac{\partial^{\vert \boldsymbol{\alpha}\vert }}{\partial x_1^{\alpha_1}\cdot\ldots\cdot \partial x_n^{\alpha_n}}$ and $\vert \boldsymbol{\alpha}\vert \coloneqq\sum_{i=1}^n{\alpha_i}$ 
	for each multi-index ${\boldsymbol{\alpha}\coloneqq (\alpha_1,\ldots,\alpha_n)\in (\mathbb{N}_0)^n}$,~and the \emph{Sobolev  norm} $\|\cdot\|_{m,\omega}\coloneqq \|\cdot\|_{\omega}+\vert \cdot\vert_{m,\omega}$,~where~$\|\cdot\|_{\omega}\coloneqq ((\cdot,\cdot)_\omega)^{\smash{\frac{1}{2}}}$ and
  \begin{align*} 
		\vert\cdot\vert _{m,\omega}\coloneqq\Bigg(\sum_{\boldsymbol{\alpha}\in (\mathbb{N}_0)^n\,:\,0<\vert\boldsymbol{\alpha}\vert\leq m }{\|\mathrm{D}^{\boldsymbol{\alpha}}(\cdot)\|_{\omega}^2}\Bigg)^{\smash{\frac{1}{2}}}\,,\\[-7mm]
	\end{align*} 
	turns $H^m(\omega)$ into a Hilbert space.\enlargethispage{5mm}
	 
	For $s \in (0,\infty)\setminus \mathbb{N}$ and an open set $\omega\subseteq \mathbb{R}^n$, $n\in \mathbb{N}$, the  \emph{Sobolev--Slobodeckij semi-norm}, for every $v \in  H^m(\omega)$,~is~defined~by
	\begin{align*}
		\vert v\vert_{s,\omega}\coloneqq \Bigg(\sum_{\vert \boldsymbol{\alpha}\vert= m}{\int_{\omega}{\int_{\omega}{\frac{\vert(\mathrm{D}^{\boldsymbol{\alpha}} v)(x)-(\mathrm{D}^{\boldsymbol{\alpha}} v)(y)\vert^2}{\vert x-y\vert^{2\theta+d}}\,\mathrm{d}x}\,\mathrm{d}y}}\Bigg)^{\smash{\frac{1}{2}}}\,,\\[-6mm]
	\end{align*}
	where $m\in \mathbb{N}_0$ and $\theta\in (0,1)$ are such that $s=m+\theta$.	
	Then, for $s \in (0,\infty)\setminus \mathbb{N}$ and an open set $\omega\subseteq \mathbb{R}^n$, $n\in \mathbb{N}$,  the \emph{Sobolev--Slobodeckij space} is defined by\vspace*{-0.5mm}
	\begin{align*}
		H^s(\omega)\coloneqq \big\{v\in H^m(\omega)\mid \vert v\vert _{s,\omega}<\infty\big\}\,,\\[-6mm]
	\end{align*} 
	where $m\hspace*{-0.1em}\in\hspace*{-0.1em} \mathbb{N}_0$ and $\theta\hspace*{-0.1em}\in\hspace*{-0.1em} (0,1)$ are such that $s\hspace*{-0.1em}=\hspace*{-0.1em}m+\theta$ and the \emph{\mbox{Sobolev--Slobodeckij}~norm}\vspace*{-0.5mm}
	\begin{align*}
	\|\cdot\|_{s,\omega}\coloneqq \|\cdot\|_{m,\omega}+\vert \cdot\vert _{s,\omega}\\[-6mm]
	\end{align*}
	%where $m\in \mathbb{N}_0$ and $\theta\in (0,1)$ are such that $s=m+\theta$, 
	turns $H^s(\omega)$ into a Hilbert space.
	
	Denote by $\textup{tr}(\cdot)\colon \hspace*{-0.1em}H^1(\Omega)\hspace*{-0.1em}\to\hspace*{-0.1em} H^{\smash{\frac{1}{2}}}(\partial\Omega)$ the trace operator  and by ${\textup{tr}((\cdot)\cdot n)\colon H(\textup{div};\Omega)\hspace*{-0.1em}\to\hspace*{-0.1em} H^{-\smash{\frac{1}{2}}}(\partial\Omega)}$ the normal trace operator, where $n\colon \partial\Omega\to \mathbb{S}^{d-1}$ denotes the outward unit normal vector~field~to~$\partial\Omega$. % and we employ the notation $H^{-\smash{\frac{1}{2}}}(\partial\Omega)\hspace*{-0.1em}\coloneqq \hspace*{-0.1em}(H^{\smash{\frac{1}{2}}}(\partial\Omega))^*$. 
	Then,  for every $v\hspace*{-0.1em}\in\hspace*{-0.1em} H^1(\Omega)$~and~${y\hspace*{-0.1em}\in\hspace*{-0.1em} H(\textup{div};\Omega)}$, there holds the integration-by-parts formula (\textit{cf}.\ \cite[Sec.\ 4.3, (4.12)]{EG21I})\vspace*{-0.5mm}
	\begin{align}\label{eq:pi_cont}
		(\nabla v,y)_{\Omega}+(v,\textup{div}\,y )_{\Omega}=\langle \textup{tr}(y)\cdot n,\textup{tr}(v)\rangle_{\partial\Omega}\,,\\[-6mm]\notag
	\end{align}
	where, for every $\widehat{y}\in H^{-\smash{\frac{1}{2}}}(\gamma)$, $\widehat{v}\in H^{\smash{\frac{1}{2}}}(\gamma)$, and $\gamma\in \{\Gamma_D,\Gamma_N,\Gamma_C,\partial\Omega\}$, we  abbreviate\vspace*{-0.5mm}
	\begin{align}\label{eq:abbreviation}
		\langle \widehat{y},\textup{tr}(\widehat{v})\rangle_{\gamma}\coloneqq \langle \widehat{y},\textup{tr}(\widehat{v})\rangle_{H^{\smash{\frac{1}{2}}}(\gamma)}\,.\\[-6mm]\notag
	\end{align}
	More precisely, in \eqref{eq:abbreviation}, for every subset $\gamma\subseteq \partial \Omega$ and $s >0$, the Hilbert space $H^s(\gamma)$~is~defined as the range of the restricted trace operator $\textup{tr}(\cdot)|_{\gamma}$ defined on $H^{s+\smash{\frac{1}{2}}}(\Omega)$
	 endowed with~the~image~norm, for every $w\in H^s(\gamma)$, defined by\vspace*{-0.5mm}
	\begin{align*}
		\|w\|_{s,\gamma}\coloneqq \inf_{v\in H^{s+\smash{\frac{1}{2}}}(\Omega)\;:\; \textup{tr}(v)|_{\gamma}=w}{\|v\|_{s +\smash{\frac{1}{2}},\Omega}}\,,\\[-6mm]
	\end{align*} 
	and $H^{-s}(\gamma)\coloneqq (H^s(\gamma))^*$ is defined as the corresponding topological dual space.
	
	Eventually, we employ the notation
	\begin{align*}
		\begin{aligned}
		 	H^1_D(\Omega)&\coloneqq  \big\{v\in 	H^1(\Omega) \mid \textup{tr}(v)=0\textup{ a.e.\ on }\Gamma_D\big\}\,,\\
			H^2_N(\textup{div};\Omega)&\coloneqq  \Bigg\{y\in H(\textup{div};\Omega)\;\bigg|\; \begin{aligned}&\langle\textup{tr}(y)\cdot n,\textup{tr}(v)\rangle_{\partial\Omega} =0\text{ for all  }v\in H^1_D(\Omega)\\&\text{with }\textup{tr}(v)=0\text{ a.e.\ on }\Gamma_C
			\end{aligned}\Bigg\}\,.
		\end{aligned}
	\end{align*}
	
	In what follows, we omit writing both $\textup{tr}(\cdot)$ and $\textup{tr}((\cdot)\cdot n)$ in this context.

	\subsection{Triangulations and standard finite element spaces}
	
	\hspace{5mm}Throughout the article, we denote by $\{\mathcal{T}_h\}_{h>0}$ a family of uniformly shape regular~triangula-tions of $\Omega\hspace*{-0.1em}\subseteq\hspace*{-0.1em} \mathbb{R}^d$, $d\hspace*{-0.1em}\in\hspace*{-0.1em}\mathbb{N}$, (\emph{cf}.\  \cite{EG21I}). 
	 Here, $h\hspace*{-0.1em}>\hspace*{-0.1em}0$ refers to the \emph{averaged mesh-size},~\textit{i.e.},~${h 
	 	\hspace*{-0.1em}=\hspace*{-0.1em} (\frac{\vert \Omega\vert}{\textup{card}(\mathcal{N}_h)})^{\frac{1}{d}}}
	 $, where $\mathcal{N}_h$  contains the vertices of the triangulation $\mathcal{T}_h$. We define the following sets of sides:
	 \begin{align*}
	 	\mathcal{S}_h&\coloneqq \mathcal{S}_h^{i}\cup \mathcal{S}_h^{\partial\Omega}\,,\\[-0.5mm]
	 	\mathcal{S}_h^{i}&\coloneqq  \big\{T\cap T'\mid T,T'\in\mathcal{T}_h\,,\text{dim}_{\mathscr{H}}(T\cap T')=d-1\big\}\,,\\[-0.5mm]
	 	\mathcal{S}_h^{\partial\Omega}&\coloneqq\big\{T\cap \partial\Omega\mid T\in \mathcal{T}_h\,,\text{dim}_{\mathscr{H}}(T\cap \partial\Omega)=d-1\big\}\,,\\[-0.5mm]
	 	\mathcal{S}_h^\gamma&\coloneqq\big\{S\in \mathcal{S}_h^{\partial\Omega}\mid \textup{int}(S)\subseteq \gamma\big\}\text{ for } \gamma\in \big\{\Gamma_D,\Gamma_N,\Gamma_C\big\}\,,
	 \end{align*}
	 where the Hausdorff dimension is defined by $\text{dim}_{\mathscr{H}}(M)\hspace{-0.15em}\coloneqq\hspace{-0.15em}\inf\{d'\hspace{-0.15em}\geq\hspace{-0.15em} 0\mid \mathscr{H}^{d'}(M)\hspace{-0.15em}=\hspace{-0.15em}0\}$~for~all~${M\hspace{-0.15em}\subseteq \hspace{-0.15em} \mathbb{R}^d}$.
	 It is also assumed that the triangulations $\{\mathcal{T}_h\}_{h>0}$ and  boundary parts $\Gamma_D$, $\Gamma_C$,~and~$\Gamma_N$~are~chosen such that  $\mathcal{S}_h^{\partial\Omega}=\mathcal{S}_h^{\Gamma_D}\dot{\cup}\mathcal{S}_h^{\Gamma_C}\dot{\cup} \mathcal{S}_h^{\Gamma_N}$, \textit{e.g.}, in the case $d=2$,  $\overline{\Gamma}_D$, $\overline{\Gamma}_C$, and $\overline{\Gamma}_N$~touch~only~in~vertices.

	For $k\in \mathbb{N}_0$ and $T\in \mathcal{T}_h$, let $\mathbb{P}^k(T)$ denote the set of polynomials of maximal~degree~$k$~on~$T$. Then, for $k\in \mathbb{N}_0$, the set of  element-wise polynomial functions  is defined by
	\begin{align*}
			\smash{\mathcal{L}^k(\mathcal{T}_h)\coloneqq \big\{v_h\in L^\infty(\Omega)\mid v_h|_T\in\mathbb{P}^k(T)\text{ for all }T\in \mathcal{T}_h\big\}\,.}
	\end{align*} 
	For \hspace*{-0.1mm}$\ell\hspace*{-0.175em}\in\hspace*{-0.175em} \mathbb{N}$, \hspace*{-0.1mm}the \hspace*{-0.1mm}(local) \hspace*{-0.1mm}$L^2$-projection \hspace*{-0.1mm}$\Pi_h\colon \hspace*{-0.175em} (L^1(\Omega))^{\ell}\hspace*{-0.175em}\to\hspace*{-0.175em} (\mathcal{L}^0(\mathcal{T}_h))^{\ell}\hspace*{-0.175em}$ \hspace*{-0.1mm}onto \hspace*{-0.1mm}element-wise~\hspace*{-0.1mm}constant~\hspace*{-0.1mm}\mbox{functions} or vector fields, respectively, for every 
$v\in  (L^1(\Omega))^{\ell} $ is defined by $\Pi_h v|_T\coloneqq\langle v\rangle_T$~for~all~${T\in \mathcal{T}_h}$.  
	
		For $m\in \mathbb{N}_0$ and $S\in \mathcal{S}_h$, let $\mathbb{P}^m(S)$ denote the set of polynomials of maximal degree~$m$~on~$S$. Then, for $m\hspace*{-0.175em}\in\hspace*{-0.175em} \mathbb{N}_0$ and $\mathcal{M}_h\hspace*{-0.175em}\in\hspace*{-0.175em} \{\mathcal{S}_h,\mathcal{S}^{i}_h,\mathcal{S}^{\partial\Omega}_h,\mathcal{S}^{\Gamma_D}_h,\mathcal{S}^{\Gamma_C}_h,\mathcal{S}^{\Gamma_N}_h\}$, the set of side-wise~polynomial functions  is defined by
	\begin{align*}
	\smash{\mathcal{L}^m(\mathcal{M}_h)\coloneqq  \big\{v_h\in L^\infty(\cup\mathcal{M}_h)\mid v_h|_S\in\mathbb{P}^m(S)\text{ for all }S\in \mathcal{M}_h\big\}\,.}
	\end{align*} 
	For \hspace*{-0.1mm}$\ell\hspace*{-0.175em}\in\hspace*{-0.175em} \mathbb{N}$, \hspace*{-0.1mm}the \hspace*{-0.1mm}(local) \hspace*{-0.1mm}$L^2$-projection \hspace*{-0.1mm}$\pi_h\colon\hspace*{-0.175em} (L^1(\cup\mathcal{S}_h))^{\ell}\hspace*{-0.175em}\to\hspace*{-0.175em} (\mathcal{L}^0(\mathcal{S}_h))^{\ell}$ \hspace*{-0.1mm}onto \hspace*{-0.1mm}side-wise \hspace*{-0.1mm}constant~\hspace*{-0.1mm}\mbox{functions} or vector  fields, respectively,  for every 
	$v_h\in  (L^1(\cup\mathcal{S}_h))^{\ell}$ is defined by $	\pi_h v_h|_S\coloneqq \langle v_h\rangle_S$~for~all~${S\in \mathcal{S}_h}$. 
	
	For  every $v_h\in \mathcal{L}^m(\mathcal{T}_h)$, $m\in \mathbb{N}_0$, and $S\in\mathcal{S}_h$, the \emph{jump across} $S$ is defined by 
		\begin{align*}
			\jump{v_h}_S\coloneqq\begin{cases}
				v_h|_{T_+}-v_h|_{T_-}&\text{ if }S\in \mathcal{S}_h^{i}\,,\text{ where }T_+, T_-\in \mathcal{T}_h\text{ satisfy }\partial T_+\cap\partial T_-=S\,,\\
				v_h|_T&\text{ if }S\in\mathcal{S}_h^{\partial\Omega}\,,\text{ where }T\in \mathcal{T}_h\text{ satisfies }S\subseteq \partial T\,.
			\end{cases}
	\end{align*}
	
	For  every $y_h\in (\mathcal{L}^m(\mathcal{T}_h))^d$, $m\in \mathbb{N}_0$, and $S\in\mathcal{S}_h$, the \emph{normal jump across} $S$ is defined by 
		\begin{align*}
			\jump{y_h\cdot n}_S\coloneqq\begin{cases}
				y_h|_{T_+}\!\cdot n_{T_+}+y_h|_{T_-}\!\cdot n_{T_-}&\text{ if }S\in \mathcal{S}_h^{i}\,,\text{ where }T_+, T_-\in \mathcal{T}_h\text{ satisfy }\partial T_+\cap\partial T_-=S\,,\\
				y_h|_T\cdot n&\text{ if }S\in\mathcal{S}_h^{\partial\Omega}\,,\text{ where }T\in \mathcal{T}_h\text{ satisfies }S\subseteq \partial T\,,
			\end{cases}
		\end{align*}
		where, for every $T\in \mathcal{T}_h$, we denote by $\smash{n_T\colon\partial T\to \mathbb{S}^{d-1}}$ the outward unit normal vector field~to~$ T$.

	\subsubsection{Crouzeix--Raviart element}
	
	\qquad The \emph{Crouzeix--Raviart finite element space} (\emph{cf}.\ \cite{CR73}) is defined as the space of element-wise affine functions that are continuous in the barycenters of interior sides, \textit{i.e.},
	\begin{align*}\mathcal{S}^{1,cr}(\mathcal{T}_h)\coloneqq \big\{v_h\in \mathcal{L}^1(\mathcal{T}_h)\mid \pi_h\jump{v_h}=0\text{ a.e.\ on }\cup \mathcal{S}_h^{i}\big\}\,.
	\end{align*}
	The Crouzeix--Raviart finite element space with homogeneous Dirichlet boundary condition~on~$\Gamma_D$ is defined by
	\begin{align*}
		\smash{\mathcal{S}^{1,cr}_D(\mathcal{T}_h)}\coloneqq \big\{v_h\in\smash{\mathcal{S}^{1,cr}(\mathcal{T}_h)}\mid  \pi_h v_h=0\text{ a.e.\ on }\cup \mathcal{S}_h^{\Gamma_D}\big\}\,.
	\end{align*}
	A basis of $\smash{\mathcal{S}^{1,cr}(\mathcal{T}_h)}$ is given via  
	$\varphi_S\hspace*{-0.1em}\in\hspace*{-0.1em} \smash{\mathcal{S}^{1,cr}(\mathcal{T}_h)}$, $S\hspace*{-0.1em}\in\hspace*{-0.1em} \mathcal{S}_h$, satisfying 
	 $\varphi_S(x_{S'})\hspace*{-0.1em}=\hspace*{-0.1em}\delta_{S,S'}$~for~all~${S,S'\hspace*{-0.1em}\in\hspace*{-0.1em} \mathcal{S}_h}$.
	The \emph{(Fortin) quasi-interpolation operator} $\smash{\Pi_h^{cr}\colon\hspace*{-0.15em} H^1(\Omega)\hspace*{-0.15em}\to\hspace*{-0.15em} \smash{\mathcal{S}^{1,\emph{cr}}(\mathcal{T}_h)}}$ (\textit{cf}.\ \cite[Secs.\ 36.2.1,~36.2.2]{EG21II}), for every $v\in H^1(\Omega)$ defined by
	\begin{align}\label{def:CR-Interpolant}
		\Pi_h^{cr}v\coloneqq \sum_{S\in \mathcal{S}_h}{\langle v\rangle_S\,\varphi_S}\,,
	\end{align}
	preserves averages of gradients and of  moments (on sides), \textit{i.e.}, for every $v\in H^1(\Omega)$,~it~holds~that
	\begin{alignat}{2}
		\nabla_h\Pi_h^{cr}v&=\Pi_h\nabla v&&\quad\text{ a.e.\ in }\Omega\,,\label{eq:grad_preservation}\\
		\pi_h\Pi_h^{cr}v&=\pi_h  v&&\quad \text{ a.e.\ on } \cup\mathcal{S}_h\,.\label{eq:trace_preservation}
	\end{alignat}
	Here, $\nabla_h\colon \mathcal{L}^1(\mathcal{T}_h)\to (\mathcal{L}^0(\mathcal{T}_h))^d$,  defined by $(\nabla_hv_h)|_T\coloneqq \nabla(v_h|_T)$ for all $v_h\in \mathcal{L}^1(\mathcal{T}_h)$ and ${T\in \mathcal{T}_h}$, denotes the element-wise gradient operator. 
	
For every $s \in [0,1]$, there  exists a constant $c>0$ (\emph{cf}.\ \cite[Lem.\ 36.1]{EG21II}), independent of $h>0$,
%which does depend only on the chunkiness of the family of triangulations $\{\mathcal{T}_h\}_{h>0}$ and on $s \in [0,1]$,\enlargethispage{7.5mm} % in a deteriorating way as $s\to 0$, 
 such that for every 
	$v\in H^{1+s}(\Omega)$ and $T\in \mathcal{T}_h$,~it~holds~that
	\begin{align}\label{eq:CR-Interpolant-Rate}
		\|v-\Pi_h^{cr} v\|_T+h_T\, \|\nabla v-\nabla \Pi_h^{cr} v\|_T\leq c\,h_T^{1+s}\,\vert v\vert_{1+s,T}\,.
	\end{align}  
	
	\subsubsection{Raviart--Thomas element}
	
	\qquad The \emph{(lowest order) Raviart--Thomas finite element space} (\emph{cf}.\ \cite{RT77}) is defined as the space of element-wise affine vector fields that have continuous constant normal components~on~interior~sides, \textit{i.e.},\enlargethispage{5mm}
	\begin{align*}
		\mathcal{R}T^0(\mathcal{T}_h)\coloneqq \left\{y_h\in(\mathcal{L}^1(\mathcal{T}_h))^d\;\bigg|\; \begin{aligned}
			&y_h|_T\cdot n_T=\textup{const}\text{ on }\partial T\text{ for all }T\in \mathcal{T}_h\,,\\ 
			&	\jump{y_h\cdot n}_S=0\text{ on }S\text{ for all }S\in \mathcal{S}_h^{i}
		\end{aligned}\right\}\,.
	\end{align*}
	The Raviart--Thomas finite element space with homogeneous normal boundary condition on $\Gamma_N$ is defined by
	\begin{align*}
		\mathcal{R}T^{0}_N(\mathcal{T}_h)\coloneqq \big\{y_h\in	\mathcal{R}T^0(\mathcal{T}_h)\mid y_h\cdot n=0\text{ a.e.\ on }\Gamma_N\big\}\,.
	\end{align*} 
	A basis of $\mathcal{R}T^0(\mathcal{T}_h)$ is given via 
	 vector fields $\psi_S\in \mathcal{R}T^0(\mathcal{T}_h)$, $S\in \mathcal{S}_h$, satisfying   $\psi_S|_{S'}\cdot n_{S'}=\delta_{S,S'}$ on $S'$ for all $S'\in \mathcal{S}_h$, where $n_S\in \mathbb{S}^{d-1}$ for all $S\in \mathcal{S}_h$ is the fixed unit normal vector on $S$ pointing from $T_-$ to $T_+$ if $T_+\cap T_-=S\in \mathcal{S}_h$. 
	For every $s>\frac{1}{2}$,
	the \emph{(Fortin) quasi-interpolation~operator} $\Pi_h^{rt}\colon (H^s(\Omega))^d\to \mathcal{R}T^0(\mathcal{T}_h)$ (\textit{cf}.\ \cite[Sec.\ 16.1]{EG21I}), for every $y\in (H^s(\Omega))^d$ defined by
	\begin{align}\label{def:RT-Interpolant}
		\Pi_h^{rt} y\coloneqq \sum_{S\in \mathcal{S}_h}{\langle y\cdot n_S\rangle_S\,\psi_S}\,,
	\end{align}
	preserves averages of divergences and of normal traces, \textit{i.e.}, for every $y\in (H^s(\Omega))^d\cap H(\textup{div};\Omega)$, it holds that
	\begin{alignat}{2}
		\textup{div}\,\Pi_h^{rt}y&=\Pi_h\textup{div}\,y&&\quad \text{ a.e.\ in }\Omega\,,\label{eq:div_preservation}\\
		\Pi_h^{rt}y\cdot n&=\pi_hy\cdot n&&\quad \text{ a.e\ on }\cup\mathcal{S}_h\,.\label{eq:normal_trace_preservation}
	\end{alignat}

	For every $s\hspace*{-0.15em} \in \hspace*{-0.15em}(\frac{1}{2},1]$, there exists a constant $c\hspace*{-0.15em}>\hspace*{-0.15em}0$ (\emph{cf}.\ \cite[Thms.\ 16.4, 16.6]{EG21I}),~\mbox{independent}~of~${h\hspace*{-0.15em}>\hspace*{-0.15em}0}$, such that for every $y\in (H^s(\Omega))^d\cap H(\textup{div};\Omega)$ and $T\in \mathcal{T}_h$, it holds that
	\begin{align}
		\|y-\Pi_h^{rt} y\|_T\leq c\,h_T^{s}\,\vert y\vert_{s,T}\,.\label{eq:RT-Interpolant-Rate}
	\end{align} 
	
%	\subsubsection{Discrete integration-by-parts formula and discrete Helmholtz decomposition}
	
	For every $v_h\in \mathcal{S}^{1,cr}(\mathcal{T}_h)$ and ${y_h\in \mathcal{R}T^0(\mathcal{T}_h)}$, we have the \emph{discrete integration-by-parts
		formula}\enlargethispage{2.5mm}
	\begin{align}
		(\nabla_hv_h,\Pi_h y_h)_{\Omega}+(\Pi_h v_h,\,\textup{div}\,y_h)_{\Omega}=(\pi_h v_h,y_h\cdot n)_{\partial\Omega}\,.\label{eq:pi}
	\end{align}  
%	In addition, according to \cite[Sec.\ 2.4]{BW21}, there holds the \emph{discrete Helmholtz decomposition}
%	\begin{align}
%		(\mathcal{L}^0(\mathcal{T}_h))^d=\textup{ker}(\textup{div}|_{\smash{\mathcal{R}T^0_N(\mathcal{T}_h)}})\oplus \nabla_h(\smash{\mathcal{S}^{1,cr}_D(\mathcal{T}_h)})\,.\label{eq:decomposition}
%	\end{align}

	\newpage
	
	\section{Scalar Signorini problem}\label{sec:continuous_signorini}
	
	\qquad In this section, we discuss the (continuous) scalar Signorini problem.\medskip

	\hspace*{-2.5mm}$\bullet$ \emph{Primal problem.} Given $f\in L^2(\Omega)$, $g\in H^{-\smash{\frac{1}{2}}}(\Gamma_N)$, $u_D\in H^{\frac{1}{2}}(\Gamma_D)$, and
	$\chi\in H^1(\Omega)$~with~${\chi=u_D}$ a.e.\ on $\Gamma_D$,
	the \emph{scalar Signorini problem} is given via the minimization of $I\colon H^1(\Omega)\to \mathbb{R}\cup\{+\infty\}$, for every $v\in H^1(\Omega)$ defined by 
	\begin{align}
		\begin{aligned}
		I(v)&\coloneqq \tfrac{1}{2}\| \nabla v\|_{\Omega}^2-(f,v)_{\Omega}-\langle g,v\rangle_{\Gamma_N}+I_K(v)
		\\&\,=\tfrac{1}{2}\| \nabla v\|_{\Omega}^2-(f,v)_{\Omega}-\langle g,v\rangle_{\Gamma_N}+\smash{I_{\{u_D\}}^{\Gamma_D}}(v)+\smash{I_+^{\Gamma_C}}(v-\chi)\,,
			\end{aligned}\label{eq:primal}
	\end{align}
	where 
	\begin{align*}
		K\coloneqq \big\{v\in H^1(\Omega)\mid v = u_D\text{ a.e.\ on }\Gamma_D\,,\; v\ge \chi\text{ a.e.\ on }\Gamma_C\big\}\,,
	\end{align*}
	and $I_K\coloneqq \smash{I_{\{u_D\}}^{\Gamma_D}}+\smash{I_+^{\Gamma_C}}((\cdot)-\chi),\smash{I_{\{u_D\}}^{\Gamma_D}},\smash{I_+^{\Gamma_C}}\colon H^1(\Omega)\to\mathbb{R}\cup\{+\infty\}$, for every $\widehat{v}\in H^1(\Omega)$~are~\mbox{defined}~by 
	\begin{align*}
		\smash{I_{\{u_D\}}^{\Gamma_D}}(\widehat{v})&\coloneqq \begin{cases}
			0&\text{ if }\widehat{v}=u_D\text{ a.e.\ on }\Gamma_D\,,\\
			+\infty &\text{ else}\,,
		\end{cases}\\
			\smash{I_+^{\Gamma_C}}(\widehat{v})&\coloneqq \begin{cases}
			0&\text{ if }\widehat{v}\ge 0\text{ a.e.\ on }\Gamma_C\,,\\
			+\infty &\text{ else}\,.
		\end{cases}
	\end{align*}
	Throughout the article, we refer to the minimization of the functional \eqref{eq:primal} as the \emph{primal problem}.\\
	Since the functional \eqref{eq:primal} is proper, strictly convex, weakly coercive, and lower semi-continuous, the direct method in the calculus of variations yields the existence of a unique~minimizer~$u\in K$, 
	called \emph{primal solution}. In what follows, we reserve the notation $u\in K$ for the primal solution.\enlargethispage{7,5mm}

	\hspace*{-2.5mm}$\bullet$ \emph{Primal variational inequality.} The primal solution $u\in K$ equivalently is the  unique solution of the following variational inequality: % In fact, $u\in K$ is minimal for \eqref{eq:primal}~if~and~only~if 
	for every $v\in K$, it holds that
	\begin{align}
		(\nabla u,\nabla u-\nabla v)_{\Omega}\leq (f,u-v)_{\Omega}+\langle g,u-v\rangle_{\Gamma_N}\,.\label{eq:variational_ineq}
	\end{align}

	\hspace*{-2.5mm}$\bullet$ \emph{Dual problem.} A \emph{(Fenchel) dual problem}  to the scalar Signorini problem is given via the maximization of $D\colon H(\textup{div};\Omega)\to \mathbb{R}\cup\{-\infty\}$, for every $y\in H(\textup{div};\Omega)$ defined by
	\begin{align}
		\begin{aligned} 
		D(y)&\coloneqq -\tfrac{1}{2}\|y\|_{\Omega}^2+\langle y\cdot n,\chi\rangle_{\partial\Omega}-\langle g,\chi\rangle_{\Gamma_N}-I_{K^*}(y)
		\\&\,=-\tfrac{1}{2}\|y\|_{\Omega}^2+\langle y\cdot n,\chi\rangle_{\partial\Omega}-\langle g,\chi\rangle_{\Gamma_N}
		-\smash{I_{\{-f\}}^{\Omega}}(\textup{div}\,y)-\smash{I_{\{g\}}^{\Gamma_N}}(y\cdot n)-\smash{I_+^{\Gamma_C}}(y\cdot n)\,,
	\end{aligned}\label{eq:dual}
	\end{align}
	where
	\begin{align*}
		K^*\coloneqq \big\{y\in H(\textup{div};\Omega)\mid \smash{I_{\{-f\}}^{\Omega}}(\textup{div}\,y)=\smash{I_{\{g\}}^{\Gamma_N}}(y\cdot n)=\smash{I_+^{\Gamma_C}}(y\cdot n)=0\big\}\,,
		%\{y\in H(\textup{div};\Omega)\mid \textup{div}\,y=-f \text{ a.e.\ in }\Omega\,,\; y\cdot n = g\text{ a.e.\ on  }\Gamma_N\,,\;y\cdot n \ge 0\text{ a.e.\ on  }\Gamma_C \}\,,
	\end{align*}
	 $I_{K^*}\coloneqq \smash{I_{\{-f\}}^{\Omega}}(\textup{div}\,\cdot)+(\smash{I_{\{g\}}^{\Gamma_N}}+\smash{I_+^{\Gamma_C}})((\cdot)\cdot n)\colon H(\textup{div};\Omega)\to \mathbb{R}\cup\{+\infty\}$, 
	$\smash{I_{\{-f\}}^{\Omega}}\colon L^2(\Omega)\to\mathbb{R}\cup\{+\infty\}$, for every $\widehat{y}\in L^2(\Omega)$ is defined by
	\begin{align*}
			\smash{I_{\{-f\}}^{\Omega}}(\widehat{y})\coloneqq \begin{cases}
			0&\text{ if }\widehat{y}=-f\text{ a.e.\ in }\Omega\,,\\
			+\infty &\text{ else}\,,
		\end{cases}
	\end{align*}
	and 
	$\smash{I_{\{g\}}^{\Gamma_N}},\smash{I_+^{\Gamma_C}}\colon H^{-\smash{\frac{1}{2}}}(\partial\Omega)\to\mathbb{R}\cup\{+\infty\}$, for every $\widehat{y}\in H^{-\smash{\frac{1}{2}}}(\partial\Omega)$, are defined by 
	\begin{align*}
	\smash{I_{\{g\}}^{\Gamma_N}}(\widehat{y})&\coloneqq \begin{cases}
			0&\text{ if }\langle \widehat{y},v\rangle_{\partial\Omega}=\langle g, v\rangle_{\Gamma_N}\text{ for all }v\in H^1_D(\Omega)\text{ with }v=0\text{ a.e.\ on }\Gamma_C\,,\\
			+\infty &\text{ else}\,,
		\end{cases}\\
		\smash{I_+^{\Gamma_C}}(\widehat{y})&\coloneqq \begin{cases}
			0&\text{ if }\langle \widehat{y},v\rangle_{\Gamma_C}\ge 0\text{ for all }v\in H^1_D(\Omega)\\&\text{ with }v=0\text{ a.e.\ on }\Gamma_N\text{ and }v\ge 0\text{ a.e.\ on }\Gamma_C\,,\\
			+\infty &\text{ else}\,.
		\end{cases}
	\end{align*}
	
	The identification of the (Fenchel) dual problem (in the sense of \cite[Rem.\ 4.2, p.\ 60/61]{ET99}) to the minimization of \eqref{eq:primal} with the maximization of  \eqref{eq:dual} can be found in the proof of the following result that also establishes the validity of a strong~duality~relation and convex~optimality~relations.\newpage
	
	\begin{proposition}[strong duality and convex duality relations]\label{prop:duality} The following statements apply:\enlargethispage{5mm}
		\begin{itemize}[noitemsep,topsep=2pt,leftmargin=!,labelwidth=\widthof{(ii)}]
			\item[(i)]  A \hspace*{-0.1mm}(Fenchel) \hspace*{-0.1mm}dual \hspace*{-0.1mm}problem \hspace*{-0.1mm}to \hspace*{-0.1mm}the \hspace*{-0.1mm}scalar \hspace*{-0.1mm}Signorini \hspace*{-0.1mm}problem \hspace*{-0.1mm}is \hspace*{-0.1mm}given \hspace*{-0.1mm}via \hspace*{-0.1mm}the \hspace*{-0.1mm}\mbox{maximization}~\hspace*{-0.1mm}of~\hspace*{-0.1mm}\eqref{eq:dual}.  
			\item[(ii)]  There exists a unique maximizer $z\in H(\textup{div};\Omega)$ of \eqref{eq:dual} satisfying the \textup{admissibility~\mbox{conditions}} 
			\begin{align}
					\textup{div}\,z&=-f\quad\text{ a.e.\ in }\Omega\label{eq:admissibility.1}\,,\\ 
				\smash{I_{\{g\}}^{\Gamma_N}}(z\cdot n)&= 0\label{eq:admissibility.2}\,,\\ 
				\smash{I_+^{\Gamma_C}}(z\cdot n) &= 0 \label{eq:admissibility.3}\,.
			\end{align}
			In addition, there
			holds a \textup{strong duality relation}, \textit{i.e.}, it holds that 
			\begin{align}
				I(u) = D(z)\,.\label{eq:strong_duality}
			\end{align}
			\item[(iii)] There hold  \textup{convex optimality relations}, \textit{i.e.}, it holds that 
			\begin{alignat}{2}
				z&=\nabla u\quad\text{ a.e.\ in }\Omega\,,\label{eq:optimality.1}\\
			\langle z\cdot n,u-\chi \rangle_{\partial\Omega}&=\langle g,u-\chi\rangle_{\Gamma_N}\label{eq:optimality.2}\,.
			\end{alignat} 
		\end{itemize}
	\end{proposition}
	
	\begin{remark}
		\begin{itemize}[noitemsep,topsep=2pt,leftmargin=!,labelwidth=\widthof{(ii)}]
			\item[(i)] If $g\in L^1(\Gamma_N)$, then  \eqref{eq:admissibility.2} is equivalent to $z\cdot n=g$ a.e.\ on $\Gamma_N$;
			\item[(ii)] If $z\cdot n|_{\Gamma_C}\in L^1(\Gamma_C)$, then  \eqref{eq:admissibility.3} is equivalent to $z\cdot n\ge 0$ a.e.\ on $\Gamma_C$.
		\end{itemize}
	\end{remark}
	
	\begin{proof}[Proof (of Proposition \ref{prop:duality}).]
		\emph{ad (i).} 
		First, if we introduce the  proper, lower semi-continuous, and~\mbox{convex}~\mbox{functionals} $G\colon (L^2(\Omega))^d\to \mathbb{R}$ and $F\colon H^1(\Omega)\to \mathbb{R}\cup\{+\infty\}$, for every $y\in (L^2(\Omega))^d$ and $v\in H^1(\Omega)$,~defined~by
		\begin{align*}
			G(y)&\coloneqq \tfrac{1}{2}\|y\|_{\Omega}^2\,,\\
			F(v)&\coloneqq -(f,v)_{\Omega}-\langle g,v\rangle_{\Gamma_N}+\smash{I_{\{u_D\}}^{\Gamma_D}}(v)+\smash{I_+^{\Gamma_C}}(v-\chi)\,,
		\end{align*}
		then, for every $v\in H^1(\Omega)$, we have that 
		\begin{align*}
			I(v)= G(\nabla v)+F(v)\,.
		\end{align*} 
		Thus, \hspace{-0.1mm}in \hspace{-0.1mm}accordance \hspace{-0.1mm}with \hspace{-0.1mm}\cite[Rem.\ \hspace{-0.1mm}4.2, \hspace{-0.1mm}p.\ \hspace{-0.1mm}60/61]{ET99}, \hspace{-0.1mm}the \hspace{-0.1mm}(Fenchel) \hspace{-0.1mm}dual \hspace{-0.1mm}problem \hspace{-0.1mm}to \hspace{-0.1mm}the~\hspace{-0.1mm}minimization~\hspace{-0.1mm}of \eqref{eq:primal} is given via the maximization of $D\colon (L^2(\Omega))^d\to   \mathbb{R}\cup\{-\infty\}$, 
		for every $y\in (L^2(\Omega))^d$~defined~by 
		\begin{align}\label{prop:duality.1}
			D(y)\coloneqq -G^*(y)- F^*(-\nabla^*y)\,,
		\end{align}
		where \hspace*{-0.15mm}$\nabla^*\hspace*{-0.2em}\colon  \hspace*{-0.2em}(L^2(\Omega))^d\hspace*{-0.2em}\to\hspace*{-0.2em} (H^1(\Omega))^*$ \hspace*{-0.15mm}is \hspace*{-0.15mm}the \hspace*{-0.15mm}adjoint \hspace*{-0.15mm}operator \hspace*{-0.15mm}to \hspace*{-0.15mm}the \hspace*{-0.15mm}gradient \hspace*{-0.15mm}operator \hspace*{-0.15mm}${\nabla\hspace*{-0.1em} \colon \hspace*{-0.2em}H^1(\Omega)\hspace*{-0.2em}\to\hspace*{-0.2em} (L^2(\Omega))^d}$. Due to \cite[Prop.\ 4.2, p.\ 19]{ET99},  for every $y\in (L^2(\Omega))^d$, we have that 
		\begin{align}\label{prop:duality.2}
			G^*(y)=\tfrac{1}{2}\|y\|_{\Omega}^2\,.
		\end{align}
		Since $v+\chi\in H^1(\Omega)$ with $v+\chi=u_D$ a.e.\ on $\Gamma_D$ for all $v\in H^1_D(\Omega)$, for every $y\in (L^2(\Omega))^d$, using the integration-by-parts formula \eqref{eq:pi_cont}, we find that
		\begin{align}\label{prop:duality.3}
			\begin{aligned}
			F^*(-\nabla^*y)&=\sup_{v\in H^1(\Omega)}{\big\{-(y,\nabla v )_{\Omega}+(f,v)_{\Omega}+\langle g,v\rangle_{\Gamma_N}-\smash{I_{\{u_D\}}^{\Gamma_D}}(v)-\smash{I_+^{\Gamma_C}}(v-\chi)\big\}}
		%	\\&=\sup_{v\in H^1_D(\Omega)}{\big\{-(y,\nabla (v+\chi) )_{\Omega}+(f,v+\chi)_{\Omega}+\langle g,v+\chi\rangle_{\Gamma_N}-\smash{I_+^{\Gamma_C}}(v)\big\}}
			\\&=\sup_{v\in H^1_D(\Omega)}{\big\{-(y,\nabla v )_{\Omega}+(f,v)_{\Omega}+\langle g,v\rangle_{\Gamma_N}-\smash{I_+^{\Gamma_C}}(v)\big\}}
			\\&\quad	-(y,\nabla \chi )_{\Omega}+(f,\chi)_{\Omega}+\langle g,\chi\rangle_{\Gamma_N}
			\\&=\begin{cases}
				\left.\begin{aligned}\smash{I_{\{-f\}}^{\Omega}}(\textup{div}\,y)+\smash{I_{\{g\}}^{\Gamma_N}}(y\cdot n)+\smash{I_+^{\Gamma_C}}(y\cdot n)
					\\[1mm]-(y,\nabla \chi )_{\Omega}+(f,\chi)_{\Omega}+\langle g,\chi\rangle_{\Gamma_N}\end{aligned}\right\}&\text{ if }y\in H(\textup{div};\Omega)\,,\\
				+\infty &\text{ else}\,.
			\end{cases}
			\\&=\begin{cases}
				\left.\begin{aligned}	\smash{I_{\{-f\}}^{\Omega}}(\textup{div}\,y)+\smash{I_{\{g\}}^{\Gamma_N}}(y\cdot n)+\smash{I_+^{\Gamma_C}}(y\cdot n)\\[1mm]
				-\langle y\cdot n,\chi\rangle_{\partial\Omega}+\langle g,\chi\rangle_{\Gamma_N}
			\end{aligned}\right\}&\text{ if }y\in H(\textup{div};\Omega)\,,\\
				+\infty &\text{ else}\,.
			\end{cases} 
			\end{aligned}
		\end{align}
		Using \eqref{prop:duality.2} and \eqref{prop:duality.3} in \eqref{prop:duality.1}, for every $y\in H(\textup{div};\Omega)$, we arrive at the representation \eqref{eq:dual}.
		Eventually, since $D= -\infty$ in $(L^2(\Omega))^d\setminus H(\textup{div};\Omega)$, it is enough to restrict \eqref{prop:duality.1} to $H(\textup{div};\Omega)$.\pagebreak
		
		\emph{ad (ii).} Since both $G\colon (L^2(\Omega))^d\to \mathbb{R}$ and $F\colon H^1(\Omega)\to \mathbb{R}\cup\{+\infty\}$ are proper, convex, and lower semi-continuous and since 
		$G\colon  (L^2(\Omega))^d\to \mathbb{R}$  is continuous at
		$\chi\in \textup{dom}(F)\cap \textup{dom}(G\circ \nabla)$,~\textit{i.e.},
		\begin{align*}
			G(y)\to G(\nabla \chi)\quad \big(y\to \nabla \chi\quad\text{ in }(L^2(\Omega))^d\big)\,,
		\end{align*}
		by \hspace*{-0.1mm}the \hspace*{-0.1mm}celebrated \hspace*{-0.1mm}Fenchel \hspace*{-0.1mm}duality \hspace*{-0.1mm}theorem \hspace*{-0.1mm}(\emph{cf}.\ \hspace*{-0.1mm}\cite[\hspace*{-0.1mm}Rem.\ \hspace*{-0.1mm}4.2, \hspace*{-0.1mm}(4.21), \hspace*{-0.1mm}p.\ \hspace*{-0.1mm}61]{ET99}),  \hspace*{-0.1mm}there~\hspace*{-0.1mm}exists~\hspace*{-0.1mm}a~\hspace*{-0.1mm}\mbox{maximizer} $z\in  (L^2(\Omega))^d$ of  \eqref{prop:duality.1} and a strong duality relation applies, \textit{i.e.},
		\begin{align}\label{strong}
			I(u)=D(z)\,.
		\end{align}
		Since $D= -\infty$ in $(L^2(\Omega))^d\setminus H(\textup{div};\Omega)$, from \eqref{strong}, we infer that $z\in  H(\textup{div};\Omega)$. Moreover, since \eqref{eq:dual} is strictly concave,  the maximizer $z\in H(\textup{div};\Omega)$ is uniquely determined.
		
		\emph{ad (iii).} \hspace*{-0.1mm}By \hspace*{-0.1mm}the \hspace*{-0.1mm}standard \hspace*{-0.1mm}(Fenchel) \hspace*{-0.1mm}convex \hspace*{-0.1mm}duality \hspace*{-0.1mm}theory \hspace*{-0.1mm}(\emph{cf}.\ \hspace*{-0.1mm}\cite[Rem.\ \hspace*{-0.1mm}4.2,~\hspace*{-0.1mm}(4.24),~\hspace*{-0.1mm}(4.25),\hspace*{-0.1mm}~p.~\hspace*{-0.1mm}61]{ET99}), there hold the convex optimality relations
		\begin{align}
			-\nabla^*z&\in \partial F(u)\,,\label{prop:duality.4}\\
			z&\in \partial G(\nabla u)\,.\label{prop:duality.5}
		\end{align}
		While the inclusion \eqref{prop:duality.5} is equivalent to the convex optimality relation \eqref{eq:optimality.1}, the inclusion~\eqref{prop:duality.4}, by the standard equality condition in the Fenchel--Young inequality (\textit{cf}.\ \cite[Prop.\ 5.1, p.\ 21]{ET99}) and the admissibility condition \eqref{eq:admissibility.1}, is equivalent~to 
%		that for every $v\in H^1(\Omega)$ with $v=u_D$ a.e.\ on $\Gamma_D$ and $v\ge \chi$ a.e.\ on $\Gamma_C$, it holds that
%		\begin{align}\label{prop:duality.4}
%			(z,\nabla u-\nabla v)_\Omega-(f,u-v)_\Omega-\langle g,u-v\rangle_{\Gamma_N}\ge 0\,.
%		\end{align}
%		Choosing $v=u\pm w$ in \eqref{prop:duality.4} for every $w\in H^1_{\Gamma_D\cup \Gamma_C}(\Omega)$,  we find that 
%		\begin{align*}
%			(z,\nabla w)_\Omega-(f,w)_\Omega-\langle g,w\rangle_{\Gamma_N}=0\,,
%		\end{align*}
%		which implies that \eqref{eq:optimality.2} and \eqref{eq:optimality.3} apply. 
%		Choosing $v=u\pm\in w$ in \eqref{prop:duality.4} for every $w\in H^1_{\Gamma_D\cup \Gamma_N}(\Omega)$ with $w\ge 0$ a.e.\ on $\Gamma_N$,  we find that 
%		\begin{align*}
%			0&\leq 	(z,\nabla w)_\Omega-(f,w)_\Omega
%			\\&=\langle z\cdot n, w\rangle_{\Gamma_N}\,,
%		\end{align*}
%		which implies that \eqref{eq:optimality.4} applies. 
		\begin{align*}
			-(z,  \nabla u)_{\Omega}&=	(-\nabla^*z,  u)_{\Omega}\\
			&=F^*(-\nabla^*z)+F(u)
			\\&=-\langle z\cdot n,\chi\rangle_{\partial\Omega}+\langle g,\chi\rangle_{\Gamma_N}-(f,u)_{\Omega}-\langle g,u\rangle_{\Gamma_N}
				\\&=-\langle z\cdot n,\chi\rangle_{\partial\Omega}-\langle g,u-\chi\rangle_{\Gamma_N}+(\textup{div}\,z,u)_{\Omega}
				\\&=\langle z\cdot n,u-\chi\rangle_{\partial\Omega}-\langle g,u-\chi\rangle_{\Gamma_N}+(\textup{div}\,z,u)_{\Omega}-\langle z\cdot n,u\rangle_{\partial\Omega}
				%\\&=\langle z\cdot n,u-\chi\rangle_{\partial\Omega}-\langle g,u-\chi\rangle_{\Gamma_N} 	-(z,  \nabla u)_{\Omega}
				\,,
		\end{align*}
		which, by the integration-by-parts formula \eqref{eq:pi_cont}, is equivalent to the claimed  convex optimality relation \eqref{eq:optimality.2}.
	\end{proof}
	
	\hspace*{-2.5mm}$\bullet$ \emph{Dual variational inequality.} A dual solution $z\in K^*$  equivalently is the unique solution of the following  variational inequality: for every $y\in K^*$, it holds that
	\begin{align}\label{dual_variational_ineq}
		(z,z-y)_{\Omega}\le \langle z\cdot n-y\cdot n, \chi\rangle_{\partial\Omega}\,.
	\end{align}

	\hspace*{-2.5mm}$\bullet$ \emph{Augmented problem.} There exists a Lagrange multiplier $\Lambda^*\in (H^1_D(\Omega))^*$ such that for every $v\in H^1_D(\Omega)$, there holds the \emph{augmented problem}
	\begin{align}
			(\nabla u,\nabla v)_{\Omega}+\langle \Lambda^*,v\rangle_{\smash{H^1_D(\Omega)}}=(f,v)_{\Omega}+\langle g,v\rangle_{\Gamma_N}\,.\label{eq:augmented_problem}
	\end{align}
	With the convex optimality relations \eqref{eq:optimality.1},\eqref{eq:optimality.2} and the integration-by-parts formula \eqref{eq:pi_cont}, for every $v\in H^1_D(\Omega)$, we find that
	\begin{align*}
		\langle \Lambda^*,v\rangle_{\smash{H^1_D(\Omega)}}=\langle z\cdot n, v\rangle_{\partial\Omega}-\langle g, v\rangle_{\Gamma_N}\,.
	\end{align*}
	In particular, the convex optimality relation \eqref{eq:optimality.2} then also reads as the \emph{complementarity~condition}\enlargethispage{7.5mm}
	\begin{align*}
			\langle \Lambda^*,u-\chi\rangle_{\smash{H^1_D(\Omega)}}=0\,.
	\end{align*}
	%The following remark collects known regularity results for the scalar Signorini problem~in~2D.\enlargethispage{12.5mm}
	
	\begin{remark}[regularity \hspace*{-0.1mm}in \hspace*{-0.1mm}2D]\label{rem:regularity} In \hspace*{-0.1mm}the \hspace*{-0.1mm}two-dimensional \hspace*{-0.1mm}case, \hspace*{-0.1mm}the \hspace*{-0.1mm}following \hspace*{-0.1mm}regularity \hspace*{-0.1mm}results~\hspace*{-0.1mm}apply:
		\begin{itemize}[noitemsep,topsep=2pt,leftmargin=!,labelwidth=\widthof{(iii)}]
			\item[(i)] If $\Omega\subseteq \mathbb{R}^2$ is a bounded domain with smooth boundary, $\Gamma_C=\partial\Omega$, and  $\chi\in H^{\frac{3}{2}}(\partial\Omega)$, then $u\in H^2(\Omega)$ (\textit{cf}.\ \cite[Lem.\ 2.2]{Spann93}).
			\item[(ii)]  If $\Omega\subseteq \mathbb{R}^2$ is a polygonal, convex, and bounded domain, $\Gamma_C=\partial\Omega$, and  $\chi\in H^{\frac{3}{2}}(\partial\Omega)$, then $u\in H^2(\Omega)$ (\textit{cf}.\ \cite[Thm.\ 4.1]{GrisvardLooss76}).
			\item[(iii)] If $\Omega\hspace*{-0.175em}\subseteq \hspace*{-0.175em}\mathbb{R}^2$ is a polygonal bounded  domain, $\Gamma_C\hspace*{-0.175em}\neq\hspace*{-0.175em}\partial\Omega$, and  $\chi\hspace*{-0.175em}\in\hspace*{-0.175em} H^{\frac{3}{2}}(\partial\Omega)$, then~${u\hspace*{-0.175em}\in\hspace*{-0.175em} H^2(U)\hspace*{-0.15em}\cap\hspace*{-0.15em} C^{1,\lambda}(U)}$ for $\lambda\in (1,\frac{1}{2})$ (\textit{cf}.\ \cite{MoussaouiKhodja92} or \cite[Thm.\ 2.1]{ApelNicaise20}), where $U\subseteq \mathbb{R}^2$ is a neighborhood of  the critical 	points, \textit{i.e.}, the points where the boundary condition changes and that are corners~of~the~domain. In addition, in \cite[Thm.\ 3.1]{ApelNicaise20}, a description of possible singular behavior close to the critical 
			points can be found.
		\end{itemize}
	\end{remark}
	\newpage

		\section{\emph{A posteriori} error analysis}\label{sec:aposteriori} 
	\qquad In this section, resorting to convex duality arguments, we derive an \emph{a posteriori} error identity for arbitrary conforming approximations of  the primal problem \eqref{eq:primal} and the~dual~problem~\eqref{eq:dual} at~the~same~time. 
	To this end, we introduce the 
	\emph{primal-dual gap estimator} ${\eta^2_{\textup{gap}}\colon K\times K^*\to \mathbb{R}}$, for every $v\in K$ and $y\in K^*$ defined by 
	\begin{align}\label{eq:primal-dual.1}
		\begin{aligned}
			\eta^2(v,y)&\coloneqq I(v)-D(y)\,. 
		\end{aligned}
	\end{align}
	
	The primal-dual gap estimator \eqref{eq:primal-dual.1} can be decomposed into two contributions that precisely measure the violation of the convex optimality relations \eqref{eq:optimality.1},\eqref{eq:optimality.2}, respectively.
	
	\begin{lemma}\label{lem:primal_dual_gap_estimator}
		For every $v\in K$ and $y\in K^*$, we have that
		\begin{align*} 
			\eta_{\textup{gap}}^2(v,y)&\coloneqq \eta_A^2(v,y)+\eta_B^2(v,y)\,,\\
			\eta_{\textup{gap},A}^2(v,y)&\coloneqq \tfrac{1}{2}\|\nabla v-y\|_{\Omega}^2\,,\\
			\eta_{\textup{gap},B}^2(v,y)&\coloneqq \langle y\cdot n, v-\chi\rangle_{\partial\Omega}-\langle g, v-\chi\rangle_{\Gamma_N}\,. 
		\end{align*}
	\end{lemma}
	
	\begin{remark}[interpretation of the components of the primal-dual gap estimator]\hphantom{                   }
		\begin{itemize}[noitemsep,topsep=2pt,leftmargin=!,labelwidth=\widthof{(iii)},font=\itshape]
			\item[(i)] The estimator $\eta_{\textup{gap},A}^2$ measures the violation of the convex optimality relation \eqref{eq:optimality.1};
			\item[(ii)] The estimator $\eta_{\textup{gap},B}^2$ measures the violation of the convex optimality relation \eqref{eq:optimality.2}.
		\end{itemize}
	\end{remark}
	
	\begin{proof}[Proof (of Lemma \ref{lem:primal_dual_gap_estimator}).]\let\qed\relax
		Using the admissibility conditions \eqref{eq:admissibility.1}--\eqref{eq:admissibility.3}, the~\mbox{integration-by-parts} formula \eqref{eq:pi_cont}, and the binomial formula, 
		for every $v\in K$ and $y\in K^*$, we find that
		\begin{align*}
			I(v)-D(y)&= \tfrac{1}{2}\| \nabla v\|_{\Omega}^2-(f,v)_{\Omega}-\langle g,v\rangle_{\Gamma_N}+\tfrac{1}{2}\| y\|_{\Omega}^2-\langle y\cdot n, \chi\rangle_{\partial\Omega}+\langle g, \chi\rangle_{\Gamma_N}\\&
			= \tfrac{1}{2}\| \nabla v\|_{\Omega}^2+(\textup{div}\,y,v)_{\Omega}+\tfrac{1}{2}\| y\|_{\Omega}^2-\langle y\cdot n, \chi\rangle_{\partial\Omega}-\langle g,v- \chi\rangle_{\Gamma_N}
			\\&
			= \tfrac{1}{2}\| \nabla v\|_{\Omega}^2-(y,\nabla v)_{\Omega}+\tfrac{1}{2}\| y\|_{\Omega}^2+\langle y\cdot n, v-\chi\rangle_{\partial\Omega}-\langle g, v-\chi\rangle_{\Gamma_N}
			\\&
			= \tfrac{1}{2}\| \nabla v-y\|_{\Omega}^2+\langle y\cdot n, v-\chi\rangle_{\partial\Omega}-\langle g, v-\chi\rangle_{\Gamma_N} \,.\tag*{$\qedsymbol$}
		\end{align*}
	\end{proof}
	Next, we  identify the \emph{optimal strong convexity measures} for the primal energy functional~\eqref{eq:primal} at the primal solution $u\in K$, \textit{i.e.}, 
	$\rho_I^2\colon K\to [0,+\infty)$, for every $v\in K$ defined by
	\begin{align}\label{def:optimal_primal_error}
		\rho_I^2(v)\coloneqq I(v)-I(u)\,,
	\end{align}
	 and for the negative of the dual energy functional \eqref{eq:dual}, \textit{i.e.}, $\rho_{-D}^2\colon\hspace*{-0.1em} K^*\hspace*{-0.1em}\to\hspace*{-0.1em} [0,+\infty)$,  
	  for every~${y\hspace*{-0.1em}\in\hspace*{-0.1em} K^*}$ defined by
	 \begin{align}\label{def:optimal_dual_error}
	 	\rho_{-D}^2(y)\coloneqq- D(y)+D(z)\,,
	 \end{align}
	 which will serve as \emph{`natural'} error quantities in the primal-dual gap identity (\textit{cf}.\ Theorem \ref{thm:prager_synge_identity}).\enlargethispage{10mm}
	
	\begin{lemma}[optimal strong convexity measures]\label{lem:strong_convexity_measures}
		The following statements apply:
		\begin{itemize}[noitemsep,topsep=2pt,leftmargin=!,labelwidth=\widthof{(ii)}]
			\item[(i)] For every $v\in K$, we have that
			\begin{align*}
				\rho_I^2(v)=\tfrac{1}{2}\|\nabla v-\nabla u\|_{\Omega}^2+\langle z\cdot n,v-\chi\rangle_{\partial\Omega}-\langle g,v-\chi\rangle_{\Gamma_N}\,.
			\end{align*}
			\item[(ii)] For every $y\in K^*$, we have that
			\begin{align*}
				\rho_{-D}^2(z)=\tfrac{1}{2}\|y-z\|_{\Omega}^2+\langle y\cdot n,u-\chi \rangle_{\partial\Omega}-\langle g,u-\chi\rangle_{\Gamma_N}\,.
			\end{align*}
		\end{itemize}
	\end{lemma}
	
	\begin{remark}\hphantom{                   }
		\begin{itemize}[noitemsep,topsep=2pt,leftmargin=!,labelwidth=\widthof{(iii)},font=\itshape]
			\item[(i)] By the convex optimality relation \eqref{eq:optimality.2},  the integration-by-parts formula \eqref{eq:pi_cont}, the convex optimality relation \eqref{eq:optimality.1}, the admissibility condition \eqref{eq:admissibility.1},  and the primal variational~inequality \eqref{eq:variational_ineq}, for~every~${v\in K}$,~we~have~that
			\begin{align*}
				\langle z\cdot n,v-\chi\rangle_{\partial\Omega}-\langle g,v-\chi\rangle_{\Gamma_N}&=\langle z\cdot n,v-u\rangle_{\partial\Omega}-\langle g,v-u\rangle_{\Gamma_N}
				\\&=(z,\nabla v -\nabla u)_{\Omega}+(\textup{div}\, z,v-u)_{\Omega}-\langle g,v-u\rangle_{\Gamma_N}
				\\&=(\nabla u,\nabla v -\nabla u)_{\Omega}-(f,v-u)_{\Omega}-\langle g,v-u\rangle_{\Gamma_N}
				\\&\ge 0\,.
			\end{align*}
			
			\item[(ii)] By the convex optimality relation \eqref{eq:optimality.2}, the integration-by-parts formula \eqref{eq:pi_cont}, the convex optimality relation \eqref{eq:optimality.1},  the admissibility condition \eqref{eq:admissibility.1}, and the dual variational inequality \eqref{dual_variational_ineq},  for every $y\in K^*$, we have that
			\begin{align*}
				\langle y\cdot n,u-\chi\rangle_{\partial\Omega}-\langle g,u-\chi\rangle_{\Gamma_N}&=\langle y\cdot n-z\cdot n,u\rangle_{\partial\Omega}-\langle y\cdot n-z\cdot n,\chi \rangle_{\partial\Omega}\\
				&=(y-z,\nabla u)_{\Omega}+(\textup{div}\,y-\textup{div}\,z,u)_{\Omega}-\langle y\cdot n-z\cdot n,\chi \rangle_{\partial\Omega}
				\\&=(y-z,z)_{\Omega}-\langle y\cdot n-z\cdot n,\chi \rangle_{\partial\Omega} 
				\\&\ge 0\,.
			\end{align*}
		\end{itemize}
	\end{remark}
	
	\begin{proof}[Proof \hspace*{-0.15mm}(of \hspace*{-0.15mm}Lemma \hspace*{-0.1mm}\ref{lem:strong_convexity_measures})]\let\qed\relax
		
		\emph{ad (i).} Using the binomial formula, the convex optimality relation \eqref{eq:optimality.1}, the admissibility condition \eqref{eq:admissibility.1},  the integration-by-parts~formula \eqref{eq:pi_cont}, and the convex optimality relation \eqref{eq:optimality.2},  for every $v\in K$, we find that
		\begin{align*}
			I(v)-I(u)&=\tfrac{1}{2}\|\nabla v\|_{\Omega}^2-\tfrac{1}{2}\|\nabla u\|_{\Omega}^2-(f,v-u)_{\Omega}-\langle g,v-u\rangle_{\Gamma_N}\\&
			=\tfrac{1}{2}\|\nabla v-\nabla u\|_{\Omega}^2+(\nabla u,\nabla v-\nabla u)_{\Omega}-(f,v-u)_{\Omega}-\langle g,v-u\rangle_{\Gamma_N}
			\\&
			=\tfrac{1}{2}\|\nabla v-\nabla u\|_{\Omega}^2+(z,\nabla v-\nabla u)_{\Omega}+(\textup{div}\,z,v-u)_{\Omega}-\langle g,v-u\rangle_{\Gamma_N}
			\\&
			=\tfrac{1}{2}\|\nabla v-\nabla u\|_{\Omega}^2+\langle z\cdot n,v-u\rangle_{\partial\Omega}-\langle g,v-u\rangle_{\Gamma_N}
			\\&
			=\tfrac{1}{2}\|\nabla v-\nabla u\|_{\Omega}^2+\langle z\cdot n,v-\chi\rangle_{\partial\Omega}-\langle g,v-\chi\rangle_{\Gamma_N}\,.
		\end{align*}
		
		\emph{ad (ii).} Using the binomial formula, the admissibility conditions \eqref{eq:admissibility.1}--\eqref{eq:admissibility.3},~the~convex~opti-mality \hspace*{-0.1mm}relation \hspace*{-0.1mm}\eqref{eq:optimality.1},  \hspace*{-0.1mm}the \hspace*{-0.1mm}integration-by-parts \hspace*{-0.1mm}formula \hspace*{-0.1mm}\eqref{eq:pi_cont}, \hspace*{-0.1mm}again, \hspace*{-0.1mm}the~\hspace*{-0.1mm}\mbox{admissibility}~\hspace*{-0.1mm}\mbox{condition}~\hspace*{-0.1mm}\eqref{eq:admissibility.1}, and  the convex optimality relation \eqref{eq:optimality.2}, for every $y\in K^*$, we find that
		\begin{align*}
			-D(y)+D(z)&=\tfrac{1}{2}\|y\|_{\Omega}^2-\tfrac{1}{2}\|z\|_{\Omega}^2+\langle z\cdot n-y\cdot n,\chi \rangle_{\partial\Omega}\\&
			=\tfrac{1}{2}\|y-z\|_{\Omega}^2+(z,y-z)_{\Omega}+\langle z\cdot n-y\cdot n,\chi \rangle_{\partial\Omega}
			\\&
			=\tfrac{1}{2}\|y-z\|_{\Omega}^2+(\nabla u, y-z)_{\Omega}+\langle z\cdot n-y\cdot n,\chi \rangle_{\partial\Omega}
			\\&
			=\tfrac{1}{2}\|y-z\|_{\Omega}^2+(\textup{div}\,z-\textup{div}\,y,u)_{\Omega}+\langle z\cdot n-y\cdot n,\chi-u \rangle_{\partial\Omega}
			\\&
			=\tfrac{1}{2}\|y-z\|_{\Omega}^2+\langle y\cdot n,u-\chi \rangle_{\partial\Omega}-\langle g,u-\chi\rangle_{\Gamma_N}\,.\tag*{$\qedsymbol$}
		\end{align*}
	\end{proof}
	
	Eventually, we have everything at our disposal to establish an \emph{a posteriori} error identity that identifies the \emph{primal-dual total error} $\rho_{\textup{tot}}^2\colon K\times K^*\to [0,+\infty)$, for every $v\in K$ and $y\in K^*$ defined by 
	\begin{align}\label{def:primal_dual_total_error}
		\rho_{\textup{tot}}^2(v,y)\coloneqq \rho_I^2(v)+\rho_{-D}^2(y)\,,
	\end{align}
	with the primal-dual gap estimator $\eta_{\textup{gap}}^2\colon K\times K^*\to [0,+\infty)$ (\textit{cf}.\ \eqref{eq:primal-dual.1}).\enlargethispage{10mm}

	\begin{theorem}[primal-dual gap identity]\label{thm:prager_synge_identity}
		For every $v\in K$ and $y\in K^*$, we have that
		\begin{align*}
			\rho_{\textup{tot}}^2(v,y)
			=\eta_{\textup{gap}}^2(v,y)\,.
		\end{align*}
	\end{theorem}
	
	\begin{proof}\let\qed\relax
		Combining the definitions \eqref{eq:primal-dual.1}, \eqref{def:optimal_primal_error}, \eqref{def:optimal_dual_error}, \eqref{def:primal_dual_total_error}, and the strong duality relation \eqref{eq:strong_duality}, 
		for every $v\in K$ and $y\in K^*$, we find that
		\begin{align*}
			\rho_{\textup{tot}}^2(v,y)&=\rho_I^2(v)+\rho_{-D}^2(y)
			\\&=I(v)-I(u)+D(z)-D(y)
			%\\&=I(v)+[D(z)-I(u)]-D(y)
			\\&=I(v)-D(y)
			\\&=\eta_{\textup{gap}}^2(v,y)\,.\tag*{$\qedsymbol$}
		\end{align*}
	\end{proof} 
	
	Note that the primal-dual gap identity (\textit{cf}.\ Theorem \ref{thm:prager_synge_identity}) applies to arbitrary conforming approximations of the primal problem \eqref{eq:primal} and the dual problem \eqref{eq:dual}.
	To~be~numerically~practicable it is necessary to have a computationally inexpensive way to approximate the primal and the dual problem at the same time. 
	In  Section \ref{sec:discrete_signorini}, exploiting orthogonality relations between the Crouzeix--Raviart  and the Raviart--Thomas element,
	we transfer all convex duality relations from Section \ref{sec:continuous_signorini} to a discrete level to arrive at a discrete reconstruction formula that allows us to approximate the primal and the dual problem at the same time using only the~\mbox{Crouzeix--Raviart}~element.\newpage
%	To this end, resorting to (discrete) convex duality relations between a non-conforming Crouzeix--Raviart approximation of the primal problem and a  Raviart--Thomas approximation of the dual problem, we derive discrete reconstruction formulas.
%	
%	For the a posteriori error estimators \eqref{eq:representation}  for being numerically practicable, it is necessary to have a 
%	computationally cheap way to obtain sufficiently accurate approximations of the dual solution %for \eqref{eq:a_posteriori_primal})
%	and/or of the primal solution,  respectively.
%	%(for \eqref{eq:a_posteriori_dual}), respectively. 
%	In Section~\ref{sec:discrete_duality}, resorting to (discrete) convex duality relations between a non-conforming Crouzeix--Raviart approximation of the primal problem and a  Raviart--Thomas approximation of the dual problem, we arrive at discrete reconstruction formulas, called \textit{generalized Marini formula}~(\textit{cf}.~\cite{Mar85,Bar21}).
%	

	\section{Discrete scalar Signorini problem}\label{sec:discrete_signorini}\enlargethispage{10.5mm}
	
	\qquad In this section, we discuss the discrete scalar Signorini problem.\medskip
	
	\hspace*{-2.5mm}$\bullet$ \emph{Discrete \hspace*{-0.15mm}primal \hspace*{-0.15mm}problem.} \hspace*{-0.15mm}Let  \hspace*{-0.15mm}$f_h\hspace*{-0.175em}\in\hspace*{-0.175em} \mathcal{L}^0(\mathcal{T}_h)$, \hspace*{-0.15mm}$g_h\hspace*{-0.175em}\in\hspace*{-0.175em} \mathcal{L}^0(\mathcal{S}_h^{\Gamma_N})$, \hspace*{-0.15mm}$u_D^h\hspace*{-0.175em}\in\hspace*{-0.175em} \mathcal{L}^0(\mathcal{S}_h^{\Gamma_D})$,~\hspace*{-0.15mm}and~\hspace*{-0.15mm}${\chi_h\hspace*{-0.175em}\in\hspace*{-0.175em} \mathcal{L}^0(\mathcal{S}_h^{\Gamma_D}\hspace*{-0.175em}\cup\hspace*{-0.175em}\mathcal{S}_h^{\Gamma_C})}$ with $\chi_h=u_D^h$ a.e.\ in $\Gamma_D$. Then, the \emph{discrete scalar Signorini problem} is given via the minimization of $I_h^{cr}\colon \mathcal{S}^{1,cr}(\mathcal{T}_h)\to \mathbb{R}\cup\{+\infty\}$, for every $v_h\in \mathcal{S}^{1,cr}(\mathcal{T}_h)$ defined by
	\begin{align}
		\begin{aligned}
		I_h^{cr}(v_h)&\coloneqq \tfrac{1}{2}\| \nabla_hv_h\|_{\Omega}^2-(f_h,\Pi_hv_h)_{\Omega}-( g_h,\pi_h v_h)_{\Gamma_N}+I_{\smash{K_h^{cr}}}(v_h)
		\\&\,= \tfrac{1}{2}\| \nabla_hv_h\|_{\Omega}^2-(f_h,\Pi_hv_h)_{\Omega}-(g_h,\pi_h v_h)_{\Gamma_N}+I_{\{u_D^h\}}^{\Gamma_D}(\pi_h v_h)+\smash{I_+^{\Gamma_C}}(\pi_h v_h-\chi_h)\,,
		\end{aligned}
		\label{eq:discrete_primal}
	\end{align}
	where
	\begin{align*}
	 	K_h^{cr}\coloneqq \big\{v_h\in \mathcal{S}^{1,cr}(\mathcal{T}_h)\mid \pi_hv_h=u_D^h\text{ a.e.\ on }\Gamma_D\,,\;\pi_h v_h\ge \chi_h\text{ a.e.\ on }\Gamma_C\big\}\,,
	\end{align*}
	and $\smash{I_{K_h^{cr}}}\coloneqq \smash{I_{\{u_D^h\}}^{\Gamma_D}(\pi_h (\cdot))+I_h^{\Gamma_C}(\pi_h(\cdot)-\chi_h)}\colon \mathcal{S}^{1,cr}(\mathcal{T}_h)\to \mathbb{R}\cup\{+\infty\}$.
	
	\hspace*{-0.5mm}In \hspace*{-0.175mm}what \hspace*{-0.175mm}follows, \hspace*{-0.175mm}we \hspace*{-0.175mm}refer \hspace*{-0.175mm}to \hspace*{-0.175mm}the \hspace*{-0.175mm}minimization \hspace*{-0.175mm}of \hspace*{-0.175mm}the \hspace*{-0.175mm}functional \hspace*{-0.175mm}\eqref{eq:discrete_primal} \hspace*{-0.175mm}as \hspace*{-0.175mm}the \hspace*{-0.175mm}\emph{discrete~\hspace*{-0.175mm}\mbox{primal}~\hspace*{-0.175mm}\mbox{problem}}. 
	Since the functional \eqref{eq:discrete_primal} is proper, strictly convex, weakly coercive, and lower semi-continuous, the direct method in the calculus of variations yields the existence~of~a~unique~minimizer~$u_h^{cr}\in K_h^{cr}$, called \hspace*{-0.175mm}the \hspace*{-0.175mm}\emph{discrete \hspace*{-0.175mm}primal \hspace*{-0.175mm}solution}. \hspace*{-0.5mm}We \hspace*{-0.175mm}reserve \hspace*{-0.175mm}the \hspace*{-0.175mm}notation \hspace*{-0.175mm}$u_h^{cr}\hspace*{-0.175em}\in\hspace*{-0.175em} K_h^{cr}$ \hspace*{-0.1mm}for \hspace*{-0.175mm}the~\hspace*{-0.175mm}\mbox{discrete}~\hspace*{-0.175mm}\mbox{primal}~\hspace*{-0.175mm}\mbox{solution}. 
	
    \hspace*{-2.5mm}$\bullet$	\emph{Discrete primal variational inequality.} The discrete primal solution $u_h^{cr}\in K_h^{cr}$ equivalently is the unique solution of the following 
      variational inequality:
     for every $v_h\in K_h^{cr}$, it holds that
	\begin{align}
		(\nabla_h u_h^{cr},\nabla_h u_h^{cr}-\nabla_h v_h)_{\Omega}\leq (f_h,\Pi_h u_h^{cr}-\Pi_h v_h)_{\Omega}+( g_h,\pi_h u_h^{cr}-\pi_h v_h)_{\Gamma_N}\,.\label{eq:discrete_variational_ineq}
	\end{align}
	
	\hspace*{-2.5mm}$\bullet$ \emph{Discrete dual problem.} The \emph{(Fenchel) dual problem} to the discrete scalar Signorini problem is given via the maximization of $D_h^{rt}\colon \mathcal{R}T^0(\mathcal{T}_h)\to \mathbb{R}\cup\{-\infty\}$, for every $y_h\in \mathcal{R}T^0(\mathcal{T}_h)$~defined~by 
	\begin{align}
		\begin{aligned}
		D_h^{rt}(y_h)&\coloneqq -\tfrac{1}{2}\| \Pi_hy_h\|_{\Omega}^2+(y_h\cdot n,\chi_h)_{\Gamma_D\cup\Gamma_C}-I_{\smash{K_h^{rt,*}}}(y_h) 
		\\&\,= -\tfrac{1}{2}\| \Pi_hy_h\|_{\Omega}^2+(y_h\cdot n,\chi_h )_{\Gamma_D\cup\Gamma_C}-
			I_{\{-f_h\}}^{\Omega} (\textup{div}\,y_h)-I_{\{g_h\}}^{\Gamma_N}(y_h\cdot n)-\smash{I_+^{\Gamma_C}}(y_h\cdot n)\,,
		\end{aligned}\hspace*{-2mm}\label{eq:discrete_dual}
	\end{align}
	where 
	\begin{align*} 
		K_h^{rt,*}\coloneqq \big\{y_h\in \mathcal{R}T^0(\mathcal{T}_h)\mid \textup{div}\, y_h =-f_h\text{ a.e.\ in }\Omega\,,\; y_h\cdot n =g_h\text{ a.e.\ }\Gamma_N\,,\; y_h\cdot n\ge 0\text{ a.e.\ on }\Gamma_C\big\}\,,
	\end{align*}
	and $\smash{I_{K_h^{rt,*}}}\coloneqq \smash{I_{\{-f_h\}}^{\Omega}(\textup{div}\,(\cdot))+(I_{\{g_h\}}^{\Gamma_N}+\smash{I_+^{\Gamma_C}})((\cdot)\cdot n)}\colon \mathcal{R}T^0(\mathcal{T}_h)\to \mathbb{R}\cup\{+\infty\}$.
	
	The identification of the (Fenchel) dual problem (in the sense of \cite[Rem.\ 4.2, p.\ 60/61]{ET99})~to~the minimization of \eqref{eq:discrete_primal} with the maximization of  \eqref{eq:discrete_dual} can be found in the proof of the following result that also establishes the validity of a discrete strong duality relation and discrete convex optimality relations.
	
		\begin{proposition}[strong duality and convex duality relations]\label{prop:discrete_duality} The following statements apply:
		\begin{itemize}[noitemsep,topsep=2pt,leftmargin=!,labelwidth=\widthof{(ii)}]
			\item[(i)]  The (Fenchel) dual problem to the discrete scalar Signorini problem is defined via the maximization of \eqref{eq:discrete_dual}.  
			\item[(ii)]  There exists a unique maximizer $z_h^{rt}\in  \mathcal{R}T^0(\mathcal{T}_h)$ of \eqref{eq:discrete_dual} satisfying the \textup{discrete~admissibility conditions}
			\begin{alignat}{2}
				\textup{div}\,z_h^{rt}&=-f_h&&\quad\text{ a.e.\ in }\Omega\label{eq:discrete_admissibility.1}\,,\\
				z_h^{rt}\cdot n &= g_h&&\quad\text{ a.e.\ on }\Gamma_N\label{eq:discrete_admissibility.2}\,,\\
				z_h^{rt}\cdot n &\ge 0 &&\quad\text{ a.e.\ on }\Gamma_C\label{eq:discrete_admissibility.3}\,,
			\end{alignat}
			In addition, there
			holds a \emph{discrete strong duality relation}, \textit{i.e.}, we have that
			\begin{align}
				I_h^{cr}(u_h^{cr}) = D_h^{rt}(z_h^{rt})\,.\label{eq:discrete_strong_duality}
			\end{align}
			\item[(iii)] There hold the \emph{discrete convex optimality relations}, \textit{i.e.}, we have that
			\begin{alignat}{2}
				\Pi_h z_h^{rt}&=\nabla_h u_h^{cr}&&\quad\text{ a.e.\ in }\Omega\label{eq:discrete_optimality.1}\,,\\
				z_h^{rt}\cdot n\,(\pi_hu_h^{cr}-\chi_h)&=0&&\quad\text{ a.e.\ on }\Gamma_C\label{eq:discrete_optimality.2}\,.
			\end{alignat}
		\end{itemize}
	\end{proposition}
	
		\begin{proof}
		\emph{ad (i).} If we introduce the functionals $G_h\colon (\mathcal{L}^0(\mathcal{T}_h))^d\to \mathbb{R}$ and $F_h\colon \mathcal{S}^{1,cr}(\mathcal{T}_h)\to \mathbb{R}\cup\{+\infty\}$, for every $\overline{y}_h\in  (\mathcal{L}^0(\mathcal{T}_h))^d$ and $v_h\in\mathcal{S}^{1,cr}(\mathcal{T}_h)$ defined by\vspace*{-0.25mm}
		\begin{align*}
			G_h(\overline{y}_h)&\coloneqq \tfrac{1}{2}\|\overline{y}_h\|_{\Omega}^2\,,\\
			F_h(v_h)&\coloneqq -(f_h,\Pi_h v_h)_{\Omega}-( g_h,\pi_h v_h)_{\Gamma_N}+I_{\{u_D^h\}}^{\Gamma_D}(\pi_h v_h)+\smash{I_+^{\Gamma_C}}(\pi_h v_h-\chi_h)\,,
		\end{align*}
		then, for every $v_h\in \smash{\mathcal{S}^{1,cr}(\mathcal{T}_h)}$, we have that\vspace*{-0.25mm}
		\begin{align*}
			\smash{I_h^{cr}(v_h)= G_h(\nabla_h v_h)+F_h(v_h)\,.}
		\end{align*}
		Thus, \hspace*{-0.175mm}in \hspace*{-0.175mm}accordance \hspace*{-0.175mm}with \hspace*{-0.175mm}\cite[\hspace*{-0.2mm}Rem.\ \hspace*{-0.2mm}4.2, \hspace*{-0.2mm}p.\ \hspace*{-0.2mm}60]{ET99}, \hspace*{-0.175mm}the \hspace*{-0.175mm}(Fenchel) \hspace*{-0.175mm}dual \hspace*{-0.175mm}problem \hspace*{-0.175mm}to \hspace*{-0.175mm}the \hspace*{-0.175mm}\mbox{minimization}~\hspace*{-0.175mm}of~\hspace*{-0.175mm}\eqref{eq:discrete_primal} is given via the maximization of $D_h^0\colon (\mathcal{L}	^0(\mathcal{T}_h))^d\to \mathbb{R}\cup\{-\infty\}$, 
		for every $\overline{y}_h\in (\mathcal{L}^0(\mathcal{T}_h))^d$~defined~by\vspace*{-0.25mm}
		\begin{align}\label{helper_functional}
			\smash{D_h^0(\overline{y}_h)\coloneqq -G^*_h(\overline{y}_h)- F^*_h(-\nabla^*_h\overline{y}_h)\,,}
		\end{align}
		where $\nabla^*_h\colon  (\mathcal{L}^0(\mathcal{T}_h))^d\to (\mathcal{S}^{1,cr}(\mathcal{T}_h))^*$ denotes the adjoint operator to $\nabla_h \colon \mathcal{S}^{1,cr}(\mathcal{T}_h)\to (\mathcal{L}^0(\mathcal{T}_h))^d$. 
		For every $\overline{y}_h\in (\mathcal{L}^0(\mathcal{T}_h))^d$, we have that\vspace*{-0.25mm}
		\begin{align}\label{prop:discrete_duality.1}
			\smash{G^*_h(\overline{y}_h)=\tfrac{1}{2}\|\overline{y}_h\|_{\Omega}^2\,.}
		\end{align}
		For fixed $\widehat{\chi}_h\in \mathcal{S}^{1,cr}(\mathcal{T}_h)$ with $\pi_h\widehat{\chi}_h=\chi_h$ a.e.\ on  $\Gamma_D\cup\Gamma_C$,  due to $\chi_h=u_D^h$ a.e.~on~$\Gamma_D$,
		for every $\overline{y}_h\in (\mathcal{L}^0(\mathcal{T}_h))^d$, using the lifting lemma (\textit{cf}.\ Lemma \ref{lem:reconstruction})  and the discrete integration-by-parts formula \eqref{eq:pi}, we find that\vspace*{-0.5mm}
		\begin{align}\label{prop:discrete_duality.1.2}
			&F^*_h(-\nabla^*_h\overline{y}_h)\\&=\sup_{v_h\in  \mathcal{S}^{1,cr}(\mathcal{T}_h)}{\hspace*{-0.1em}\big\{\hspace*{-0.1em}-\hspace*{-0.1em}(\overline{y}_h,\nabla_h v_h )_{\Omega}\hspace*{-0.1em}+\hspace*{-0.1em}(f_h,\Pi_h v_h)_{\Omega}\hspace*{-0.1em}+\hspace*{-0.1em}(g_h,\pi_hv_h)_{\Gamma_N}\hspace*{-0.1em}+\hspace*{-0.1em}I_{\{u_D^h\}}^{\Gamma_D}(\pi_h v_h)\hspace*{-0.1em}+\hspace*{-0.1em}\smash{I_+^{\Gamma_C}}(\pi_hv_h-\chi_h)\big\}}\notag 
			\\&=\sup_{v_h\in  \mathcal{S}^{1,cr}_D(\mathcal{T}_h)}{\hspace*{-0.1em}\big\{\hspace*{-0.1em}-\hspace*{-0.1em}(\overline{y}_h,\nabla_h v_h )_{\Omega}\hspace*{-0.1em}+\hspace*{-0.1em}(f_h,\Pi_h v_h)_{\Omega}\hspace*{-0.1em}+\hspace*{-0.1em}( g_h,\pi_hv_h)_{\Gamma_N}\hspace*{-0.1em}+\hspace*{-0.1em}\smash{I_+^{\Gamma_C}}(\pi_hv_h)\big\}}\notag
			\\&\quad-(\overline{y}_h,\nabla_h \widehat{\chi}_h )_{\Omega}+(f_h,\Pi_h\widehat{\chi}_h)_{\Omega}+( g_h,\pi_h\widehat{\chi}_h)_{\Gamma_N}\notag
			\\&= 
			\begin{cases}
			\left.	\begin{aligned}
				I_{\{-f_h\}}^{\Omega}(\textup{div}\,y_h)+I_{\{g_h\}}^{\Gamma_N}(y_h\cdot n)+\smash{I_+^{\Gamma_C}}(y_h\cdot n)\\
				-(\Pi_hy_h,\nabla_h\widehat{\chi}_h )_{\Omega}+(f_h,\Pi_h\widehat{\chi}_h)_{\Omega}+(g_h,\pi_h\chi_h)_{\Gamma_N}\end{aligned}\right\}&\text{ if }\overline{y}_h=\Pi_hy_h\text{ for some }y_h\in \mathcal{R}T^0(\mathcal{T}_h)\,,\\[-0.75mm]
				+\infty&\text{ else}\,,
			\end{cases} \notag
				\\&= 
			\begin{cases}
				\left.	\begin{aligned}
					I_{\{-f_h\}}^{\Omega}(\textup{div}\,y_h)+I_{\{g_h\}}^{\Gamma_N}(y_h\cdot n)+\smash{I_+^{\Gamma_C}}(y_h\cdot n)\\
						-( y_h\cdot n,\chi_h)_{\Gamma_D\cup\Gamma_C}\end{aligned}\right\}
						&\text{ if }\overline{y}_h=\Pi_hy_h\text{ for some }y_h\in \mathcal{R}T^0(\mathcal{T}_h)\,,\\[-0.75mm]
				+\infty&\text{ else}\,.
			\end{cases} \notag
		\end{align}  
		Using \eqref{prop:discrete_duality.1} and \eqref{prop:discrete_duality.1.2} in \eqref{helper_functional}, for every $\overline{y}_h=\Pi_h y_h\in  \Pi_h(\mathcal{R}T^0(\mathcal{T}_h))$, where $y_h\in \mathcal{R}T^0(\mathcal{T}_h)$, 
		we arrive at the~\mbox{representation}~\eqref{eq:discrete_dual}.
		Since $D_h^0= -\infty$ in $(\mathcal{L}^0(\mathcal{T}_h))^d\setminus \Pi_h(\mathcal{R}T^0(\mathcal{T}_h))$, it is enough to restrict~\eqref{eq:discrete_dual}~to~$ \Pi_h(\mathcal{R}T^0(\mathcal{T}_h))$.~Then, we have that $D_h^0(\Pi_h y_h)=D_h^{rt}(y_h)$ for all $y_h\in\mathcal{R}T^0(\mathcal{T}_h)$. 
		
		\emph{ad (ii).}  Since $G_h\colon\hspace*{-0.1em} (\mathcal{L}^0(\mathcal{T}_h))^d\hspace*{-0.1em}\to\hspace*{-0.1em} \mathbb{R}$ and $F_h\colon \hspace*{-0.1em} \mathcal{S}^{1,cr}(\mathcal{T}_h)\hspace*{-0.1em}\to\hspace*{-0.1em} \mathbb{R}\cup\{+\infty\}$ are proper,~convex,~and~lower semi-continuous and since $G_h\colon \hspace*{-0.15em}(\mathcal{L}^0(\mathcal{T}_h))^d\hspace*{-0.15em}\to \hspace*{-0.15em}\mathbb{R}$  is continuous~at~${\widehat{\chi}_h\hspace*{-0.15em}\in\hspace*{-0.15em} \textup{dom}(F_h)\hspace*{-0.15em}\cap \hspace*{-0.15em}\textup{dom}(G_h\hspace*{-0.15em}\circ\hspace*{-0.15em} \nabla_h)}$, \textit{i.e.},\vspace*{-0.25mm}
		\begin{align*}
			\smash{G_h(\overline{y}_h)\to G_h(\nabla_h\widehat{\chi}_h)\quad\big(\overline{y}_h\to \nabla_h\widehat{\chi}_h\quad\text{ in }(\mathcal{L}^0(\mathcal{T}_h))^d\big)\,,}
		\end{align*}
		by \hspace*{-0.1mm}the \hspace*{-0.1mm}celebrated \hspace*{-0.1mm}Fenchel \hspace*{-0.1mm}duality \hspace*{-0.1mm}theorem \hspace*{-0.1mm}(\emph{cf}.\ \hspace*{-0.1mm}\cite[\hspace*{-0.1mm}Rem.\ \hspace*{-0.1mm}4.2, \hspace*{-0.1mm}(4.21), \hspace*{-0.1mm}p.\ \hspace*{-0.1mm}61]{ET99}),  \hspace*{-0.1mm}there~\hspace*{-0.1mm}exists~\hspace*{-0.1mm}a~\hspace*{-0.1mm}\mbox{maximizer}   $z_h^0\in  (\mathcal{L}^0(\mathcal{T}_h))^d$ of  \eqref{helper_functional} and a discrete strong duality relation applies, \textit{i.e.},\vspace*{-0.25mm}
		\begin{align*}
			\smash{I_h^{cr}(u_h^{cr})=D_h^0(z_h^0)\,.}
		\end{align*}
		Since $D_h^0= -\infty$ in $(\mathcal{L}^0(\mathcal{T}_h))^d\setminus \Pi_h(\mathcal{R}T^0(\mathcal{T}_h))$, there exists  $z_h^{rt}\in \mathcal{R}T^0(\mathcal{T}_h)$ such that $z_h^0=\Pi_h z_h^{rt}$~a.e.\ in $\Omega$.
		In particular, we have that $D_h^0(z_h^0)=D_h^{rt}(z_h^{rt})$, so that $z_h^{rt}\in \mathcal{R}T^0(\mathcal{T}_h)$ is a maximizer of \eqref{eq:discrete_dual} and the discrete strong duality relation \eqref{eq:discrete_strong_duality} applies.  By the strict convexity of  ${G_h\colon \hspace*{-0.1em}(\mathcal{L}^0(\mathcal{T}_h))^d\hspace*{-0.1em}\to\hspace*{-0.1em}\mathbb{R}}$ and the divergence constraint, the maximizer $z_h^{rt}\in \mathcal{R}T^0(\mathcal{T}_h)$ is uniquely~determined.
		
		\emph{ad (iii).} \hspace*{-0.15mm}By \hspace*{-0.15mm}the \hspace*{-0.15mm}standard \hspace*{-0.15mm}(Fenchel) \hspace*{-0.15mm}convex \hspace*{-0.15mm}duality \hspace*{-0.15mm}theory \hspace*{-0.15mm}(\emph{cf}.\ \hspace*{-0.15mm}\cite[Rem.\ \hspace*{-0.15mm}4.2, \hspace*{-0.15mm}(4.24),~\hspace*{-0.15mm}(4.25),~\hspace*{-0.15mm}p.~\hspace*{-0.15mm}61]{ET99}), there hold the convex optimality relations\vspace*{-0.25mm}\enlargethispage{11.5mm}
		\begin{align}
			-\nabla_h^*\Pi_h z_h^{rt}&\in \partial F_h(u_h)\,,\label{prop:discrete_duality.2}\\
			\Pi_h z_h^{rt}&\in \partial G_h(\nabla_h u_h^{cr})\,.\label{prop:discrete_duality.3}
		\end{align}
		The inclusion  \eqref{prop:discrete_duality.3} is equivalent to the discrete convex optimality relation \eqref{eq:discrete_optimality.1}. The inclusion \eqref{prop:discrete_duality.2}, by the standard equality condition in the Fenchel--Young inequality (\textit{cf}.\ \cite[Prop.\ 5.1, p.\ 21]{ET99}), $\pi_hu_h^{cr}=\chi_h$ a.e.\ on $\Gamma_D$, and the discrete admissibility conditions \eqref{eq:discrete_admissibility.1},\eqref{eq:discrete_admissibility.2},~is~\mbox{equivalent}~to 
		\begin{align*}
			-(\Pi_h z_h^{rt},  \nabla_h u_h^{cr})_{\Omega}&=	(-\nabla^*_h\Pi_h z_h^{rt},  u_h^{cr})_{\Omega}\\
			&=F^*_h(-\nabla^*_h\Pi_h z_h^{rt})+F_h(u_h^{cr})
			\\&=-(z_h^{rt}\cdot n,\chi_h)_{\Gamma_D\cup\Gamma_C}-(f_h,\Pi_h u_h^{cr})_{\Omega}-(g_h,\pi_hu_h^{cr})_{\Gamma_N} 
			\\&=( z_h^{rt}\cdot n,\pi_hu_h^{cr}-\chi_h)_{\Gamma_C}+(\textup{div}\,z_h^{rt},\Pi_hu_h^{cr})_{\Omega}-( z_h^{rt}\cdot n,\pi_hu_h^{cr})_{\partial\Omega}\,,
		\end{align*}
		which, \hspace*{-0.15mm}by \hspace*{-0.15mm}the \hspace*{-0.15mm}discrete \hspace*{-0.15mm}integration-by-parts \hspace*{-0.15mm}formula \hspace*{-0.15mm}\eqref{eq:pi}, \hspace*{-0.15mm}is \hspace*{-0.15mm}equivalent \hspace*{-0.15mm}to \hspace*{-0.15mm}${( z_h^{rt}\hspace*{-0.1em}\cdot\hspace*{-0.1em} n,\pi_hu_h^{cr}\hspace*{-0.1em}-\hspace*{-0.1em}\chi_h)_{\Gamma_C}\hspace*{-0.175em}=\hspace*{-0.175em}0}$. 
		Since $z_h^{rt}\cdot n\ge 0 $ a.e.\ on $\Gamma_C$ and $\pi_hu_h^{cr}-\chi_h\ge 0$ a.e.\ on $\Gamma_C$, we conclude that
		%the discrete convex optimality~relation 
		\eqref{eq:discrete_optimality.1} applies. 
	\end{proof}

	\hspace*{-2.5mm}$\bullet$ \emph{Discrete dual variational inequality.} A discrete dual solution $z_h^{rt}\hspace*{-0.1em}\in\hspace*{-0.1em} \smash{K_h^{rt,*}}$ equivalently is the unique \mbox{solution} of the following variational inequality:
	 for every $y_h\in \smash{K_h^{rt,*}}$, it holds that\vspace*{-0.5mm}
	\begin{align}
			\smash{(\Pi_h z_h^{rt},\Pi_h z_h^{rt}-\Pi_h y_h)_{\Omega}\geq ( z_h^{rt}\cdot n - y_h\cdot n,\chi_h)_{\Gamma_N}\,.}\label{eq:discrete_variational_dual_ineq}
	\end{align}
	
	\hspace*{-2.5mm}$\bullet$ \emph{Discrete augmented problem.} There exists a discrete Lagrange multiplier $\smash{\Lambda_h^{cr,*}}\in (\mathcal{S}^{1,cr}_D(\mathcal{T}_h))^*$ such that for every $v_h\in \mathcal{S}^{1,cr}_D(\mathcal{T}_h)$, there holds the \emph{discrete augmented problem}\vspace*{-0.5mm}
	\begin{align}
			\langle \smash{\Lambda_h^{cr,*}},v_h\rangle_{\smash{\mathcal{S}^{1,cr}_D(\mathcal{T}_h)}}=\smash{(\nabla_h u_h^{cr},\nabla_h v_h)_{\Omega}-(f_h,\Pi_h v_h)_{\Omega}-(g_h,\pi_hv_h)_{\Gamma_N}\,.}\label{eq:discrete_augmented_problem}
	\end{align}
	With the convex optimality relation \eqref{eq:discrete_optimality.1} and the discrete  integration-by-parts formula \eqref{eq:pi}, introducing the discrete Lagrange multiplier $\smash{\overline{\lambda}}_h^{cr}\coloneqq z_h^{rt}\cdot n|_{\Gamma_C}\in \mathcal{L}^0(\mathcal{S}_h^{\Gamma_C})$, 
	for every $v_h\in \mathcal{S}^{1,cr}_D(\mathcal{T}_h)$, we find that\vspace*{-1mm}
	\begin{align*}
		\smash{\langle \smash{\Lambda_h^{cr,*}},v_h\rangle_{\smash{\mathcal{S}^{1,cr}_D(\mathcal{T}_h)}}=( \smash{\overline{\lambda}}_h^{cr}, \pi_hv_h)_{\Gamma_C}\,.}
	\end{align*}

	We define the \emph{discrete~flux}\vspace*{-1mm}
	\begin{align}\label{eq:generalized_marini}
		z_h^{rt}\coloneqq \nabla_hu_h^{cr}-\frac{f_h}{d}(\textup{id}_{\mathbb{R}^d}-\Pi_h\textup{id}_{\mathbb{R}^d})\in(\mathcal{L}^1(\mathcal{T}_h))^d\,,
	\end{align} 
	which 
	%Then, the following proposition proves that the discrete flux \eqref{eq:generalized_marini} 
	is admissible in the discrete dual problem and %even 
	a discrete dual solution.\enlargethispage{11mm}

	\begin{proposition}\label{prop:discrete_reconstruction}  The discrete flux $z_h^{rt}\in\smash{(\mathcal{L}^1(\mathcal{T}_h))^d}$, defined by \eqref{eq:generalized_marini}, 
		satisfies $z_h^{rt}\in \smash{\mathcal{R}T^0(\mathcal{T}_h)}$, the discrete admissibility relations \eqref{eq:discrete_admissibility.1}--\eqref{eq:discrete_admissibility.3},
			the discrete convex optimality relations~\eqref{eq:discrete_optimality.1},\eqref{eq:discrete_optimality.2}, and is a discrete dual solution.
	\end{proposition} 
	
	\begin{proof}\let\qed\relax
		\emph{ad $z_h^{rt}\in \mathcal{R}T^0(\mathcal{T}_h)$ with \eqref{eq:discrete_admissibility.1}--\eqref{eq:discrete_admissibility.3} and \eqref{eq:discrete_optimality.1}.}  Due to \eqref{eq:discrete_augmented_problem}, the lifting lemma (\textit{cf}.\ Lemma \ref{lem:reconstruction}) implies that $z_h^{rt}\in \mathcal{R}T^0(\mathcal{T}_h)$ with \eqref{eq:discrete_admissibility.1}, \eqref{eq:discrete_admissibility.2}, and \eqref{eq:discrete_optimality.1}. In addition,  for every $v_h\in \mathcal{S}^{1,cr}_D(\mathcal{T}_h)$ with $\pi_h v_h\ge 0$ a.e.\ on $\Gamma_C$, the discrete integration-by-parts formula \eqref{eq:pi} and the discrete primal variational inequality (\textit{cf}.\ \eqref{eq:discrete_variational_ineq}) yield that\vspace*{-0.5mm}
		\begin{align}\label{prop:discrete_duality.4}
			\begin{aligned} 
			(z_h^{rt}\cdot n, \pi_h v_h)_{\Gamma_C}&=(\Pi_h z_h^{rt},\nabla_h v_h)_{\Omega}+(\textup{div}\,z_h^{rt},\Pi_h v_h)_{\Omega}-(z_h^{rt}\cdot n,\pi_h v_h)_{\Gamma_N}
			\\&=(\nabla_h u_h^{cr},\nabla_h v_h)_{\Omega}-(f_h,\Pi_h v_h)_{\Omega}-(g_h,\pi_h v_h)_{\Gamma_N}\ge 0\,.
		\end{aligned}
		\end{align}
		Thus, choosing $v_h\hspace*{-0.1em}=\hspace*{-0.1em}\varphi_S$ for all $S\hspace*{-0.1em}\in\hspace*{-0.1em} \smash{\mathcal{S}_h^{\Gamma_C}}$ in \eqref{prop:discrete_duality.4} and exploiting that $\pi_h\varphi_S\hspace*{-0.1em}=\hspace*{-0.1em}\chi_S$,~for~every~$S\hspace*{-0.1em}\in\hspace*{-0.1em}\smash{\mathcal{S}_h^{\Gamma_C}}$, we find that 
		admissibility  condition \eqref{eq:discrete_admissibility.3} is satisfied.

		\emph{ad \eqref{eq:discrete_optimality.2}.} If $S\in \smash{\mathcal{S}_h^{\Gamma_C}}$ is such that $\pi_h u_h^{cr}>\chi_h$ on $S$, then there exists some $\alpha_S<0$~such~that
		$\pi_h u_h^{cr}+\alpha_S\chi_S\ge \chi_h$ on $S$. Thus, from $v_h= u_h^{cr}+\alpha_S\varphi_S\in \mathcal{S}^{1,cr}_D(\mathcal{T}_h)$~in~\eqref{eq:discrete_variational_ineq}, 
		we get $z_h^{rt}\cdot n=0$~on~$S$.

		\emph{ad Maximality.} Using the discrete convex optimality relations \eqref{eq:discrete_optimality.1},\eqref{eq:discrete_optimality.2},~the~discrete~integra-tion-by-parts formula \eqref{eq:pi}, and the discrete admissibility conditions \eqref{eq:discrete_admissibility.1}--\eqref{eq:discrete_admissibility.3},~we~observe~that\vspace*{-0.5mm}
		\begin{align*}
			I_h^{cr}(u_h^{cr})&=\tfrac{1}{2}\|\nabla_hu_h^{cr}\|_\Omega^2-(f_h,\Pi_h u_h^{cr})_{\Omega}-(g_h,\pi_h u_h^{cr})_{\Gamma_N}
			\\&=\tfrac{1}{2}\|\nabla_hu_h^{cr}\|_\Omega^2+(\textup{div}\,z_h^{rt},\Pi_h u_h^{cr})_{\Omega}-(z_h^{rt}\cdot n,\pi_h u_h^{cr})_{\Gamma_N}
			\\&=\tfrac{1}{2}\|\nabla_hu_h^{cr}\|_\Omega^2-(\Pi_hz_h^{rt},\nabla_hu_h^{cr})_\Omega+(z_h^{rt}\cdot n,\pi_h u_h^{cr})_{\Gamma_D\cup\Gamma_C }
			\\&=-\tfrac{1}{2}\|\Pi_h z_h^{rt}\|_\Omega^2
			+(z_h^{rt}\cdot n,\pi_h u_h^{cr})_{\Gamma_D\cup\Gamma_C }\\&
			= D_h^{rt}(z_h^{rt})\,, 
		\end{align*}
		which, \hspace*{-0.15mm}by \hspace*{-0.15mm}the \hspace*{-0.15mm}discrete \hspace*{-0.15mm}strong \hspace*{-0.15mm}duality \hspace*{-0.15mm}relation \hspace*{-0.15mm}\eqref{eq:discrete_strong_duality}, \hspace*{-0.15mm}shows \hspace*{-0.15mm}that \hspace*{-0.15mm}$z_h^{rt}\!\in\! \smash{\mathcal{R}T^0(\mathcal{T}_h)}$ \hspace*{-0.15mm}is \hspace*{-0.15mm}maximal~\hspace*{-0.15mm}for~\hspace*{-0.15mm}\eqref{eq:discrete_dual}.~\qedsymbol
	\end{proof}\newpage

	\section{\emph{A priori} error analysis}\label{sec:apriori}\enlargethispage{7.5mm}
	
	\qquad In this section, resorting to the discrete convex duality relations established in Section \ref{sec:discrete_signorini},~we derive an \emph{a priori} error identity for  the discrete primal problem \eqref{eq:discrete_primal} and the discrete~dual~problem \eqref{eq:discrete_dual} at the same time. From this \emph{a priori} error identity, in turn, we derive an \emph{a priori} error estimate with an explicit error decay rate that is quasi-optimal. 
	To this end, we~proceed~analogously to the continuous setting (\textit{cf}.\ Section \ref{sec:aposteriori}) and start with examining the \emph{discrete~primal-dual gap estimator} $\eta_{\textup{gap},h}^2\colon K_h^{cr}\times K_h^{rt,*}\to [0,+\infty)$, for every $v_h\in K_h^{cr}$ and $y_h\in K_h^{rt,*}$~defined~by\vspace*{-1mm}
	\begin{align}\label{eq:discrete_primal_dual_gap_estimator}
		\eta_{\textup{gap},h}^2(v_h,y_h)\coloneqq I_h^{cr}(v_h)-D_h^{rt}(y_h)\,.
	\end{align}
	
	The \hspace*{-0.1mm}discrete \hspace*{-0.1mm}primal-dual \hspace*{-0.1mm}gap \hspace*{-0.1mm}estimator \hspace*{-0.1mm}(\textit{cf}.\ \hspace*{-0.1mm}\eqref{eq:discrete_primal_dual_gap_estimator})  \hspace*{-0.1mm}can \hspace*{-0.1mm}be \hspace*{-0.1mm}decomposed \hspace*{-0.1mm}into \hspace*{-0.1mm}two \hspace*{-0.1mm}\mbox{contributions}~\hspace*{-0.1mm}that precisely measure the violation of the discrete convex optimality relations \eqref{eq:discrete_optimality.1},\eqref{eq:discrete_optimality.2},~respectively.
	
	\begin{lemma}[discrete primal-dual gap estimator]\label{lem:discrete_primal_dual_gap_estimator}
		For every $v_h\hspace*{-0.1em}\in\hspace*{-0.1em} K_h^{cr}$ and $y_h\hspace*{-0.1em}\in\hspace*{-0.1em} K_h^{rt,*}$,~we~have~that\vspace*{-0.5mm}
		\begin{align*}
			\eta_{\textup{gap},h}^2(v_h,y_h)&\coloneqq  \eta_{A,\textup{gap},h}^2(v_h,y_h)+\eta_{B,\textup{gap},h}^2(v_h,y_h)\,,\\
			\eta_{A,\textup{gap},h}^2(v_h,y_h)&\coloneqq 	\tfrac{1}{2} \|\nabla_h v_h-\Pi_h y_h\|_{\Omega}^2\,,\\
			\eta_{B,\textup{gap},h}^2(v_h,y_h)&\coloneqq (y_h\cdot n,\pi_hv_h-\chi_h)_{\Gamma_C}\,.
		\end{align*}
	\end{lemma}
	
	\begin{remark}[interpretation of the components of the discrete primal-dual gap estimator]\hphantom{                   }
		\begin{itemize}[noitemsep,topsep=2pt,leftmargin=!,labelwidth=\widthof{(iii)},font=\itshape]
			\item[(i)] The estimator $\eta_{A,\textup{gap},h}^2$ measures the violation of the discrete convex optimality~relation~\eqref{eq:discrete_optimality.1};
			\item[(ii)] The estimator $\eta_{B,\textup{gap},h}^2$ measures the violation of the discrete convex optimality~relation~\eqref{eq:discrete_optimality.2}.
		\end{itemize}
	\end{remark}
	
	\begin{proof}[Proof (of Lemma \ref{lem:discrete_primal_dual_gap_estimator})]\let\qed\relax
		Using the discrete admissibility conditions \eqref{eq:discrete_admissibility.1}--\eqref{eq:discrete_admissibility.3}, the discrete integra-tion-by-parts formula \eqref{eq:pi}, and the binomial formula, 
		for every $v_h\in K_h^{cr}$ and $y_h\in K_h^{rt,*}$, due to $ \pi_h v_h=\chi_h$ a.e.\ on $\Gamma_D$, 
		we find that
		\begin{align*}
			I_h^{cr}(v_h)-D_h^{rt}(y_h)&= \tfrac{1}{2}\| \nabla_h v_h\|_{\Omega}^2-(f_h,\Pi_h v_h)_{\Omega}-(g_h,\pi_h v_h)_{\Gamma_N}+\tfrac{1}{2}\| \Pi_h y_h\|_{\Omega}^2-(y_h\cdot n, \chi_h)_{\Gamma_D\cup\Gamma_C}\\&
			= \tfrac{1}{2}\| \nabla_h v_h\|_{\Omega}^2+(\textup{div}\,y_h,\Pi_hv_h)_{\Omega}+\tfrac{1}{2}\| \Pi_h y_h\|_{\Omega}^2\\&\quad-(y_h\cdot n,\pi_h v_h)_{\Gamma_N}-( y_h\cdot n, \chi_h)_{\Gamma_D\cup\Gamma_C}
			\\&
			= \tfrac{1}{2}\| \nabla_h v_h\|_{\Omega}^2-(\Pi_hy_h,\nabla_h v_h)_{\Omega}+\tfrac{1}{2}\| \Pi_h y_h\|_{\Omega}^2+(y_h\cdot n, \pi_h v_h-\chi_h)_{\Gamma_D\cup\Gamma_C}
			\\&
			= \tfrac{1}{2}\| \nabla_h v_h-y_h\|_{\Omega}^2+(y_h\cdot n, \pi_h v_h-\chi_h)_{\Gamma_C}\,.\tag*{$\qedsymbol$}
		\end{align*}
	\end{proof}
	
	Next, \hspace*{-0.1mm}we  \hspace*{-0.1mm}identify \hspace*{-0.1mm}the \hspace*{-0.1mm}\emph{optimal \hspace*{-0.1mm}strong \hspace*{-0.1mm}convexity \hspace*{-0.1mm}measures} \hspace*{-0.1mm}for \hspace*{-0.1mm}the \hspace*{-0.1mm}discrete \hspace*{-0.1mm}primal~\hspace*{-0.15mm}energy~\hspace*{-0.1mm}\mbox{functional} \eqref{eq:discrete_primal} \hspace{-0.175mm}at \hspace{-0.175mm}the \hspace{-0.175mm}discrete \hspace{-0.15mm}primal \hspace{-0.175mm}solution \hspace{-0.175mm}$u_h^{cr}\hspace*{-0.2em}\in\hspace*{-0.2em} K_h^{cr}$, \hspace{-0.175mm}\textit{i.e.}, 
	\hspace{-0.175mm}$\rho_{I_h^{cr}}^2\colon\hspace*{-0.2em} K_h^{cr}\hspace*{-0.2em}\to\hspace*{-0.2em} [0,+\infty)$,~\hspace{-0.175mm}for~\hspace{-0.175mm}every~\hspace{-0.175mm}${v_h\hspace*{-0.2em}\in\hspace*{-0.2em} K_h^{cr}}$~\hspace{-0.175mm}\mbox{defined}~\hspace{-0.175mm}by
	\begin{align}\label{def:discrete_optimal_primal_error}
		\smash{\rho_{I_h^{cr}}^2(v_h)\coloneqq I_h^{cr}(v_h)-I_h^{cr}(u_h^{cr})\,,}
	\end{align}
	and for the negative of the discrete dual energy functional \eqref{eq:discrete_dual}, \textit{i.e.}, $\rho_{\smash{-D_h^{rt}}}^2\colon\hspace*{-0.1em} K_h^{rt,*}\hspace*{-0.1em}\to\hspace*{-0.1em} [0,+\infty)$,  
	for every~$\smash{y_h\hspace*{-0.1em}\in\hspace*{-0.1em} K_h^{rt,*}}$ defined by
	\begin{align}\label{def:discrete_optimal_dual_error}
		\smash{\rho_{-D_h^{rt}}^2(y_h)\coloneqq- D_h^{rt}(y_h)+D_h^{rt}(z_h^{rt})\,,}
	\end{align}
	which \hspace*{-0.1em}will \hspace*{-0.1em}serve \hspace*{-0.1em}as \hspace*{-0.1em}\emph{`natural'} \hspace*{-0.1em}error \hspace*{-0.1em}quantities \hspace*{-0.1em}in \hspace*{-0.1em}the \hspace*{-0.1em}discrete \hspace*{-0.1em}primal-dual \hspace*{-0.1em}gap \hspace*{-0.1em}identity \hspace*{-0.1em}(\textit{cf}.\ \hspace*{-0.1em}Theorem~\hspace*{-0.1em}\ref{thm:discrete_prager_synge_identity}).
	
	\begin{lemma}[discrete optimal strong convexity measures]\label{lem:discrete_strong_convexity_measures}
		The following statements apply:
		\begin{itemize}[noitemsep,topsep=2pt,leftmargin=!,labelwidth=\widthof{(ii)}]
			\item[(i)] For every $v_h\in K_h^{cr}$, we have that\vspace*{-0.5mm}
			\begin{align*}
				\smash{\rho_{I_h^{cr}}^2(v_h)=\tfrac{1}{2}\|\nabla_h v_h-\nabla_h u_h^{cr}\|_{\Omega}^2+( z_h^{rt}\cdot n,\pi_hv_h-\chi_h)_{\Gamma_C}\,.}
			\end{align*}
			\item[(ii)] For every $y_h\in K_h^{rt,*}$, we have that\vspace*{-0.5mm}
			\begin{align*}
				\smash{\rho_{-D_h^{rt}}^2(y_h)=\tfrac{1}{2}\|\Pi_h y_h-\Pi_h z_h^{rt}\|_{\Omega}^2+(y_h\cdot n,\pi_h u_h^{cr}-\chi_h)_{\Gamma_C}\,.}
			\end{align*}
		\end{itemize}
		\end{lemma}
		
			\begin{remark}
			Note that for every $v_h\in K_h^{cr}$, it holds that $\pi_hv_h-\chi_h\ge 0$ a.e.\ on $\Gamma_C$, and for every $y_h\in K_h^{rt,*}$, we have that
			$y_h\cdot n\ge 0$ a.e.\ on $\Gamma_C$, so that for every $v_h\in K_h^{cr}$ and $y_h\in K_h^{rt,*}$, we have that\vspace*{-1mm}
			\begin{align*}
				y_h\cdot n\, (\pi_hv_h-\chi_h)\ge 0\quad\text{ a.e.\ on }\Gamma_C\,.
			\end{align*}
		\end{remark}

		\begin{proof}[Proof (of Lemma \ref{lem:discrete_strong_convexity_measures})]\let\qed\relax
		\emph{ad (i).} Using the binomial formula, the discrete admissibility~\mbox{conditions} \eqref{eq:discrete_admissibility.1},\eqref{eq:discrete_admissibility.2}, the discrete convex optimality relation~\eqref{eq:discrete_optimality.1}, the discrete integration-by-parts formula \eqref{eq:pi}, and the discrete convex optimality relation~\eqref{eq:discrete_optimality.2},  for every $v_h\in K_h^{cr}$, due to $\pi_hv_h= u_D^h=\pi_hu_D^{cr}$ a.e.\ on $\Gamma_D$, we find that 
		\begin{align*}
			I_h^{cr}(v_h)-I_h^{cr}(u_h^{cr})&=\tfrac{1}{2}\|\nabla_h v_h\|_{\Omega}^2-\tfrac{1}{2}\|\nabla_h u_h^{cr}\|_{\Omega}^2-(f_h,\Pi_h v_h-\Pi_h u_h^{cr})_{\Omega}-(g_h,\pi_h v_h-\pi_h u_h^{cr})_{\Gamma_N}\\&
			=\tfrac{1}{2}\|\nabla_h v_h-\nabla_h u_h^{cr}\|_{\Omega}^2+(\Pi_hz_h^{rt},\nabla_h v_h-\nabla_h u_h^{cr})_{\Omega}\\&\quad-(\textup{div}\,z_h^{rt},\Pi_h v_h-\Pi_h u_h^{cr})_{\Omega}-(z_h^{rt}\cdot n,\pi_h v_h-\pi_h u_h^{cr})_{\Gamma_N}
			%\\&
		%	=\tfrac{1}{2}\|\nabla_h v_h-\nabla_h u_h^{cr}\|_{\Omega}^2+(\Pi_hz_h^{rt},\nabla_h v_h-\nabla_h u_h^{cr})_{\Omega}\\&\quad+(\textup{div}\,z_h^{rt},\Pi_h v_h-\Pi_h u_h^{cr})_{\Omega}-(z_h^{rt}\cdot n,\pi_h v_h-\pi_h u_h^{cr})_{\Gamma_N}
			\\&
			=\tfrac{1}{2}\|\nabla_h v_h-\nabla_h u_h^{cr}\|_{\Omega}^2+(z_h^{rt}\cdot n,\pi_h v_h-\pi_h u_h^{cr})_{\Gamma_C}
			\\&
			=\tfrac{1}{2}\|\nabla_h v_h-\nabla_h u_h^{cr}\|_{\Omega}^2+(z_h^{rt}\cdot n,\pi_h v_h-\chi_h)_{\Gamma_C}\,.
		\end{align*}
		
		\emph{ad (ii).} Using the binomial formula, the admissibility conditions \eqref{eq:discrete_admissibility.1}--\eqref{eq:discrete_admissibility.3}, the convex optimality relation \eqref{eq:discrete_optimality.1},  the discrete integration-by-parts formula \eqref{eq:pi}, again, the admissibility condition \eqref{eq:discrete_admissibility.1}, and  the convex optimality relation \eqref{eq:discrete_optimality.2}, for every $y_h\in K_h^{rt,*}$, due to $y_h\cdot n= g_h = z_h^{rt}\cdot n$ a.e.\ on $\Gamma_N$ and $\pi_hu_h= u_D^h=\chi_h$ a.e.\ on $\Gamma_D$, we find that 
		\begin{align*}
			-D_h^{rt}(y_h)+D_h^{rt}(z_h^{rt})&=\tfrac{1}{2}\|\Pi_hy_h\|_{\Omega}^2-\tfrac{1}{2}\|\Pi_hz_h^{rt}\|_{\Omega}^2+(z_h^{rt}\cdot n-y_h\cdot n,\chi_h)_{\partial\Omega}%\\&
			%=(\Pi_h z_h^{rt},\Pi_h y_h-\Pi_hz_h^{rt})_{\Omega}+\tfrac{1}{2}\|\Pi_hy_h-\Pi_hz_h^{rt}\|_{\Omega}^2+( z_h^{rt}\cdot n-y_h\cdot n,\chi_h)_{\partial\Omega}
			\\&
			=\tfrac{1}{2}\|\Pi_h y_h-z_h^{rt}\|_{\Omega}^2+(\nabla_h u_h^{cr}, \Pi_h y_h-\Pi_h z_h^{rt})_{\Omega}+( z_h^{rt}\cdot n-y_h\cdot n,\chi_h)_{\partial\Omega}
			\\&
			=\tfrac{1}{2}\|\Pi_h y_h-\Pi_h z_h^{rt}\|_{\Omega}^2\\&\quad+(\textup{div}\,z_h^{rt}-\textup{div}\,y_h,\Pi_hu_h^{cr})_{\Omega}+(z_h^{rt}\cdot n-y_h\cdot n,\chi_h-\pi_hu_h^{cr})_{\partial\Omega}
			\\&
			=\tfrac{1}{2}\|\Pi_h y_h-\Pi_h z_h^{rt}\|_{\Omega}^2+(y_h\cdot n,\pi_hu_h^{cr}-\chi_h )_{\Gamma_C}\,.\tag*{$\qedsymbol$}
		\end{align*}
		\end{proof}

		Eventually, we have everything at our disposal to establish a discrete \emph{a posteriori}~error~identity that identifies the \emph{discrete primal-dual total error} $\rho_{\textup{tot},h}^2\colon K_h^{cr}\times K_h^{rt,*}\to [0,+\infty)$, for every $v_h\in K_h^{cr}$ and $y_h\in K_h^{rt,*}$ defined by 
		\begin{align}\label{eq:discrete_primal_dual_error}
			\smash{\rho_{\textup{tot},h}^2(v_h,y_h)\coloneqq \rho_{I_h^{cr}}^2(v_h)+\rho_{-D_h^{rt}}^2(y_h)\,,}
		\end{align}
		with the discrete primal-dual gap estimator $\eta_{\textup{gap},h}^2\colon K_h^{cr}\times K_h^{rt,*}\to [0,+\infty)$ (\textit{cf}.\ \eqref{eq:discrete_primal_dual_gap_estimator}).
		
	\begin{theorem}[discrete primal-dual gap identity]\label{thm:discrete_prager_synge_identity}
		For every $v_h\hspace*{-0.15em}\in\hspace*{-0.15em} K_h^{cr}$ and $y_h\hspace*{-0.15em}\in\hspace*{-0.15em} K_h^{rt,*}$,~we~have~that
		\begin{align*}
			\smash{\rho_{\textup{tot},h}^2(v_h,y_h)=\eta_{\textup{gap},h}^2(v_h,y_h)\,.}
		\end{align*}
	\end{theorem}
	
	\begin{proof}\let\qed\relax
		Using the definitions \eqref{eq:discrete_primal_dual_gap_estimator}, \eqref{def:discrete_optimal_primal_error},  \eqref{def:discrete_optimal_dual_error}, \eqref{eq:discrete_primal_dual_error}, and the discrete strong duality relation~\eqref{eq:discrete_strong_duality}, for every $v_h\in K_h^{cr}$ and $y_h\in K_h^{rt,*}$, we find that
		\begin{align*}
			\rho_{\textup{tot},h}^2(v_h,y_h)&=\rho_{I_h^{cr}}^2(v_h)+\rho_{-D_h^{rt}}^2(y_h)
			\\&=I_h^{cr}(v_h)-I_h^{cr}(u_h^{cr})+D_h^{rt}(z_h^{rt})-D_h^{rt}(y_h)
			%\\&=I_h^{cr}(v_h)+[D_h^{rt}(z_h^{rt})-I_h^{cr}(u_h^{cr})]-D_h^{rt}(y_h)
			\\&=I_h^{cr}(v_h)-D_h^{rt}(y_h)
			\\&=\eta_{\textup{gap},h}^2(v_h,y_h)\,.\tag*{$\qedsymbol$}
		\end{align*}
	\end{proof}
	
	Inserting \hspace*{-0.1mm}the \hspace*{-0.1mm}(Fortin) \hspace*{-0.1mm}quasi-interpolations \hspace*{-0.1mm}\eqref{def:CR-Interpolant} \hspace*{-0.1mm}and \hspace*{-0.1mm}\eqref{def:RT-Interpolant} \hspace*{-0.1mm}of \hspace*{-0.1mm}the \hspace*{-0.1mm}primal \hspace*{-0.1mm}and \hspace*{-0.1mm}the \hspace*{-0.1mm}dual \hspace*{-0.1mm}solution,~\hspace*{-0.1mm}respectively, \hspace*{-0.1mm}in \hspace*{-0.1mm}the \hspace*{-0.1mm}primal-dual \hspace*{-0.1mm}gap \hspace*{-0.1mm}identity \hspace*{-0.1mm}(\textit{cf}.\ \hspace*{-0.1mm}Theorem \hspace*{-0.1mm}\ref{thm:discrete_prager_synge_identity}), \hspace*{-0.1mm}we \hspace*{-0.1mm}arrive \hspace*{-0.1mm}at \hspace*{-0.1mm}an \hspace*{-0.1mm}\textit{a \hspace*{-0.1mm}priori}~\hspace*{-0.1mm}error~\hspace*{-0.1mm}\mbox{identity}, in which the right-hand side represents the residual in the discrete formulation.\enlargethispage{8mm}
	
	\begin{theorem}[\textit{a priori} error identity and error decay rates]\label{thm:apriori_identity}
		If $f_h\coloneqq \Pi_h f\in \mathcal{L}^0(\mathcal{T}_h)$, $g_h\coloneqq \pi_h g\in \mathcal{L}^0(\mathcal{S}_h^{\Gamma_N})$,~where~$g\in L^2(\Gamma_N)$, $u_D^h\coloneqq \pi_h u_D\in \mathcal{L}^0(\mathcal{S}_h^{\Gamma_D})$, and $\chi_h\coloneqq  \pi_h\Pi_h^{cr}\chi\in \mathcal{L}^0(\mathcal{S}_h^{\Gamma_C})$, then the following statements~\mbox{apply}:
		\begin{itemize}[noitemsep,topsep=2pt,leftmargin=!,labelwidth=\widthof{(iii)}]
			\item[(i)] If $z\in (L^p(\Omega))^d$, where $p>2$,  then $\Pi_h^{cr}u\in K_h^{cr}$, $\Pi_h^{rt}z\in K_h^{rt,*}$, and %the \textit{a priori}~error~identity 
			\begin{align*}
				\smash{\rho_{\textup{tot},h}^2(\Pi_h^{cr}u,\Pi_h^{rt}z)= \tfrac{1}{2} \|\Pi_h z-\Pi_h \Pi_h^{rt}z\|_{\Omega}^2+(z\cdot n,\pi_h\Pi_h^{cr}(u-\chi)-(u-\chi))_{\Gamma_C}\,.}
			\end{align*}
			\item[(ii)] If $u,\chi\in H^{1+s}(\Omega)$, where $s \in (\frac{1}{2},1]$, and if $\Delta\chi\in L^2(\Omega)$ or $d=2$, 
			then 
			\begin{align*}
				\smash{\rho_{\textup{tot},h}^2(\Pi_h^{cr}u,\Pi_h^{rt}z)\leq c\, h^{2s}\,(\|u\|_{1+s,\Omega}^2+\|\chi\|_{1+s,\Omega}^2)\,.}
			\end{align*} 
		\end{itemize}
	\end{theorem}

	\begin{remark}\label{rem:apriori_identity}
		The \textit{a priori} error estimate in Theorem \ref{thm:apriori_identity}(ii) holds as soon as $u\in H^1(\partial\Omega)$:\vspace*{-0.5mm}
		\begin{itemize}[noitemsep,topsep=2pt,leftmargin=!,labelwidth=\widthof{(ii)}]
			\item[(i)] According to \cite[Cor. 3.7]{BehrndtGesztesyMitrea22}, the trace operator\vspace*{-0.75mm}
			\begin{align*}
				\smash{\textup{Tr}\coloneqq (\textup{tr},\textup{tr}\circ\nabla)^\top 	\colon H_\Delta(\Omega)\coloneqq\big\{ v\in H^{\frac{3}{2}}(\Omega)\mid \Delta v\in L^2(\Omega)\big\}\to  H^1(\partial\Omega)\times (L^2(\partial\Omega))^d\,,}\\[-6mm]
			\end{align*}
			is well-defined, linear, and continuous. Thus, if $u\in \smash{H^{\frac{3}{2}}(\Omega)}$, due to ${\Delta u=\textup{div}\, z=-f\in L^2(\Omega)}$, we have that $u\in \smash{H_\Delta(\Omega)}$ and, thus, the traces $u|_{\partial \Omega}\in H^1(\partial\Omega)$ and $(\nabla u)|_{\partial\Omega}\in (L^2(\Omega))^d$.
			
			\item[(ii)] If \hspace*{-0.1mm}$\Omega\hspace*{-0.175em}\subseteq\hspace*{-0.175em}\mathbb{R}^2$ \hspace*{-0.1mm}is \hspace*{-0.1mm}polygonal \hspace*{-0.1mm}bounded \hspace*{-0.1mm}Lipschitz \hspace*{-0.1mm}domain~\hspace*{-0.1mm}and~\hspace*{-0.1mm}$s\in (\frac{1}{2},1]$, \hspace*{-0.1mm}then (\textit{cf}.\ \hspace*{-0.1mm}\cite[Thm.\ \hspace*{-0.1mm}1.5.2.1]{Grisvard85}) \hspace*{-0.1mm}$\textup{Tr}\hspace*{-0.175em}\coloneqq\hspace*{-0.175em} (\textup{tr},\textup{tr}\hspace*{-0.1em}\circ\hspace*{-0.1em}\nabla)^\top 	\hspace*{-0.2em}\colon\hspace*{-0.175em} H^{1+s}(\Omega)\hspace*{-0.175em}\to \hspace*{-0.175em} H^{s+\frac{1}{2}}(\partial\Omega)\hspace*{-0.175em}\times\hspace*{-0.175em} (H^{s-\frac{1}{2}}(\partial\Omega))^d$ \hspace*{-0.1mm}is \hspace*{-0.1mm}well-defined, \hspace*{-0.1mm}linear \hspace*{-0.1mm}and \hspace*{-0.1mm}continuous.\vspace*{-0.25mm}
		\end{itemize} 
	\end{remark}
	
	\begin{proof}[Proof (of Theorem \ref{thm:apriori_identity})]
		\emph{ad (i).} First, using \eqref{eq:trace_preservation}, we observe that\vspace*{-1mm}
		\begin{align*}
			\pi_h \Pi_h^{cr}u=\pi_h u\begin{cases}
				=\pi_hu_D=u_D^h&\text{ a.e.\ on }\Gamma_D\,,\\
				\ge \pi_h \chi= \chi_h &\text{ a.e.\ on }\Gamma_C\,,
			\end{cases}\\[-6.5mm]\notag
		\end{align*}
		\textit{i.e.}, it holds that $\Pi_h^{cr}u\in K_h^{cr}$.
		Second, if $z\in (L^p(\Omega))^d$, where $p>2$, according to \cite[Sec.\ 17.1]{EG21I}, $\Pi_h^{rt}z\in \mathcal{R}T^0_N(\mathcal{T}_h)$ is well-defined and
		 using \eqref{eq:div_preservation} and \eqref{eq:normal_trace_preservation}, we observe that\vspace*{-0.75mm}
		\begin{align*}
			\begin{aligned}  
			\textup{div}\,\Pi_h^{rt}z&=\Pi_h \textup{div}\,z=-f_h\qquad\qquad\hspace*{0.5mm}\text{ a.e.\ in }\Omega\,, \\[-1mm]
			\Pi_h^{rt} z\cdot n&=\pi_h(z\cdot n) \begin{cases}
				=\pi_h g=g_h&\text{ a.e.\ on }\Gamma_N\,,\\
				\ge 0 &\text{ a.e.\ on }\Gamma_C\,,
			\end{cases} 
		\end{aligned}\\[-6.5mm]\notag
		\end{align*}
		\textit{i.e.}, it holds that $\Pi_h^{rt}z\in K_h^{rt,*}$. Then, 
		using Theorem \ref{thm:discrete_prager_synge_identity} together with Lemma \ref{lem:discrete_primal_dual_gap_estimator}~and~Lemma~\ref{lem:discrete_strong_convexity_measures} as well as the convex optimality relation \eqref{eq:optimality.2}, \eqref{eq:grad_preservation}, and \eqref{eq:normal_trace_preservation}, we find that\vspace*{-0.5mm}
		\begin{align}\label{thm:apriori_identity.1}
			\begin{aligned} 
			\smash{\rho_{\textup{tot},h}^2}(\Pi_h^{cr}u,\Pi_h^{rt}z)&= \tfrac{1}{2} \|\nabla_h \Pi_h^{cr}u-\Pi_h \Pi_h^{rt}z\|_{\Omega}^2+(\Pi_h^{rt}z\cdot n,\pi_h\Pi_h^{cr}u-\chi_h)_{\Gamma_C}
			\\&= \tfrac{1}{2} \|\Pi_h z-\Pi_h \Pi_h^{rt}z\|_{\Omega}^2+(z\cdot n,\pi_h\Pi_h^{cr}(u-\chi)-(u-\chi))_{\Gamma_C}
			\\&\eqqcolon \smash{I_h^1}+\smash{I_h^2}\,.
			\end{aligned}\\[-6.5mm]\notag
		\end{align} 
		%which is the claimed \textit{a priori} error identity.
		
		\emph{ad (ii).} We need to estimate $\smash{I_h^1}$ and $\smash{I_h^2}$:\enlargethispage{12.5mm}
		
		\emph{ad $I_h^1$.} Using the $L^2$-stability property of $\Pi_h$ and the fractional approximation properties~of~$\Pi_h^{rt}$ (\emph{cf}.\ \eqref{eq:RT-Interpolant-Rate}), we find that\vspace*{-0.5mm}
		\begin{align}\label{thm:apriori_identity.2}
			\begin{aligned} 
			I_h^1\leq \tfrac{1}{2} \|z-\Pi_h^{rt}z\|_{\Omega}^2\leq c\, h^{2s }\,\vert z\vert_{s,\Omega}^2\leq c\, h^{2s }\,\|u\|_{1+s,\Omega}^2\,.
			\end{aligned}\\[-6.5mm]\notag
		\end{align} 
		
		\emph{ad $I_h^2$.}
		Abbreviating $\tilde{u}\coloneqq u-\chi\in \smash{H^{1+s}(\Omega)}$, % and using the convex optimality relation \eqref{eq:optimality.2}, 
		we find that\vspace*{-0.5mm}
		\begin{align}\label{thm:apriori_identity.3}
			\begin{aligned} 
				(z\cdot n,\pi_h\Pi_h^{cr}\tilde{u}-\tilde{u})_{\Gamma_C} 
				&=(z\cdot n,\Pi_h^{cr}\tilde{u}-\tilde{u})_{\Gamma_C}+(z\cdot n,\pi_h\Pi_h^{cr}\tilde{u}-\Pi_h^{cr}\tilde{u})_{\Gamma_C}
				\\&\eqqcolon \smash{I_h^{2,1}}+\smash{I_h^{2,2}}\,. 
			\end{aligned}\\[-6.5mm]\notag
		\end{align} 
 		
 		\emph{ad $I_h^{2,1}$.} Using that $\Pi_h^{cr}\tilde{u}-\tilde{u}\perp \mathcal{L}^0(\mathcal{S}_h)$ (\textit{cf}.\ \eqref{eq:trace_preservation}), that $\pi_h(z\cdot n)=(\pi_hz)\cdot n$ a.e.\ in $\cup\mathcal{S}_h$, a local trace inequality  (\textit{cf}.\ \cite[Rem.\ 12.19,~(12.17)]{EG21I}), and the fractional approximation properties~of~$\pi_h $ %(\emph{cf}.\ \cite[Rem.\ 18.17]{EG21I})~
 		and~$\Pi_h^{cr}$~(\emph{cf}.~\eqref{eq:CR-Interpolant-Rate}), we obtain\vspace*{-0.5mm} 
 		\begin{align}\label{thm:apriori_identity.4}
 			\begin{aligned} 
 			\smash{I_h^{2,1}}&=(z\cdot n-\pi_h(z\cdot n),\Pi_h^{cr}\tilde{u}-\tilde{u})_{\Gamma_C}
 			%\\&\leq \|(z- \pi_hz)\cdot n\|_{\Gamma_C}\|\tilde{u}-\Pi_h^{cr}\tilde{u}\|_{\Gamma_C}
 			%\\&\leq \|z-\pi_hz\|_{\Gamma_C}\|\tilde{u}-\Pi_h^{cr}\tilde{u}\|_{\Gamma_C}
 			\\&\leq \|z-\pi_hz\|_{\Gamma_C}\,(h^{\smash{-\frac{1}{2}}}\,\|\tilde{u}-\Pi_h^{cr}\tilde{u}\|_{\Omega}+h^{\smash{\frac{1}{2}}}\,\|\nabla\tilde{u}-\nabla_h\Pi_h^{cr}\tilde{u}\|_{\Omega})
 			\\&\leq c\, h^{\smash{s -\frac{1}{2}}}\vert z\vert_{\smash{s -\frac{1}{2}},\Gamma_C}\,h^{\smash{s +\frac{1}{2}}}\vert  \tilde{u}\vert_{\smash{s +\frac{1}{2}},\Omega}
 			\\&\leq  c\, h^{\smash{2s}} (\|u\|_{1+s,\Omega}^2+\| \chi\|_{1+s,\Omega}^2)\,.
 			\end{aligned}\\[-6.5mm]\notag
 		\end{align} 
 		
 		\emph{ad $I_h^{2,2}$.} We decompose $I_h^{2,2}$ into local contributions, \textit{i.e.}, we define\vspace*{-1mm}
 		\begin{align}\label{thm:apriori_identity.5}
 		%\begin{aligned} 
 			I_h^{2,2}%&
 			=\sum_{S\in \mathcal{S}_h^{\smash{\Gamma_C}}}{(z\cdot n,\pi_h\Pi_h^{cr}\tilde{u}-\Pi_h^{cr}\tilde{u})_S}
 			%\\&	
 			\eqqcolon \sum_{S\in \mathcal{S}_h^{\smash{\Gamma_C}}}{I_S^{2,2}}\,.\\[-7mm]\notag
 			%\end{aligned}
 	\end{align} 
 		Next, we distinguish the cases $\vert S\setminus (\{\tilde{u}>0\}\cap \Gamma_C)\vert=0 $ (\textit{i.e.}, no contact), $\vert S\setminus  (\{\tilde{u}=0\}\cap \Gamma_C)\vert=0$ (\textit{i.e.}, contact), and $\vert S\setminus (\{\tilde{u}=0\}\cap \Gamma_C)\vert>0 $ (\textit{i.e.}, both contact and no contact) (equivalent to $\vert S\setminus(\{\tilde{u}>0\}\cap \Gamma_C)\vert>0$):
 		In doing so, we~use~the~identity\vspace*{-0.5mm}
 		\begin{align}\label{representation}
 			\Pi_h^{cr}\tilde{u}-\pi_h\Pi_h^{cr}\tilde{u}=\nabla_{\!S}\Pi_h^{cr}\tilde{u}\cdot (\textup{id}_{\mathbb{R}^d}-\pi_h \textup{id}_{\mathbb{R}^d})\quad\text{ in }S\,,\\[-6mm]\notag
 		\end{align}
 		where, for each $v\in H^1(S)$, 
 		we denote by   $\nabla_{\!S}v\in (L^2(S))^d$ the tangential gradient, 
 		%, \textit{i.e.}, the gradient of the trace $v|_S\in H^1(S)$ (\textit{cf}.\ Remark \ref{rem:apriori_identity}), 
 		which,~for~every $v\hspace*{-0.15em}\in \hspace*{-0.15em} H_{\Delta}(T_S)\cap H^{1+s}(T_S)$, where $T_S\hspace*{-0.15em}\in\hspace*{-0.15em}\mathcal{T}_h$ is such that $S\hspace*{-0.15em}\subseteq\hspace*{-0.15em} \partial T_S$, satisfies~\mbox{(\textit{cf}.\ \cite[Rem.\ 12.19,~(12.17)]{EG21I})}\vspace*{-0.5mm}
 		\begin{align}\label{eq:grad_trace}
 			\|\nabla_{\!S} v\|_S\leq \|(\nabla v)|_S\|_S\leq c\,(\smash{h_S^{-\frac{1}{2}}}\|\nabla v\|_{T_S}+h_S^{\smash{s-\frac{1}{2}}}\vert\nabla v\vert_{s,T_S})\,.
 		\end{align}
 		
 		In particular, from \eqref{eq:grad_trace} together with the fractional approximation properties of $\Pi_h^{cr}$~(\emph{cf}.~\eqref{eq:CR-Interpolant-Rate}) and $\vert \nabla\Pi_h^{cr}\tilde{u}\vert_{s,T_S}=0$ since $\nabla\Pi_h^{cr}\tilde{u}|_{T_S}=\textup{const}$, it follows that
 		\begin{align}\label{eq:grad_trace_interpol}
 			\begin{aligned}
 					\|\nabla_{\!S} \tilde{u}-\nabla_{\!S} \Pi_h^{cr}\tilde{u}\|_S&\leq (\smash{h_S^{-\frac{1}{2}}}\| \nabla\tilde{u}- \nabla\Pi_h^{cr}\tilde{u}\|_{T_S}+h_S^{\smash{s-\frac{1}{2}}}\vert \nabla \tilde{u}\vert_{s,T_S})
 					\\&\leq h_S^{\smash{s-\frac{1}{2}}}\vert  \tilde{u}\vert_{\smash{s +\frac{1}{2}},T_S}\,.
 			\end{aligned}
 		\end{align}

 		\emph{ad $\vert S\setminus (\{\tilde{u}>0\}\cap \Gamma_C)\vert=0$ (\textit{i.e.}, no contact).} In this case, due to the convex optimality~relation \eqref{eq:optimality.2}, we have that $z\cdot n=0$ a.e.\ on $S$ and, thus,
 		\begin{align}\label{thm:apriori_identity.6.0}
 			\smash{I_S^{2,2}=0}\,.
 		\end{align}
 		
 		 \emph{ad $\vert S\setminus  (\{\tilde{u}=0\}\cap \Gamma_C)\vert=0$ (\textit{i.e.}, contact).} In this case, we~have~that $\tilde{u}=0$ a.e.\ on $S$, which~implies~that
 		  $\nabla_{\!S} \tilde{u}=0$ a.e.\ on $S$. Therefore,  using that $\pi_h\Pi_h^{cr}\tilde{u}-\Pi_h^{cr}\tilde{u}\perp \mathcal{L}^0(\mathcal{S}_h)$, \eqref{representation}, the fractional approximation properties of $\pi_h $ (\emph{cf}.\ \cite[Rem.\ 18.17]{EG21I}), and \eqref{eq:grad_trace_interpol}, we obtain
 		 	\begin{align}\label{thm:apriori_identity.6}
 		 	\begin{aligned} 
 		 	\smash{I_S^{2,2}}&=(z\cdot n-\pi_h(z\cdot n),\pi_h\Pi_h^{cr}\tilde{u}-\Pi_h^{cr}\tilde{u})_S
 		 	\\&=(z\cdot n-\pi_h(z\cdot n),(\nabla_{\!S} \tilde{u}-\nabla_{\!S} \Pi_h^{cr}\tilde{u})\cdot(\textup{id}_{\mathbb{R}^d}-\pi_h\textup{id}_{\mathbb{R}^d}))_S
 		 	\\&\leq h_S\,\|z-\pi_hz\|_S\|\nabla_{\!S} \tilde{u}-\nabla_{\!S} \Pi_h^{cr}\tilde{u}\|_S 
 		 	\\&\leq c\, h_S\,h_S^{\smash{s -\frac{1}{2}}}\vert z\vert_{\smash{s -\frac{1}{2}},S}\,h_S^{\smash{s-\frac{1}{2}}}\vert  \tilde{u}\vert_{\smash{s +\frac{1}{2}},T_S}
 		 	\\&\leq  c\, h_S^{\smash{2s}}\, (\|u\|_{1+s,T_S}^2+\| \chi\|_{1+s,T_S}^2)\,.
 		 		\end{aligned}
 	 \end{align} 
 		 
 		 \emph{ad $\vert S\setminus (\{\tilde{u}\hspace*{-0.1em}=\hspace*{-0.1em}0\}\cap \Gamma_C)\vert\hspace*{-0.1em}>\hspace*{-0.1em}0$ (\textit{i.e.}, both contact and no contact).}
 		 On the one hand, we have that $\tilde{u}=0$ a.e.\ on $S\cap \{\tilde{u}=0\}$, which implies that $\nabla_{\!S} \tilde{u}=0$ a.e.\ on $S\cap \{\tilde{u}=0\}$. Using that $\pi_h\Pi_h^{cr}\tilde{u}-\Pi_h^{cr}\tilde{u}\perp \mathcal{L}^0(\mathcal{S}_h)$, \eqref{representation}, \cite[Lem.~8.2.3]{DHHR11}, the fractional approximation properties of $\pi_h $, that $ \pi_h \nabla_{\!S}\Pi_h^{cr} \tilde{u}= \nabla_{\!S}\Pi_h^{cr} \tilde{u}$, the $L^2$-stability of $\pi_h $, and \eqref{eq:grad_trace_interpol}, we obtain\enlargethispage{10mm}
 		 	\begin{align}\label{thm:apriori_identity.7}
 		 	\begin{aligned} 
 		 I_S^{2,2}&=(z\cdot n-\pi_h(z\cdot n),(\langle\nabla_{\!S} \tilde{u}\rangle_{S\cap \{\tilde{u}=0\}}
 		 %\nabla_{\!S} \tilde{u}
 		 -\nabla_{\!S} \Pi_h^{cr}\tilde{u})\cdot(\textup{id}_{\mathbb{R}^d}-\pi_h\textup{id}_{\mathbb{R}^d}))_S
 		% \\&\quad+ (z\cdot n-\pi_h(z\cdot n),(\langle\nabla_{\!S} \tilde{u}\rangle_{S\cap \{\tilde{u}=0\}}-\nabla_{\!S} \tilde{u})\cdot(\textup{id}_{\mathbb{R}^d}-\pi_h\textup{id}_{\mathbb{R}^d}))_S
 		 \\&\leq \|z-\pi_hz\|_S\,h_S\,(\|\nabla_{\!S} \tilde{u}-\nabla_{\!S} \Pi_h^{cr}\tilde{u}\|_S+\|\langle\nabla_{\!S} \tilde{u}\rangle_{S\cap \{\tilde{u}=0\}}-\nabla_{\!S} \tilde{u}\|_S)
 		 \\&\leq c\,\|z-\pi_hz\|_S\,h_S\,\big(\|\nabla_{\!S} \tilde{u}-\nabla_{\!S} \Pi_h^{cr}\tilde{u}\|_S+\smash{\tfrac{\vert S\vert}{\vert S\cap \{\tilde{u}=0\} \vert}}\|\pi_h(\nabla_{\!S} \tilde{u})-\nabla_{\!S} \tilde{u}\|_S\big)
 		 \\&\leq c\,\|z-\pi_hz\|_S\,h_S\,\big(\|\nabla_{\!S} \tilde{u}-\nabla_{\!S} \Pi_h^{cr}\tilde{u}\|_S\\&\quad+\smash{\tfrac{\vert S\vert}{\vert S\cap \{\tilde{u}=0\} \vert}}(\|\pi_h(\nabla_{\!S} \tilde{u}-\nabla_{\!S}\Pi_h^{cr} \tilde{u})\|_S+\|\nabla_{\!S} \tilde{u}-\nabla_{\!S}\Pi_h^{cr} \tilde{u}\|_S)\big)
 		% \\&\leq c\, h_S^{\smash{s -\frac{1}{2}}}\vert z\vert_{\smash{s -\frac{1}{2}},S}\,h_S\,\big(h_S^{\smash{s -\frac{1}{2}}}\vert  \tilde{u}\vert_{\smash{s +\frac{1}{2}},T_S}+\tfrac{\vert S\vert}{\vert S\cap \{\tilde{u}=0\} \vert}h_S^{\smash{s -\frac{1}{2}}}\vert  \tilde{u}\vert_{\smash{s +\frac{1}{2}},T_S}\big)
 		 \\&\leq  c\,\tfrac{\vert S\vert }{\vert S\cap \{\tilde{u}=0\}\vert }\,\smash{h_S^{\smash{2s}}}\,(\|u\|_{1+s,T_S}^2+\|\chi\|_{1+s,T_S}^2)\,.
 			\end{aligned}
 	\end{align} 
 		 On the other hand, we have that $z\cdot n=0$ a.e.\ in $S\cap \{\tilde{u}>0\}$. Using that ${\pi_h\Pi_h^{cr}\tilde{u}-\Pi_h^{cr}\tilde{u}\perp \mathcal{L}^0(\mathcal{S}_h)}$, that $\langle z\cdot n\rangle_{S\cap \{\tilde{u}>0\}}=\langle z\rangle_{S\cap \{\tilde{u}>0\}}\cdot n$ a.e.\ in $S$,  \eqref{representation}, \cite[Lem.\ 8.2.3]{DHHR11},  the fractional approximation properties of $\pi_h $, and \eqref{eq:grad_trace_interpol}, we~obtain\vspace*{-0.5mm}
 		 	\begin{align}\label{thm:apriori_identity.8}
 		 	\begin{aligned} 
 		 	I_S^{2,2}&=(z\cdot n,(\nabla_{\!S} \tilde{u}-\nabla_{\!S} \Pi_h^{cr}\tilde{u})\cdot(\textup{id}_{\mathbb{R}^d}-\pi_h\textup{id}_{\mathbb{R}^d}))_S
 		 	\\&=(z\cdot n-\langle z\cdot n\rangle_{S\cap \{\tilde{u}>0\}},(\nabla_{\!S} \tilde{u}-\nabla_{\!S} \Pi_h^{cr}\tilde{u})\cdot(\textup{id}_{\mathbb{R}^d}-\pi_h\textup{id}_{\mathbb{R}^d}))_S 
 		 	\\&\leq  \|z-\langle z\rangle_{S\cap \{\tilde{u}>0\}}\|_S\,h_S\,\|\nabla_{\!S} \tilde{u}-\nabla_{\!S} \Pi_h^{cr}\tilde{u}\|_S
 		 	\\&\leq  c\,\smash{\tfrac{\vert S\vert }{\vert S\cap \{\tilde{u}>0\}\vert }}\, \|z-\pi_hz\|_S\,h_S\,\|\nabla_{\!S} \tilde{u}-\nabla_{\!S} \Pi_h^{cr}\tilde{u}\|_S
 		 %	\\&\leq  c\,\tfrac{\vert S\vert }{\vert S\cap \{\tilde{u}>0\}\vert }\,h_S^{\smash{s -\frac{1}{2}}}\vert z\vert_{\smash{s -\frac{1}{2}},S}\,h_S\,h_S^{\smash{s -\frac{1}{2}}}\vert  \tilde{u}\vert_{\smash{s +\frac{1}{2}},T_S}
 		 		\\&\leq  c\,\smash{\tfrac{\vert S\vert }{\vert S\cap \{\tilde{u}>0\}\vert }}\,\smash{h_S^{\smash{2s}}}\,(\|u\|_{1+s,T_S}^2+\|\chi\|_{1+s,T_S}^2)\,.
 			\end{aligned}
 	\end{align} 
 		Since $S=(S\cap \{\tilde{u}>0\})\dot{\cup}(S\cap \{\tilde{u}=0\})$, we have that $\vert S\cap \{\tilde{u}>0\}\vert \ge \frac{1}{2}\vert S\vert$~or~${\vert S\cap \{\tilde{u}=0\}\vert \ge \frac{1}{2}\vert S\vert}$.
 		Using in the first case \eqref{thm:apriori_identity.7} and in the second case \eqref{thm:apriori_identity.8}, we arrive at 
 		 \begin{align}\label{thm:apriori_identity.9}
 		 	\smash{I_S^{2,2} \leq c\, h_S^{\smash{2s}}\,(\|u\|_{1+s,T_S}^2+\|\chi\|_{1+s,T_S}^2)}\,.
 		 \end{align}
 		 Combining \eqref{thm:apriori_identity.6.0}, \eqref{thm:apriori_identity.6}, and \eqref{thm:apriori_identity.9} in \eqref{thm:apriori_identity.5}, we deduce that
 		 \begin{align}\label{thm:apriori_identity.10}
 		 	\smash{I_h^{2,2}\leq c\, h^{\smash{2s}}\,(\|u\|_{1+s,\Omega}^2+\|\chi\|_{1+s,\Omega}^2)}\,.
 		 \end{align}
 		 Using \eqref{thm:apriori_identity.4} and \eqref{thm:apriori_identity.10} in \eqref{thm:apriori_identity.3}, we infer that
 		 \begin{align}\label{thm:apriori_identity.11}
 		 	\smash{I_h^2\leq c\, h^{\smash{2s}}\,(\|u\|_{1+s,\Omega}^2+\|\chi\|_{1+s,\Omega}^2)}\,.
 		 \end{align}
 		 Eventually, using \eqref{thm:apriori_identity.11} and \eqref{thm:apriori_identity.2} in \eqref{thm:apriori_identity.1}, we conclude the claimed \textit{a priori} error estimate.
	\end{proof}\newpage

	\section{Numerical experiments}\label{sec:experiments}\enlargethispage{5mm}
	
	\qquad In this section, we review the theoretical findings of Section \ref{sec:aposteriori} and Section \ref{sec:apriori}  via numerical experiments.
	All experiments were carried out using the finite element software package~\mbox{\textsf{FEniCS}} (version 2019.1.0, \emph{cf}.\ \cite{fenics}). 
	All graphics are created using the \textsf{Matplotlib}  \mbox{library}~(version~3.5.1,~\emph{cf}.~\cite{Hunter07}).
	
	\subsection{Implementation details} 
	
	\hspace{5mm}We compute the discrete primal solution ${u_h^{cr}\in \smash{\mathcal{S}^{1,\emph{cr}}_D(\mathcal{T}_h)}}$ and the associated discrete Lagrange multiplier $\smash{\overline{\lambda}}_h^{cr}\in \mathcal{L}^0(\mathcal{S}_h^{\Gamma_C}) $ jointly satisfying the discrete augmented problem~\eqref{eq:discrete_augmented_problem}~via 
	the primal-dual active set strategy interpreted as a semi-smooth Newton method.~For~sake~of~completeness, in the case $u_D\hspace*{-0.1em}=\hspace*{-0.1em}0$, we will~briefly~outline~important implementation details related~with~this~strategy. 
	
	We fix an ordering of the sides ${(S_i)_{i=1,\dots, N_h^{cr}}}$ and an ordering of the~\hspace{-0.15mm}elements~\hspace{-0.15mm}${(T_i)_{i=1,\dots, N_h^0}}$, where $N_h^{cr}\coloneqq \textup{card}(\mathcal{S}_h\setminus\mathcal{S}_h^{\Gamma_D})$, $N_h^{\smash{cr,C}}=\textrm{card}(\mathcal{S}_h^{\Gamma_C})$ and $N_h^0\coloneqq \textup{card}(\mathcal{T}_h)$ such that\footnote{In practice, the element $\widehat{T}\in \mathcal{T}_h$ for which $\mathbb{R}\chi_{\widehat{T}}\perp\smash{\Pi_h(\mathcal{S}^{1,\emph{cr}}_D(\mathcal{T}_h))}$ is found via searching and erasing~a~zero~column (if existent) in the matrix $((\Pi_ h\varphi_{S_i},\chi_T)_{\Omega})_{i=1,\dots,N_h^{cr},T\in \mathcal{T}_h}\in \smash{\mathbb{R}^{N_h^{cr}\times N_h^0}}$ leading to $\mathrm{P}_h^{cr,0}\in \smash{\mathbb{R}^{N_h^{cr}\times N_h^{\smash{cr,0}}}}$.} 
	\begin{align*}
		\textup{span}(\{\varphi_{S_i}\mid i= 1,\dots, N_h^{\smash{cr}}\})&=\mathcal{S}^{1,cr}_D(\mathcal{T}_h)\,,\\
			\textup{span}(\{\chi_{S_i}\mid i= 1,\dots, N_h^{\smash{cr,C}}\})&=\mathcal{L}^0(\mathcal{S}_h^{\Gamma_C})\,,\\
		\textup{span}(\{\chi_{T_i}\mid i= 1,\dots, N_h^{\smash{cr,0}}\})&=\Pi_h(\mathcal{S}^{1,cr}_D(\mathcal{T}_h))\,,
	\end{align*}
	where $N_h^{\smash{cr,0}}=\textup{dim}(\Pi_h(\mathcal{S}^{1,cr}_D(\mathcal{T}_h)))\in \{N_h^0,N_h^0-1\}$ because of ${\textup{codim}_{\mathcal{L}^0(\mathcal{T}_h)}(\Pi_h(\mathcal{S}^{1,cr}_D(\mathcal{T}_h)))\in\{0,1\}}$ (\emph{cf}.\ \cite[Cor.\ 3.2]{BW21}).
	Then, if we define the matrices
	\begin{align*}
		\mathrm{S}_h^{cr}&\coloneqq((\nabla_h\varphi_{S_i},\nabla_h\varphi_{S_j})_{\Omega})_{i,j=1,\dots,N_h^{cr}}\in \mathbb{R}^{N_h^{cr}\times N_h^{cr}}\,,\\
		\mathrm{P}_h^{cr,0}&\coloneqq((\Pi_ h\varphi_{S_i},\chi_{T_j})_{\Omega})_{i=1,\dots,N_h^{cr},j=1,\dots,N_h^{\smash{cr,0}}}\in \mathbb{R}^{N_h^{cr}\times N_h^{\smash{cr,0}}}\,,\\
		\mathrm{p}_h^{cr,C}&\coloneqq((\Pi_ h\varphi_{S_i},\chi_{S_j})_{\Gamma_C})_{i=1,\dots,N_h^{cr},j=1,\dots,N_h^{\smash{cr,C}}}\in \mathbb{R}^{N_h^{cr}\times N_h^{\smash{cr,C}}}\,,
		\intertext{and, assuming for the entire section that $\chi_h\coloneqq \pi_h\Pi_h^{cr} \chi\in \mathcal{L}^0(\mathcal{S}_h)$, the vectors}
		\mathrm{X}_h^{cr}&\coloneqq ((\chi_h,\chi_{S_i})_{\Gamma_C})_{i=1,\dots,N_h^{cr}}\in \mathbb{R}^{ N_h^{\smash{cr}}}\,,\\
		\mathrm{F}_h^0&\coloneqq((f_h,\chi_{T_i})_{\Omega})_{i=1,\dots,N_h^{\smash{cr,0}}}\in \mathbb{R}^{ N_h^{\smash{cr,0}}}\,,\\
		\mathrm{G}_h^{cr}&\coloneqq((g_h,\chi_{S_i})_{\Gamma_N})_{i=1,\dots,N_h^{\smash{cr}}}\in \mathbb{R}^{ N_h^{\smash{cr}}}\,,
	\end{align*}
	the same argumentation as in \cite[Lem.\ 5.3]{Bartels15} shows that  the discrete augmented problem \eqref{eq:discrete_augmented_problem} is equivalent to finding vectors $(\mathrm{U}_h^{cr},\smash{\overline{\Lambda}}_h^{cr})^\top \in \mathbb{R}^{N_h^{cr}}\times \mathbb{R}^{N_h^{\smash{cr,C}}}$ such that
	\begin{align}
		\begin{aligned}\mathrm{S}_h^{cr}\mathrm{U}_h^{cr}+\mathrm{p}_h^{cr,C}\smash{\overline{\Lambda}}_h^{cr}&=\mathrm{P}_h^{cr,0}\mathrm{F}_h^0+\mathrm{G}_h^{cr}&&\quad \text{ in }\mathbb{R}^{N_h^{cr}}\,,\\
			\mathscr{C}_h(\mathrm{U}_h^{cr},\smash{\overline{\Lambda}}_h^{cr})&=0_{\smash{\tiny\mathbb{R}^{N_h^{\smash{cr,C}}}}}&&\quad \text{ in }\mathbb{R}^{N_h^{\smash{cr,C}}}\,,
		\end{aligned}\label{eq:details.1}
	\end{align}
	where for given $\alpha>0$, the mapping $\mathscr{C}_h\colon \mathbb{R}^{N_h^{cr}}\!\times \mathbb{R}^{N_h^{\smash{cr,C}}}\!\!\to\!\mathbb{R}^{N_h^{\smash{cr,C}}}$ for every ${(\mathrm{U}_h,\smash{\overline{\Lambda}}_h)^\top\!\!\in\! \mathbb{R}^{N_h^{cr}}\!\!\times\! \mathbb{R}^{N_h^{\smash{cr,C}}}}$ is defined by\footnote{Here, for   $a=(a_i)_{i=1,\ldots,n},b=(b_i)_{i=1,\ldots,n }\in \mathbb{R}^n$, $n\in \mathbb{N}$, we define $\min\{a,b\}=(\min\{a_i,b_i\})_{i=1,\ldots,n}\in \mathbb{R}^n$.}
	\begin{align*}
		\mathscr{C}_h(\mathrm{U}_h,\smash{\overline{\Lambda}}_h)\coloneqq \smash{\overline{\Lambda}}_h-\min\big\{0_{\smash{\tiny\mathbb{R}^{N_h^{\smash{cr,C}}}}},\smash{\overline{\Lambda}}_h+\alpha\,(\mathrm{p}_h^{cr,0})^{\top}(\mathrm{U}_h-\mathrm{X}_h^{cr})\big\}\quad\text{ in }\mathbb{R}^{N_h^{\smash{cr,C}}}\,.
	\end{align*}
	More precisely, the discrete primal solution $u_h^{cr}\in \smash{\mathcal{S}^{1,\emph{cr}}_D(\mathcal{T}_h)}$ and the associated discrete Lagrange multiplier ${\smash{\overline{\lambda}}_h^{cr}\in \mathcal{L}^0(\mathcal{S}_h^{\Gamma_C})}$ jointly satisfying the discrete augmented problem \eqref{eq:discrete_augmented_problem}~as~well~as the vectors  $(\mathrm{U}_h^{cr},\smash{\overline{\Lambda}}_h^{cr})^\top \in \mathbb{R}^{N_h^{cr}}\times \mathbb{R}^{N_h^{\smash{cr,C}}}$ satisfying \eqref{eq:details.1}, respectively, are related by\footnote{Here, for each $i=1,\dots,N$, $N\in \mathbb{N}$, we denote by $\mathrm{e}_i=(\delta_{ij})_{j=1,\dots,N}\in  \mathbb{R}^N$,~the~\mbox{$i$-th}~unit~\mbox{vector}.}
	\begin{align*}
		\begin{aligned}
			u_h^{cr}%&
			=\sum_{i=1}^{N_h^{cr}}{(\mathrm{U}_h^{cr}\cdot \mathrm{e}_i)\varphi_{S_i}}\in \mathcal{S}^{1,\emph{cr}}_D(\mathcal{T}_h)\,,\qquad
			\smash{\overline{\lambda}}_h^{cr}%&
			=\sum_{i=1}^{N_h^{cr,C}}{(\smash{\overline{\Lambda}}_h^{cr}\cdot \mathrm{e}_i)\chi_{S_i}}\in \mathcal{L}^0(\mathcal{S}_h^{\Gamma_C})\,.
		\end{aligned}
	\end{align*}
	Next, define the mapping $\mathscr{F}_h\colon \!\mathbb{R}^{N_h^{cr}}\!\times \mathbb{R}^{N_h^{\smash{cr,C}}}\!\!\to \!\mathbb{R}^{N_h^{cr}}\!\times \mathbb{R}^{N_h^{\smash{cr,C}}}\!$ for every ${(\mathrm{U}_h,\smash{\overline{\Lambda}}_h)^\top\!\!\in \!\mathbb{R}^{N_h^{cr}}\!\times \!\mathbb{R}^{N_h^{\smash{cr,C}}}}\!$~\hspace{-0.15mm}by
	\begin{align*}
		\mathscr{F}_h(\mathrm{U}_h,\smash{\overline{\Lambda}}_h)\coloneqq \bigg[\begin{array}{c}
			\mathrm{S}_h^{cr}\mathrm{U}_h+\mathrm{p}_h^{cr,C}\smash{\overline{\Lambda}}_h- \mathrm{P}_h^{cr,0}\mathrm{F}_h^0-\mathrm{G}_h^{cr}\\
			\mathscr{C}_h(\mathrm{U}_h,\smash{\overline{\Lambda}}_h)
		\end{array}\bigg]\quad \text{ in }\mathbb{R}^{N_h^{cr}}\times \mathbb{R}^{N_h^{\smash{cr,C}}}\,.
	\end{align*}
	Then, the non-linear system \eqref{eq:details.1} is equivalent to finding $(\mathrm{U}_h^{cr},\smash{\overline{\Lambda}}_h^{cr})^\top \in \mathbb{R}^{N_h^{cr}}\times \mathbb{R}^{N_h^{\smash{cr,C}}}$ such that 
	\begin{align*}
		\mathscr{F}_h(\mathrm{U}_h^{cr},\smash{\overline{\Lambda}}_h^{cr})=0_{\smash{\tiny\mathbb{R}^{N_h^{cr}}\times \mathbb{R}^{N_h^{\smash{cr,C}}}}}\quad\text{ in }\mathbb{R}^{N_h^{cr}}\times \mathbb{R}^{N_h^{\smash{cr,C}}}\,.
	\end{align*} 
	By analogy with \cite[Thm.\ 5.11]{Bartels15}, one finds that the mapping $\mathscr{F}_h\colon \mathbb{R}^{N_h^{cr}}\times \mathbb{R}^{N_h^{\smash{cr,C}}}\to \mathbb{R}^{N_h^{cr}}\times \mathbb{R}^{N_h^{\smash{cr,C}}}$ is Newton-differentiable at every $(\mathrm{U}_h,\smash{\overline{\Lambda}}_h)^\top\in \mathbb{R}^{N_h^{cr}}\times \mathbb{R}^{N_h^{\smash{cr,C}}}$ and with the (active) set
	\begin{align*}
		\mathscr{A}_h\coloneqq \mathscr{A}_h(\mathrm{U}_h,\smash{\overline{\Lambda}}_h)\coloneqq\big\{i\in \{1,\dots,N_h^{\smash{cr,C}}\}\mid (\smash{\overline{\Lambda}}_h+\alpha(\mathrm{p}_h^{cr,C})^{\top}(\mathrm{U}_h-\mathrm{X}_h^{cr}))\cdot \mathrm{e}_i<0\big\}\,,
	\end{align*}
	for every $\smash{(\mathrm{U}_h,\smash{\overline{\Lambda}}_h)^\top\in \mathbb{R}^{N_h^{cr}}\times \mathbb{R}^{N_h^{\smash{cr,C}}}}$, we have that
	\begin{align*}
		\mathrm{D} \mathscr{F}_h(\mathrm{U}_h,\smash{\overline{\Lambda}}_h)\coloneqq \bigg[\begin{array}{cc}
			\mathrm{S}_h^{cr} & \mathrm{p}_h^{cr,C} \\
			 \mathrm{I}_{\mathscr{A}_h} (\mathrm{p}_h^{cr,C})^\top & \mathrm{I}_{\mathscr{A}^c_h}
		\end{array}\bigg]\quad\text{ in }\mathbb{R}^{N_h^{cr}+N_h^{\smash{cr,C}}}\times \mathbb{R}^{N_h^{cr}+N_h^{\smash{cr,C}}}\,,
	\end{align*}
	where $\mathrm{I}_{\mathscr{A}_h},\mathrm{I}_{\mathscr{A}^c_h}\coloneqq \mathrm{I}_{N_h^{\smash{cr,C}}\times N_h^{\smash{cr,C}}}-\mathrm{I}_{\mathscr{A}_h}\in \mathbb{R}^{N_h^{\smash{cr,C}}}\times \mathbb{R}^{N_h^{\smash{cr,C}}}$ for every $i,j\in \{1,\dots,N_h^{\smash{cr,0}}\}$ are~defined~by
	$(\mathrm{I}_{\mathscr{A}_h})_{ij}\coloneqq 1$ if $i=j\in \mathscr{A}_h$ and $(\mathrm{I}_{\mathscr{A}_h})_{ij}\coloneqq 0$ else.
	
	For a given iterate $(\mathrm{U}_h^{k-1},\smash{\overline{\Lambda}}_h^{k-1})^\top\!\in \hspace*{-0.1em}\mathbb{R}^{N_h^{cr}} \hspace*{-0.1em}\times \hspace*{-0.1em} \mathbb{R}^{N_h^{\smash{cr,C}}}$, 
	one step of the~semi-smooth~\mbox{Newton}~method  determines a direction $(\delta\mathrm{U}_h^{k-1},\delta\smash{\overline{\Lambda}}_h^{k-1})^\top\in \mathbb{R}^{N_h^{cr}}\times \mathbb{R}^{N_h^{\smash{cr,C}}}$~such~that
	\begin{align}
		\mathrm{D} \mathscr{F}_h(\mathrm{U}_h^{k-1},\smash{\overline{\Lambda}}_h^{k-1})(\delta\mathrm{U}_h^{k-1},\delta\smash{\overline{\Lambda}}_h^{k-1})^\top=-\mathscr{F}_h(\mathrm{U}_h^{k-1},\smash{\overline{\Lambda}}_h^{k-1})\quad\text{ in }\mathbb{R}^{N_h^{cr}}\times \mathbb{R}^{N_h^{\smash{cr,C}}}\,.\label{eq:details.3}
	\end{align}
	Setting the update
	$(\mathrm{U}_h^k,\smash{\overline{\Lambda}}_h^k)^\top\coloneqq(\mathrm{U}_h^{k-1}+\delta\mathrm{U}_h^{k-1},\smash{\overline{\Lambda}}_h^{k-1}+\delta\smash{\overline{\Lambda}}_h^{k-1})^\top\in \mathbb{R}^{N_h^{cr}}\times \mathbb{R}^{N_h^{\smash{cr,C}}} $ and the~active set $\mathscr{A}_h^{k-1}\coloneqq \mathscr{A}_h(\mathrm{U}_h^{k-1},\smash{\overline{\Lambda}}_h^{k-1})$, the 
	linear system~\eqref{eq:details.3} can equivalently be re-written as
	\begin{align}\label{eq:equiv_equ}
		\begin{aligned}
			\mathrm{S}_h^{cr} \mathrm{U}_h^k+\mathrm{p}_h^{cr,C}\smash{\overline{\Lambda}}_h^k&=\mathrm{P}_h^{cr,0}\mathrm{F}_h^0&&\quad\text{ in }\mathbb{R}^{N_h^{cr}}\,,\\
			\mathrm{I}_{\smash{(\mathscr{A}_h^{k-1})^c}}\smash{\overline{\Lambda}}_h^k&=0_{\smash{\tiny \mathbb{R}^{N_h^{\smash{cr,0}}}}}&&\quad\text{ in }\mathbb{R}^{N_h^{\smash{cr,C}}}\,,\\
			\mathrm{I}_{\smash{\mathscr{A}_h^{k-1}}}(\mathrm{p}_h^{cr,C})^{\top}\mathrm{U}_h^k&= \mathrm{I}_{\smash{\mathscr{A}_h^{k-1}}}(\mathrm{p}_h^{cr,C})^{\top}\mathrm{X}_h^{cr}&& \quad\text{ in }\mathbb{R}^{N_h^{\smash{cr,C}}}\,.
		\end{aligned}
	\end{align}
	\hspace*{5mm}The semi-smooth Newton method  \eqref{eq:details.3} can, thus, equivalently be formulated in the following form, which is a version of a primal-dual active set strategy.\enlargethispage{5mm}

	\begin{algorithm}[primal-dual active set strategy]\label{alg:semi-smooth} Choose parameters $\alpha>0$ and $\varepsilon_{\textup{STOP}}>0$. Moreover, let
		$(\mathrm{U}_h^0,\smash{\overline{\Lambda}}_h^0)^\top\in \mathbb{R}^{N_h^{cr}}\times \mathbb{R}^{N_h^{\smash{cr,C}}}$ be an initial guess and set $k=1$. Then,~for~every~$k\in \mathbb{N}$:
		\begin{itemize}[noitemsep,topsep=2pt,leftmargin=!,labelwidth=\widthof{(iii)},font=\itshape]
			\item[(i)] Define the most recent active set
			$$\mathscr{A}^{k-1}_h\coloneqq \mathscr{A}_h(\mathrm{U}_h^{k-1},\smash{\overline{\Lambda}}_h^{k-1})\coloneqq \big\{i\in \{1,\dots,N_h^{\smash{cr,C}}\}\mid (\smash{\overline{\Lambda}}_h^{k-1}+\alpha(\mathrm{p}_h^{cr,C})^{\top}(\mathrm{U}_h^{k-1}-\mathrm{X}_h^{cr}))\cdot \mathrm{e}_i<0\big\}\,.$$
			
			\item[(ii)] Compute the iterate $(\mathrm{U}_h^k,\smash{\overline{\Lambda}}_h^k)^\top\in \mathbb{R}^{N_h^{cr}}\times \mathbb{R}^{N_h^{\smash{cr,C}}}$ such that
			\begin{align*}
				\bigg[\begin{array}{cc}
					\mathrm{S}_h^{cr} & \mathrm{p}_h^{cr,C} \\
					 \mathrm{I}_{\smash{\mathscr{A}_h^{k-1}}}(\mathrm{p}_h^{cr,C})^\top & \mathrm{I}_{\smash{(\mathscr{A}_h^{k-1})^c}}
				\end{array}\bigg]
				\bigg[\begin{array}{c}
					\mathrm{U}_h^k \\
					\smash{\overline{\Lambda}}_h^k
				\end{array}\bigg]=\bigg[\begin{array}{c}
					\mathrm{P}_h^{cr,0}\mathrm{F}_h^0\\
					\mathrm{I}_{\smash{\mathscr{A}^{k-1}_h}}(\mathrm{p}_h^{cr,C})^{\top}\mathrm{X}_h^{cr}
				\end{array}\bigg]\,.
			\end{align*}
			
			\item[(iii)] Stop if $\vert \mathrm{U}_h^k-\mathrm{U}_h^{k-1}\vert \leq \varepsilon_{\textup{STOP}}$; otherwise, increase $k\to k+1$ and continue with step (i).
		\end{itemize}
		
	\end{algorithm}
	
	\begin{remark}[Important implementation details]
		\begin{itemize}[noitemsep,topsep=2pt,leftmargin=!,labelwidth=\widthof{(iii)},font=\itshape]
			\item[(i)] Algorithm \ref{alg:semi-smooth} converges~\mbox{super-linearly} if $(\mathrm{U}_h^0,\smash{\overline{\Lambda}}_h^0)^\top\in \mathbb{R}^{N_h^{cr}}\times \mathbb{R}^{N_h^{\smash{cr,C}}}$ is sufficiently close to the solution ${(\mathrm{U}_h^{cr},\smash{\overline{\Lambda}}_h^{cr})^\top\in \mathbb{R}^{N_h^{cr}}\times \mathbb{R}^{N_h^{\smash{cr,C}}}}$. As the Newton-differentiability only holds in finite-dimensional~situations and~\mbox{deteriorates}~as $N_h^{cr}+N_h^{\smash{cr,C}}\hspace*{-0.175em}\to\hspace*{-0.175em} \infty$, the condition on the initial guess becomes more critical for ${N_h^{cr}\hspace*{-0.175em} +\hspace*{-0.175em} N_h^{\smash{cr,C}}\hspace*{-0.2em}\to\hspace*{-0.2em} \infty}$.
			\item[(ii)] The degrees of freedom related to the entries $\smash{\overline{\Lambda}}_h^k|_{\smash{(\mathscr{A}^{k-1}_h)^c}}$ can be eliminated from the linear system of equations in Algorithm \ref{alg:semi-smooth}, step (ii) (see also \eqref{eq:equiv_equ}$_2$).
			\item[(iii)] Since only a finite number of active sets are possible, the algorithm terminates within a finite number of iterations at the exact solution $(\mathrm{U}_h^{cr},\smash{\overline{\Lambda}}_h^{cr})^\top \in \mathbb{R}^{N_h^{cr}}\times \mathbb{R}^{N_h^{\smash{cr,C}}}$.~For~this~reason, in practice, the stopping criterion in step (iii) is reached with $\vert \mathrm{U}_h^{k^*}-\mathrm{U}_h^{k^*-1}\vert =0$~for~some~${k^*\in \mathbb{N}}$, in which case, one has that $\mathrm{U}_h^{k^*}=\mathrm{U}_h^{cr}$, provided $\varepsilon_{\textup{STOP}}>0$ is sufficiently small.
		%	\item[(iv)] The linear system emerging in each semi-smooth Newton step (\emph{cf}.\ Algorithm \ref{alg:semi-smooth}, step (ii)) is solved using a sparse direct solver from \textup{\textsf{SciPy}}~(version 1.8.1, \emph{cf}.\ \cite{SciPy}). 
		\end{itemize}
	\end{remark}\newpage
	
	\subsection{Numerical experiments concerning the a priori error analysis}\label{subsec:num_a_priori}
	
	\hspace{5.5mm}In this subsection, we review the theoretical findings of Section \ref{sec:apriori}.\enlargethispage{11mm}

	For  our numerical experiments, we choose a setup from \cite[Sec.\ 7]{SWW16}, \cite[§6]{CH20}, or \cite[Sec.\ 5.1]{AP24_arxiv}.
	More precisely, let $\Omega\coloneqq (0,1)^2$,  $	\Gamma_C\coloneqq
	(0,1)\times \{0\}$,  $\Gamma_D\coloneqq \partial\Omega \setminus \Gamma_C$ (\textit{cf}.\ Figure \ref{fig:apriori}(\textit{left})), 
	\textit{i.e.}, $\Gamma_N\coloneqq\emptyset$, and $ \chi  \coloneqq 0\in H^2(\Omega)$. Then, we compute  $f\in L^2(\Omega)$~such~that~the primal solution~$u\in K $, in~polar coordinates centered at $(0.5, 0)^\top\in \Gamma_C$, \textit{i.e.}, for every $x=(x_1,x_2)^\top\in \Omega$,~setting\vspace*{-0.5mm}
	\begin{align*}
	 	r(x)\coloneqq \smash{\big((x_1-\tfrac{1}{2})^2+x_2^2\big)^{\frac{1}{2}}}	\,,\qquad 
		\theta(x)\coloneqq \smash{\arccos\big(\tfrac{x_1-\frac{1}{2}}{r(x)}\big)}\,,\\[-6mm]
	\end{align*}
	 for every $ x\in \Omega$, is defined by\vspace*{-1mm}
	\begin{align*}
		u(x)\coloneqq- \smash{10\,\psi(r(x))\,r(x)^{\frac{3}{2}}\sin(\tfrac{3}{2}\theta(x))}\,.\\[-6mm]
	\end{align*}
	Here, %the function
	 $\psi\colon [0,\infty)\to \mathbb{R}$ (\textit{cf}.\ Figure \ref{fig:apriori}(\textit{right})) is the zero extension of a ninth-order spline with respect to the single element partition of $[0,0.45]$ which satisfies $\psi(r)>0$ for all $r\in (0.05,0.045)$, $\psi(r)=0$ for all $r\in[0.45,\infty)$,~and\vspace*{-0.5mm} %$1-\psi(0)=\psi(0.45)=\psi^{i}(0)=\psi^{i}(0.45)$ for all $i=1,\ldots,4$. In~this~example, we have that
	\begin{align*}
			1-\psi(0)=\psi(0.45)=\psi^{i}(0)=\psi^{i}(0.45)=0\quad\text{ for all }i=1,\ldots,4\,.\\[-6mm]
	\end{align*}
	
	\begin{figure}[H]\vspace*{-4mm}
		\centering

\tikzset{every picture/.style={line width=0.5pt}} %set default line width to 0.75pt        

\hspace*{5mm}
		\includegraphics[width=8.75cm]{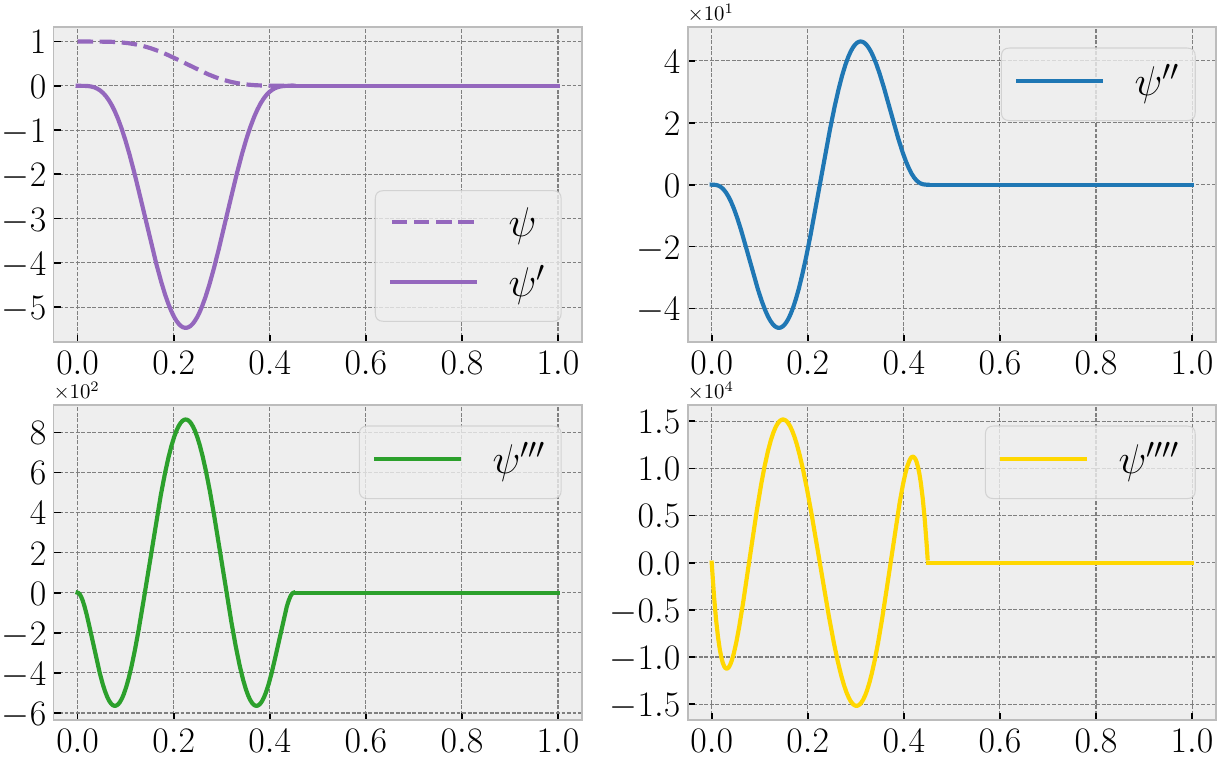}\vspace*{-1mm}
		\caption{\textit{left}:  $\Omega$, $\Gamma_C$, $\Gamma_D$, $\{u>0\}\cap \Gamma_C=\{z\cdot n=0\}\cap \Gamma_C$, and $\{u=0\}\cap \Gamma_C=\{z\cdot n>0\}\cap \Gamma_C$; \textit{right}: $\psi, \psi^{\prime},\psi^{\prime\prime}, \psi^{\prime\prime\prime}, \psi^{\prime\prime\prime\prime}\colon [0,1]\to\mathbb{R}$.}\label{fig:apriori}
	\end{figure}\vspace*{-3mm}
	In~this~example, we have that
	$u\in H^2(\Omega)$, so that Theorem \ref{thm:apriori_identity}(ii) suggests an experimental
	convergence rate of about $\mathcal{O}(h_k^2)= \mathcal{O}(N_k)$, where $N_k\coloneqq\textup{dim}(\mathcal{S}^{1,cr}_D(\mathcal{T}_{h_k}))+\textup{dim}(\mathcal{L}^0(\mathcal{S}_{h_k}^{\Gamma_C}))$,~${k\in \mathbb{N}}$, for the discrete primal-dual total errors (\textit{cf}.\ \eqref{eq:discrete_primal_dual_error}), which are equal to the discrete primal-dual~gap estimators (\textit{cf}.\ \eqref{eq:discrete_primal_dual_gap_estimator}), \textit{i.e.}, we expect (\textit{cf}.\ \mbox{Theorem} \ref{thm:apriori_identity}(i))\vspace*{-0.5mm}
	\begin{align*}
		\rho_{\textup{tot},h_k}^2(\Pi_{h_k}^{cr}u,\Pi_{h_k}^{rt} z)=	\rho_{\textup{gap},h_k}^2(\Pi_{h_k}^{cr}u,\Pi_{h_k}^{rt} z)=\mathcal{O}(h_k^2)=\mathcal{O}(N_k)\,.\\[-6mm]
	\end{align*}

	An initial triangulation $\mathcal
	T_{h_0}$, $h_0=\smash{\sqrt{2}}$, is constructed by subdividing the unit square $\Omega $ along its diagonal from $(0,0)^\top$ to $(1,1)^\top$ into two triangles.
	Refined     triangulations~$\mathcal T_{h_k}$,~${k=1,\dots,7}$, where $h_{k+1}=\frac{h_k}{2}$ for all $k=1,\dots,7$, are 
	obtained by 
	applying the red-refinement routine~(\textit{cf}.~\cite{Verfuerth13}).
	
	For the resulting series of triangulations $\mathcal{T}_{h_k}$, $k=1,\ldots,7$, we apply~the~\mbox{primal-dual} active set strategy (\textit{cf}.\ Algorithm \ref{alg:semi-smooth})
	to compute the discrete primal solution $u_{h_k}^{cr}\in  K_{h_k}^{cr}$,~${k=1,\ldots,7}$, the discrete Lagrange multiplier $\smash{\overline{\lambda}}_{h_k}^{cr}\in \mathcal{L}^0(\mathcal{S}_{h_k}^{\Gamma_C})$, $k\hspace*{-0.15em}=\hspace*{-0.15em}1,\ldots,7$,~and,~\mbox{subsequently}, resorting to \eqref{eq:generalized_marini}, the discrete dual solution $ z_{h_k}^{rt}\in K_{h_k}^{rt,*}$, $k=1,\ldots,7$. Then,
	we compute~the~error quantities
	\begin{align}\label{eq:error_quantities}
		\left.\begin{aligned}
			e_k^{\textup{tot}}&\coloneqq 	\smash{\rho_{\textup{tot},h_k}^2(\Pi_{h_k}^{cr}u,\Pi_{h_k}^{rt} z)}\,,\\
			e_k^{\textup{gap}}&\coloneqq 	\smash{\rho_{\textup{gap},h_k}^2(\Pi_{h_k}^{cr}u,\Pi_{h_k}^{rt} z)}\,,\\ 
			e_k^{\Delta}&\coloneqq \vert 		\smash{\rho_{\textup{tot},h_k}^2(\Pi_{h_k}^{cr}u,\Pi_{h_k}^{rt} z)-\rho_{\textup{gap},h_k}^2(\Pi_{h_k}^{cr}u,\Pi_{h_k}^{rt} z)\vert}\,,
		\end{aligned}\quad\right\}\quad k=1,\ldots,7\,.
	\end{align} 
	
	For determining the convergence rates, % of the error quantities in \eqref{eq:error_quantities}, 
	the experimental order of convergence~(EOC),~\textit{i.e.},\vspace*{-1mm}
	\begin{align*}
		\texttt{EOC}_k(e_k)\coloneqq \frac{\log(e_k)-\log(e_{k-1})}{\log(h_k)-\log(h_{k-1})}\,,\quad k=1,\ldots,7\,,\\[-6mm]
	\end{align*}
	where, for every $k=1,\ldots, 7$, we denote by $e_k$, either $e_k^{\textup{gap}}$, $e_k^{\textup{tot}}$, or $e_k^{\Delta}$, respectively, is recorded.
	
	\newpage In Figure \ref{fig:apriori_exp1}, we report the expected  optimal convergence rate of $\texttt{EOC}_k(e_k^{\textup{tot}})\approx\texttt{EOC}_k(e_k^{\textup{gap}})\approx  2$, $k=1,\ldots, 7$,
	\textit{i.e.}, an error decay of order 
	$\mathcal{O}(h_k^2)=  \mathcal{O}(N_k)$, $k=1,\ldots, 7$. In addition, we observe that the \textit{a priori} error identity in Theorem~\ref{thm:apriori_identity}(i) is asymptotically satisfied. More precisely, for 
	the error between the discrete primal-dual total error  (\textit{cf}.\ \eqref{eq:discrete_primal_dual_error}) and the discrete primal-dual gap estimator (\textit{cf}.\ \eqref{eq:discrete_primal_dual_gap_estimator}), we report a convergence rate of about $\texttt{EOC}_k(e_k^{\Delta})\approx 3.7$, $k=1,\ldots, 7$,
	\textit{i.e.}, an error decay of order 
	$\mathcal{O}(h_k^{3.7})=  \mathcal{O}(N_k^{1.85})$, $k=1,\ldots, 7$, which is the quadrature error involved in the computation of the quasi-interpolants $\Pi_{h_k}^{cr}u\in  K_{h_k}^{cr}$,~${k=1,\ldots,7}$,~and~${\Pi_{h_k}^{rt}z\in K_{h_k}^{rt,*}}$,~${k=1,\ldots,7}$.\enlargethispage{15mm}

	 \begin{figure}[H]
		\centering
		\includegraphics[width=7.25cm]{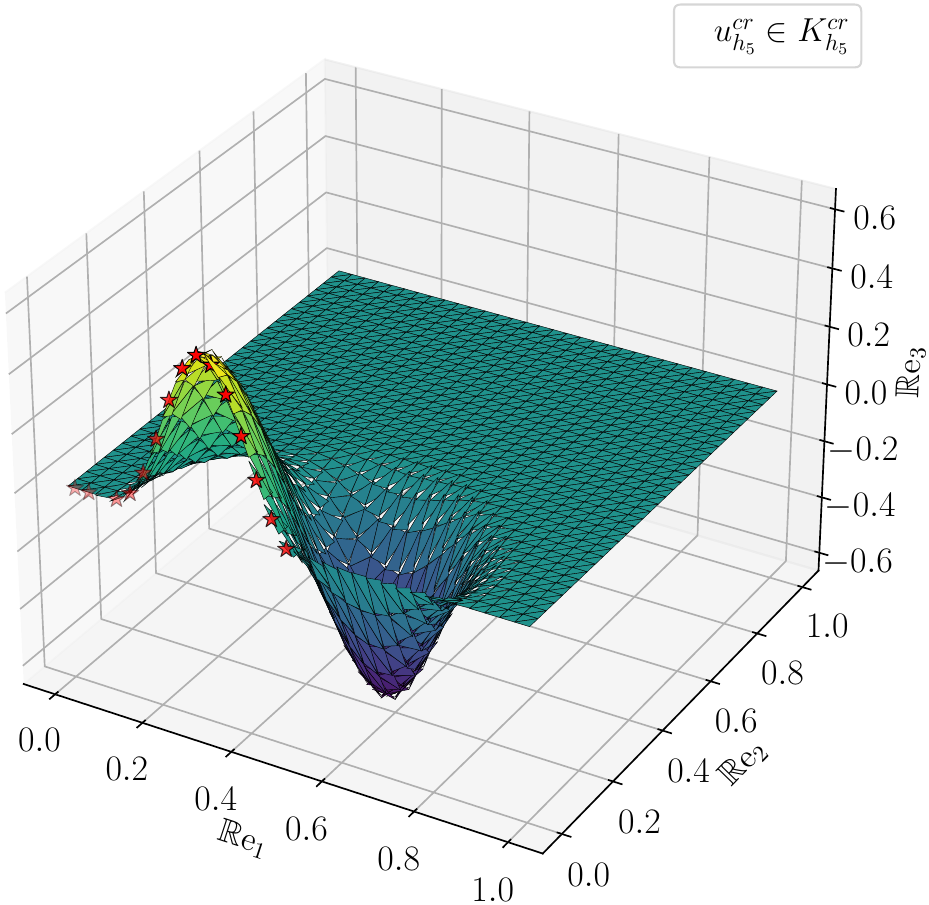}\includegraphics[width=7.25cm]{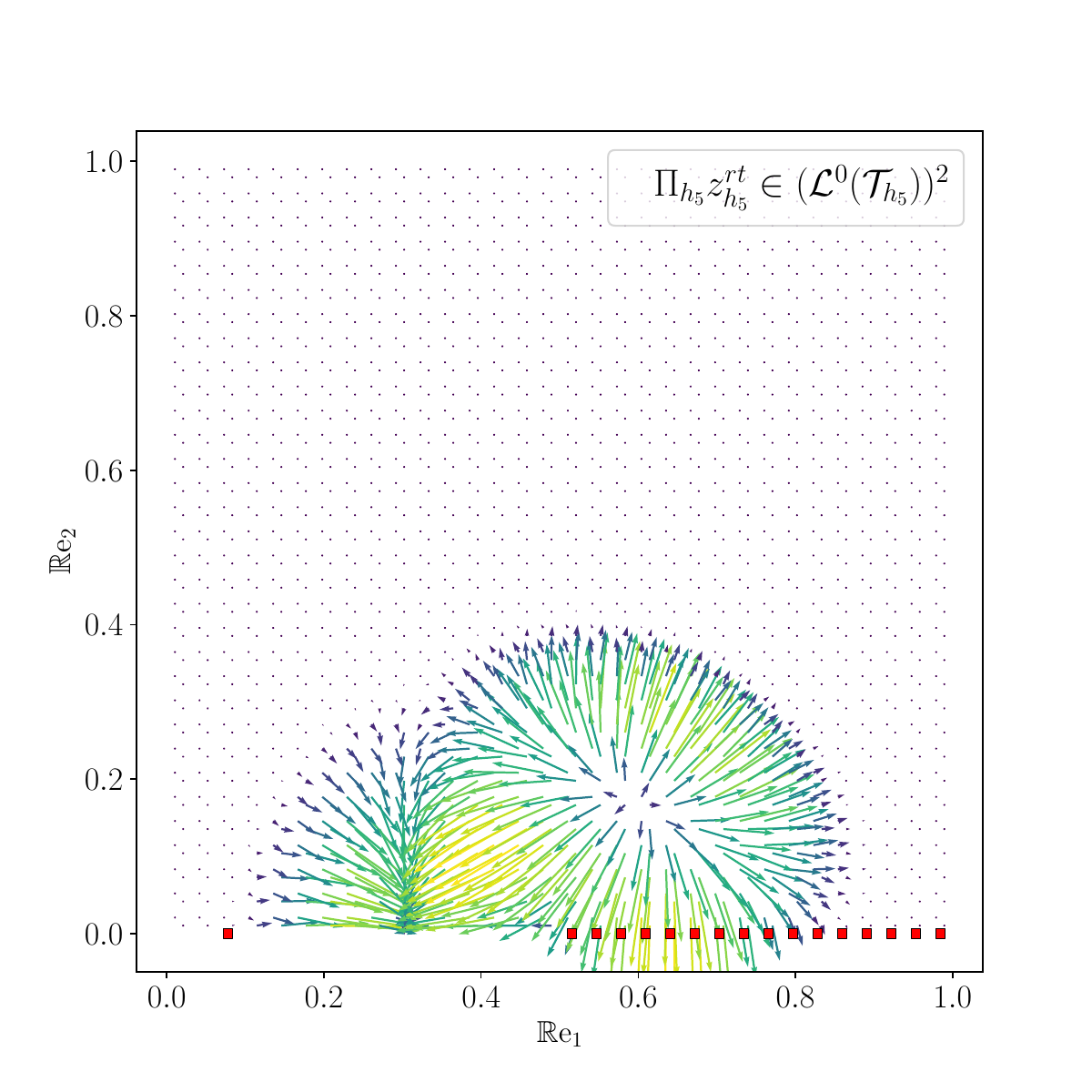}
		\caption{\textit{left}:  discrete primal solution $u_{h_5}^{cr}\in K_{h_5}^{cr}$, where red stars mark sides $S\in \mathcal{T}_{h_5}$~with $\pi_{h_5}u_{h_5}^{cr}|_{S}>0$; \textit{right}:
			\hspace*{-0.1mm}(local) \hspace*{-0.1mm}$L^2$-projection \hspace*{-0.1mm}(onto \hspace*{-0.1mm}$(\mathcal{L}^0(\mathcal{T}_{h_5}))^d$) \hspace*{-0.1mm}of \hspace*{-0.1mm}discrete \hspace*{-0.1mm}dual~\hspace*{-0.1mm}\mbox{solution}~\hspace*{-0.1mm}${z_{h_5}^{rt}\hspace*{-0.1em}\in \hspace*{-0.1em} K_{h_5}^{rt,*}}$, where red squares mark sides $S\in \mathcal{T}_{h_5}$ with $z_{h_5}^{rt}\cdot n|_{S}>0$. We find that $z_{h_5}^{rt}\cdot n\,\pi_{h_5}u_{h_5}^{cr}=0$~a.e.~on~$\Gamma_C$.}
	\end{figure}\vspace*{-5mm}
	 
	 \begin{figure}[H]
	 	%\hspace*{-2mm}
	 	\includegraphics[width=14.5cm]{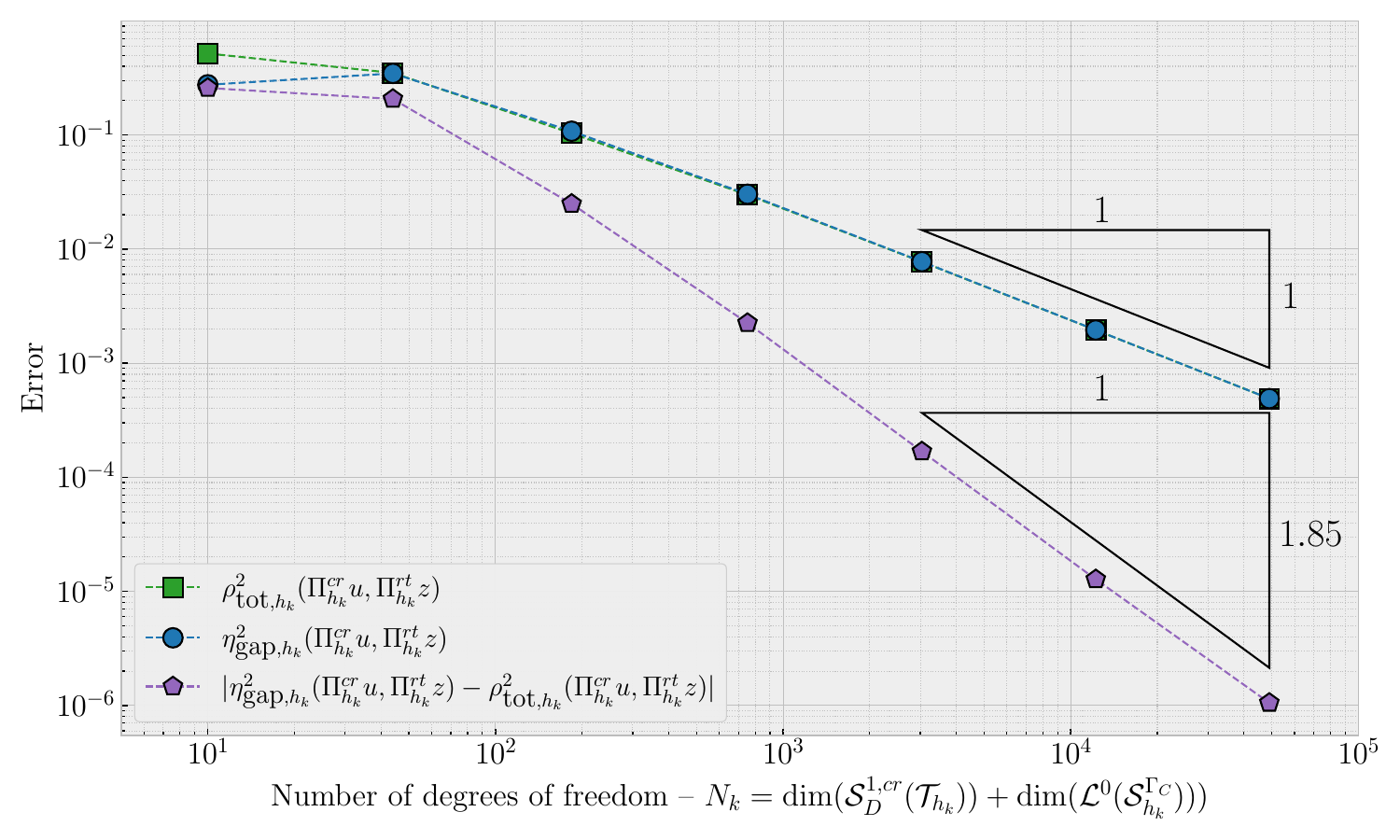}
	 	\caption{Logarithmic plots of the experimental convergence rates of the error quantities \eqref{eq:error_quantities}. We observe the experimental orders of convergence $\texttt{EOC}_k(e_k^{\textup{tot}})\approx\texttt{EOC}_k(e_k^{\textup{gap}})\approx  2$, $k=1,\ldots, 7$, and $ \texttt{EOC}_k(e_k^{\Delta})\approx 3.7$, $k=1,\ldots, 7$.}
	 	\label{fig:apriori_exp1}
	 \end{figure} 
	 
	 \newpage
	 \subsection{Numerical experiments concerning a posteriori error analysis}\enlargethispage{5mm}
	 
	 \hspace*{5mm}In this subsection, we review the theoretical findings of Section \ref{sec:aposteriori}. 
	 
	 More precisely, we employ the local refinement indicators $\eta_{\textup{gap},T}^2\colon K\times K^*_{\textup{tr}_n \in L^2(\Gamma_C)}\to [0,+\infty)$, $T\in \mathcal{T}_h$, where
	 \begin{align*}
	 	K^*_{\textup{tr}_n \in L^2(\Gamma_C)}\coloneqq \big\{y\in K^*\mid y\cdot n\in L^2(\Gamma_C)\big\}\,,
	 \end{align*}
		induced by  the primal-dual gap estimator (\textit{cf}.\ \eqref{eq:primal-dual.1}),  for every $v\in K$, $y\in K^*_{\textup{tr}_n \in L^2(\Gamma_C)}$,~and~${T\in \mathcal{T}_h}$, defined by  
	  \begin{align}\label{def:local_refinement_indicators}
	 		\eta_{\textup{gap},T}^2(v,y)\coloneqq  \tfrac{1}{2}\|\nabla v- y\|_T^2+(y\cdot n, v-\chi)_{\partial T\cap \Gamma_C}\,,
	 \end{align} 
	 in an adaptive mesh-refinement scheme. The definition of the local refinement indicators (\textit{cf}.\ \eqref{def:local_refinement_indicators}) is motivated by 
	the representation of  the primal-dual gap estimator (\textit{cf}.\ \eqref{eq:primal-dual.1})~in~Lemma~\ref{lem:primal_dual_gap_estimator}. 
	 
	 The numerical experiments are based on the following \emph{adaptive algorithm}:
	 
	 \begin{algorithm}[AFEM]\label{alg:afem}
	 	Let $\varepsilon_{\textup{STOP}}>0$, $\theta\in (0,1)$, and  $\mathcal{T}_0$ an initial  triangulation of $\Omega$. Then, for every $k\in \mathbb{N}\cup \{0\}$:
	 	\begin{description}[noitemsep,topsep=1pt,labelwidth=\widthof{\textit{('Estimate')}},leftmargin=!,font=\normalfont\itshape]
	 		\item[('Solve')]\hypertarget{Solve}{}
	 		\hspace*{-0.5mm}Compute \hspace*{-0.15mm}the \hspace*{-0.15mm}discrete \hspace*{-0.15mm}primal \hspace*{-0.15mm}solution \hspace*{-0.15mm}$u_{h_k}^{cr}\hspace*{-0.225em} \in\hspace*{-0.175em} K_{h_k}^{cr}$ \hspace*{-0.15mm}and 
	 		\hspace*{-0.15mm}the \hspace*{-0.15mm}discrete \hspace*{-0.15mm}dual \hspace*{-0.15mm}solution~\hspace*{-0.15mm}${z_{h_k}^{rt}\hspace*{-0.225em}\in\hspace*{-0.175em} K_{h_k}^{rt,*}}\hspace*{-0.175em}$.
	 		
	 		Post-process \hspace*{-0.15mm}$u_{h_k}^{cr}\hspace*{-0.2em}\in\hspace*{-0.175em} K_{h_k}^{cr}$ \hspace*{-0.15mm}and  \hspace*{-0.15mm}$z_{h_k}^{rt}\hspace*{-0.2em}\in\hspace*{-0.175em}  K_{h_k}^{rt,*}$ \hspace*{-0.15mm}to \hspace*{-0.15mm}obtain \hspace*{-0.15mm}a \hspace*{-0.15mm}conforming~\hspace*{-0.15mm}\mbox{approximations}~\hspace*{-0.15mm}${\overline{u}_{h_k}^{cr}\hspace*{-0.2em}\in\hspace*{-0.175em} K}$ \hspace*{-0.15mm}and $\overline{z}_{h_k}^{rt}\in  K^*$ \hspace*{-0.15mm}of \hspace*{-0.15mm}the \hspace*{-0.15mm}primal \hspace*{-0.15mm}solution \hspace*{-0.15mm}$u\in K$ \hspace*{-0.15mm}and \hspace*{-0.15mm}the \hspace*{-0.15mm}dual \hspace*{-0.15mm}solution \hspace*{-0.15mm}$z\in K^*$, \hspace*{-0.15mm}respectively; 
	 		
	 		\item[('Estimate')]\hypertarget{Estimate}{} \hspace*{-0.5mm}Compute \hspace*{-0.15mm}the \hspace*{-0.15mm}resulting \hspace*{-0.15mm}local \hspace*{-0.15mm}refinement \hspace*{-0.15mm}primal-dual \hspace*{-0.15mm}indicators \hspace*{-0.15mm}$\smash{\{\eta^2_{\textup{gap},T}(\overline{u}_{h_k}^{cr},\overline{z}_{h_k}^{rt})\}_{T\in \mathcal{T}_{h_k}}}$. If $%\smash{\rho^2_{\textup{tot}}}(\overline{u}_k^{cr},\overline{z}_k^{rt})=
	 		\smash{\eta^2_{\textup{gap}}}(\overline{u}_{h_k}^{cr},\overline{z}_{h_k}^{rt})\leq \varepsilon_{\textup{STOP}}$, then \textup{STOP}; otherwise, continue with step (\hyperlink{Mark}{'Mark'});
	 		\item[('Mark')]\hypertarget{Mark}{}  Choose a minimal (in terms of cardinality) subset $\mathcal{M}_{h_k}\subseteq\mathcal{T}_{h_k}$ such that\vspace*{-0.5mm} %the following bulk-criterion is satisfied:
	 		\begin{align*}
	 			\sum_{T\in \mathcal{M}_{h_k}}{\eta_{\textup{gap},T}^2(\overline{u}_{h_k}^{cr},\overline{z}_{h_k}^{rt})}\ge \theta^2\sum_{T\in \mathcal{T}_{h_k}}{\eta_{\textup{gap},T}^2(\overline{u}_{h_k}^{cr},\overline{z}_{h_k}^{rt})}\,;
	 		\end{align*}
	 		\item[('Refine')]\hypertarget{Refine}{} Perform a conforming refinement of $\mathcal{T}_{h_k}$ to obtain $\mathcal{T}_{h_{k+1}}$~such~that~each~\mbox{element} $T\in \mathcal{M}_{h_k}$  is `refined' in $\mathcal{T}_{h_{k+1}}$.  
	 		Increase~$k\mapsto k+1$~and~continue~with~step~(\hyperlink{Solve}{'Solve'}).
	 	\end{description}
	 \end{algorithm}
	 
	 \begin{remark}[Implementation details]
	 	\begin{description}[noitemsep,topsep=1pt,labelwidth=\widthof{\textit{(iii)}},leftmargin=!,font=\normalfont\itshape]
	 		\item[(i)] The discrete primal solution $u_{h_k}^{cr}\in  K_{h_k}^{cr}$ and the discrete Lagrange multiplier
	 		$\smash{\overline{\lambda}}_{h_k}^{cr}\in \mathcal{L}^0(\mathcal{S}_{h_k}^{\Gamma_C})$ in step (\hyperlink{Solve}{'Solve'}) are computed using the primal-dual active set strategy 
	 		(\textit{cf}.\ Algorithm \ref{alg:semi-smooth}) for the parameter $\alpha = 1$;
	 		\item[(ii)] The computation of the discrete dual solution in step (\hyperlink{Solve}{'Solve'})  is based on the reconstruction formula \eqref{eq:generalized_marini}. Note that $z_{h_k}^{rt}\in K^*$ if and only if $f=f_{h_k}\in \mathcal{L}^0(\mathcal{T}_{h_k})$ and~${g=g_{h_k}\in \mathcal{L}^0(\mathcal{S}_{h_k}^{\Gamma_N})}$;
	 		\item[(iii)] If $\chi|_{\Gamma_D\cup\Gamma_C}\hspace*{-0.175em}\in\hspace*{-0.175em} \mathcal{L}^1(\mathcal{S}_{h_k}^{\Gamma_D}\cup\mathcal{S}_{h_k}^{\Gamma_C})$, \textit{i.e.}, $u_D\hspace*{-0.175em}\in \hspace*{-0.175em}\mathcal{L}^1(\mathcal{S}_{h_k}^{\Gamma_D})$,
	 		%	$u_D=0$ a.e.\ on $\Gamma_D$ and $\chi=0$ a.e.\ on $\Gamma_C$,
	 		then as an admissible~\mbox{approximation}~${\overline{u}_{h_k}^{cr}\hspace*{-0.175em}\in\hspace*{-0.175em} K}$ in step (\hyperlink{Solve}{'Solve'}), we employ a contact boundary modified node-averaging~\mbox{quasi-interpolant}, \textit{i.e.},\enlargethispage{5mm}
	 		\begin{align*}
	 			\overline{u}_{h_k}^{cr}&\coloneqq \sum_{\nu \in \mathcal{N}_{h_k}}{\{ u_{h_k}^{cr}\}_\nu \,\varphi_\nu}\in K\,,\\
	 			\text{where }\{u_{h_k}^{cr}\}_\nu& \coloneqq\begin{cases}
	 				\frac{1}{\textup{card}(\mathcal{T}_{h_k}(\nu))}\sum_{T \in \mathcal{T}_{h_k}(\nu)}{(u_{h_k}^{cr}|_T)(\nu)}&\text{ if }\nu \in \Omega\cup\Gamma_N\,,\\[0.5mm]
	 				\max\Big\{\chi(\nu),\frac{1}{\textup{card}(\mathcal{T}_{h_k}(\nu))}\sum_{T \in \mathcal{T}_{h_k}(\nu)}{(u_{h_k}^{cr}|_T)(\nu)}\Big\}&\text{ if }\nu \in \Gamma_C\,,\\[0.5mm]
	 				u_D(\nu)&\text{ if }\nu \in \Gamma_D\,,
	 			\end{cases}
	 		\end{align*}
	 		where  
	 		$(\varphi_{\nu})_{\nu\in \mathcal{N}_{h_k}}\subseteq \mathcal{S}^1(\mathcal{T}_{h_k})$ denotes the shape basis of $\mathcal{S}^1(\mathcal{T}_{h_k})\coloneqq \mathcal{L}^1(\mathcal{T}_{h_k})\cap H^1(\Omega)$ and, for every $\nu\in \mathcal{N}_{h_k}$, we denote by
	 		$\mathcal{T}_{h_k}(\nu)\coloneqq \{T\in \mathcal{T}_{h_k}\mid \nu \in T\}$ the set of elements containing $\nu$;
	 		\item[(iv)] By the primal-dual gap identity (\textit{cf}.\ Theorem \ref{thm:prager_synge_identity}), the stopping criterion in step  (\hyperlink{Estimate}{'Estimate'}) guarantees accuracy of  $\overline{u}_{h_k}^{cr}\in K$ and $z_{h_k}^{rt}\in K^*$ 
	 		in terms of the primal-dual total error (\textit{cf}.\ \eqref{def:primal_dual_total_error} with Lemma \ref{lem:discrete_strong_convexity_measures}), \textit{i.e.},  $\smash{\rho^2_{\textup{tot}}}(\overline{u}_k^{cr},\overline{z}_k^{rt})=
	 		\smash{\eta^2_{\textup{gap}}}(\overline{u}_{h_k}^{cr},\overline{z}_{h_k}^{rt})\leq \varepsilon_{\textup{STOP}}$ in step (\hyperlink{Estimate}{'Estimate'}).
	 		\item[(i)] If not otherwise specified, we employ the parameter $\theta=\smash{\frac{1}{2}}$ in step (\hyperlink{Estimate}{'Mark'}).
	 		\item[(ii)] To find the  set $\mathcal{M}_{h_k}\hspace*{-0.175em}\subseteq\hspace*{-0.175em} \mathcal{T}_{h_k}$ in step (\hyperlink{Mark}{'Mark'}), we resort to the D\"orfler marking~\mbox{strategy}~(\textit{cf}.~\cite{Doerfler96}).
	 		\item[(iii)] The  (minimal) conforming refinement of $\mathcal{T}_{h_k}$ with respect to  $\mathcal{M}_{h_k}$ in step (\hyperlink{Refine}{'Refine'}) is obtained by deploying the \textit{red}-\textit{green}-\textit{blue}-refinement algorithm (\textit{cf}.\ \cite{Carstensen04}).
	 	\end{description}
	 \end{remark}
	
	\subsubsection{Example with unknown primal and dual solution}
	
	\qquad In this example, let  $\Omega\coloneqq (-1,1)^2 $,  $
	\Gamma_C \coloneqq (-1,1)\times \{-1\}$,  $\Gamma_D \coloneqq 
	((-1,1)\times\{1\})\cup (\{1\}\times (0,1) )$ (\textit{cf}.\ Figure \ref{fig:aposteriori1}(\textit{left})),
	\textit{i.e.}, $\Gamma_N\hspace*{-0.1em}\coloneqq\hspace*{-0.1em} \partial\Omega\hspace*{-0.1em}\setminus\hspace*{-0.1em} (\Gamma_D\hspace*{-0.1em}\cup\hspace*{-0.1em} \Gamma_C )$, 
	 $f\hspace*{-0.1em}=\hspace*{-0.1em}-1\hspace*{-0.1em}\in\hspace*{-0.1em} L^2(\Omega)$, $g\hspace*{-0.1em}=\hspace*{-0.1em}0\hspace*{-0.1em}\in\hspace*{-0.1em} L^2(\Gamma_N)$,~and~${\chi\hspace*{-0.1em}\in\hspace*{-0.1em} H^1_D(\Omega)}$ with $\chi(x)\hspace*{-0.1em}\coloneqq\hspace*{-0.1em}\min\{\frac{1}{2}(\vert x_1\vert-\frac{1}{2}),0\}$ for all $x\hspace*{-0.1em}=\hspace*{-0.1em}(x_1,x_2)^\top\hspace*{-0.1em}\in\hspace*{-0.1em} \Gamma_C$. In this case, the primal~solution~$u\hspace*{-0.1em}\in\hspace*{-0.1em} K$ is not known and since the Dirichlet part $\Gamma_D$ and the Neumann part $\Gamma_N$ touch in 
	 $(1,0)^\top$ with interior angle $\pi$ (\textit{cf}.\ Figure \ref{fig:aposteriori1}(\textit{left})), it 
	 cannot be expected to satisfy $u\in H^2(\Omega)$, so that uniform mesh refinement is expected to yield a reduced error decay rate compared to the quasi-optimal linear error decay rate.

	Algorithm \ref{alg:afem} refines the mesh towards the contact set $\Gamma_C$ and  $(1,0)^\top$ (\textit{cf}. Figure \ref{fig:aposteriori2}), where we expect a singularity. In Figure \ref{fig:aposteriori1}(\textit{right}), 
	 one finds that uniform mesh refinement (\textit{i.e.}, $\theta=1$ in Algorithm~\ref{alg:afem}) yields the 
	% expected 
	 reduced convergence rate $\smash{h_k \sim N_k^{\smash{-\frac{2}{3}}}}$, $k=0,\dots,4$, while adaptive mesh refinement (\textit{i.e.}, $\theta=\frac{1}{2}$ in Algorithm \ref{alg:afem}) 
	 yields  the optimal convergence~rate~$\smash{h_k^2\hspace*{-0.1em}\sim\hspace*{-0.1em} N_k^{-1}}$,~${k\hspace*{-0.1em}=\hspace*{-0.1em}0,\dots,20}$.
	\begin{figure}[H]

		\tikzset{every picture/.style={line width=0.5pt}} %set default line width to 0.75pt        
		
		\hspace*{-2.5mm}
		\hspace*{-1.5mm}\includegraphics[width=7.5cm]{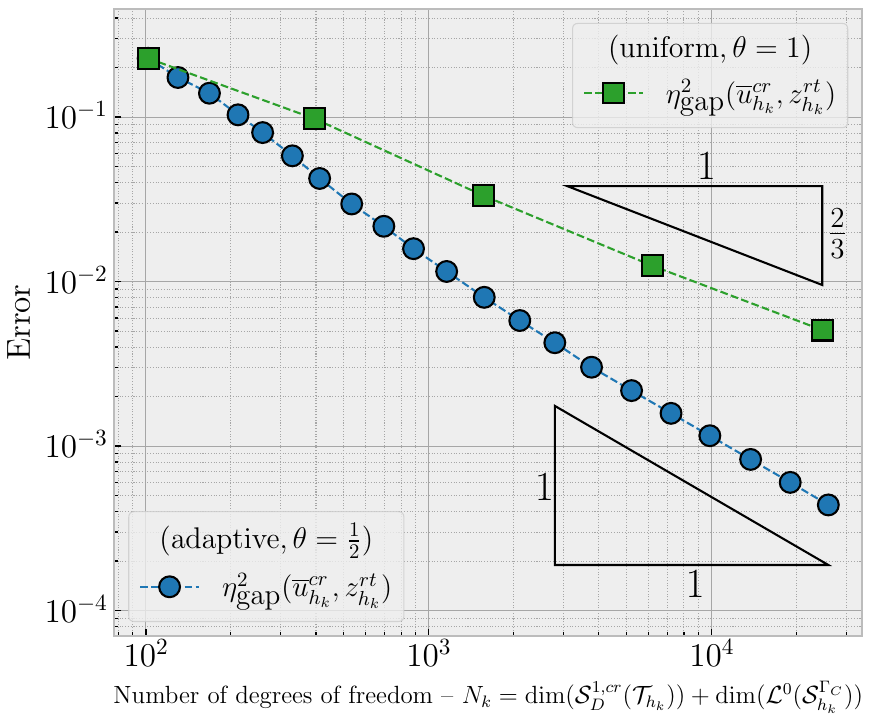}
		\caption{\textit{left}: $\Omega$, $\Gamma_C$, $\Gamma_D$, $\Gamma_N$, and $(1,0)^\top$; 
	\textit{right}: primal-dual gap estimator $\smash{\eta^2_{\textup{gap}}}(\overline{u}_{h_k}^{cr},z_{h_k}^{rt})$ %; \textit{right}: primal energy  $I(\overline{u}_{h_k}^{cr})$, dual energy $D(z_{h_k}^{rt})$, and primal energy $I(u)$ approximated resorting to Aitken’s $\delta^2$-process (\textit{cf}. \cite{Aitken26});
			%each 
			for $k=0,\dots,20$, when employing adaptive mesh refinement  (\textit{i.e.}, $\theta=\frac{1}{2}$ in Algorithm \ref{alg:afem}),  and for $k=0,\dots, 4 $, when employing uniform mesh refinement (\textit{i.e.}, $\theta=1$ in Algorithm \ref{alg:afem}).}
			\label{fig:aposteriori1}
	\end{figure}
	
	\begin{figure}[H]
		\hspace*{-1mm}\includegraphics[width=14.75cm]{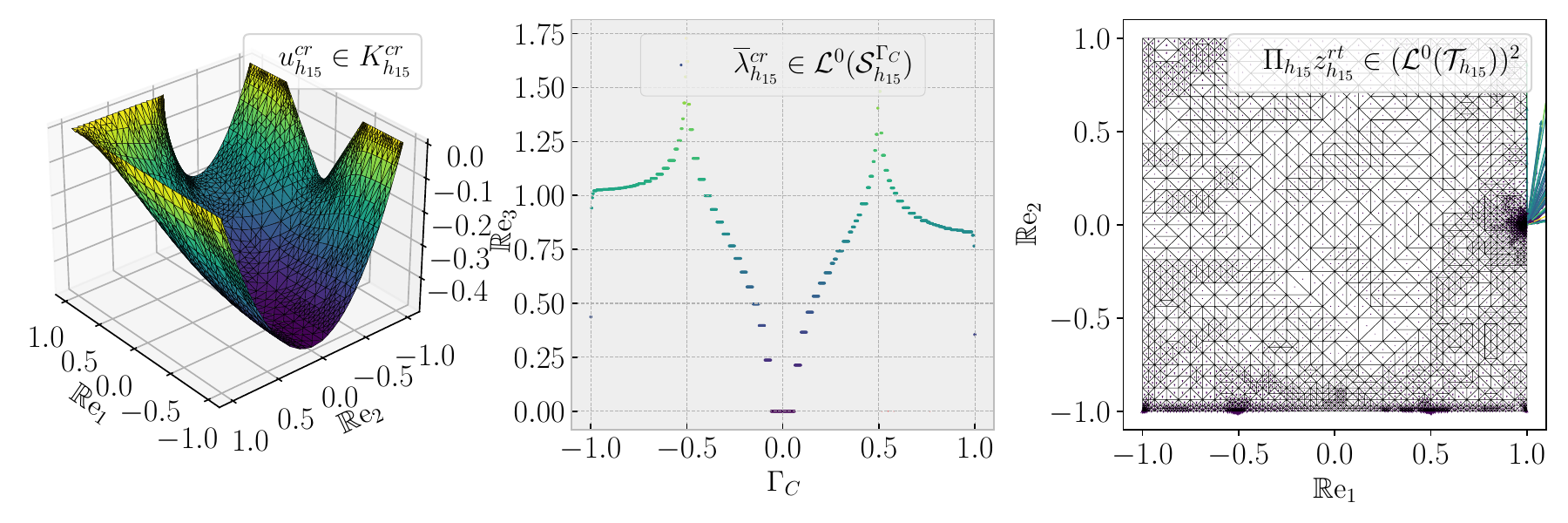}
		\caption{\textit{left}: discrete primal solution $u_{h_{15}}^{cr}\in K_{h_{15}}^{cr}$; MIDDLE: discrete Lagrange multiplier
			$\smash{\overline{\lambda}}_{h_{15}}^{cr}\in \mathcal{L}^0(\mathcal{S}_{h_{15}}^{\Gamma_{C}})$; \textit{right}: (local) $L^2$-projection %(onto $(\mathcal{L}^0(\mathcal{T}_h))^2$) 
			of the discrete dual solution ${z_{h_{15}}^{rt}\in K_{h_{15}}^{rt,*}}$.}
		\label{fig:aposteriori2}
	\end{figure}
	\newpage
	\appendix
	\section{Appendix}
	
	\hspace*{5mm}In this appendix, we prove a lifting lemma that for a given element-wise constant vector field, a given element-wise constant function, and a given side-wise constant function defined on Neumann sides, 
	jointly satisfying a compatibility condition, 
	provides a Raviart--Thomas vector field whose (local) $L^2$-projection coincides with the element-wise constant vector field, whose divergence coincides with the element-wise constant function, and whose normal traces coincide with the side-wise constant function on Neumann sides. 
	
	\begin{lemma}[lifting]\label{lem:reconstruction}
		Let $\overline{y}_h\in (\mathcal{L}^0(\mathcal{T}_h))^d$, $f_h\in \mathcal{L}^0(\mathcal{T}_h)$, and $g_h\in \mathcal{L}^0(\mathcal{S}_h^{\Gamma_N})$ be such that  for every $v_h\in \mathcal{S}^{1,cr}_D(\mathcal{T}_h)$ with $\pi_h v_h=0$ a.e.\ on $\Gamma_C$, there holds the \emph{compatibility condition}
		\begin{align}\label{lem:reconstruction.0.0}
			(\overline{y}_h,\nabla_h v_h)_{\Omega}-(f_h,\Pi_h v_h)_{\Omega}-(g_h,\pi_h v_h)_{\Gamma_N}=0\,.
		\end{align}
		Then, the vector field $y_h\in  (\mathcal{L}^1(\mathcal{T}_h))^d$ defined by
		\begin{align}\label{lem:reconstruction_formula}
			y_h\coloneqq \overline{y}_h-\frac{f_h}{d}(\textup{id}_{\mathbb{R}^d}-\Pi_h \textup{id}_{\mathbb{R}^d})\quad\text{ a.e.\ in }\Omega\,,
		\end{align}
		satisfies $y_h\in  \mathcal{R}T^0(\mathcal{T}_h)$ and
		\begin{alignat}{2} 
				\Pi_h y_h&=\overline{y}_h&&\quad \text{ a.e.\ in }\Omega\,,\label{lem:reconstruction.0.1}\\
				\textup{div}\, y_h&=-f_h&&\quad \text{ a.e.\ in }\Omega\,,\label{lem:reconstruction.0.2}\\
				y_h\cdot n&=g_h&&\quad \text{ a.e.\ on }\Gamma_N\,.\label{lem:reconstruction.0.3} 
		\end{alignat}
	\end{lemma}
	\begin{proof}
		From the definition  \eqref{lem:reconstruction_formula}, it follows directly that
		\eqref{lem:reconstruction.0.1} is satisfied.
		Since,~due~to~$\vert \Gamma_D\vert>0$, $\textup{div}\colon \mathcal{R}T^0(\mathcal{T}_h)\to \mathcal{L}^0(\mathcal{T}_h)$ 
		is surjective, there exists $\widehat{y}_h\in \mathcal{R}T^0(\mathcal{T}_h)$ such that $\textup{div}\,\widehat{y}_h=-f_h$~a.e.~in~$\Omega$. 
		Then, using the discrete integration-by-parts formula \eqref{eq:pi} and \eqref{lem:reconstruction.0.0}, for~every~$v_h\in \mathcal{S}^{1,cr}_D(\mathcal{T}_h)$ with $\pi_hv_h=0$ a.e.\ on $\Gamma_N\cup\Gamma_C$, we find that
		\begin{align}\label{lem:reconstruction.1}
			\begin{aligned}
				(\Pi_h\widehat{y}_h,\nabla_h v_h)_{\Omega}&=-(\textup{div}\,\widehat{y}_h,\Pi_h v_h)_{\Omega}\\&=(f_h,\Pi_h v_h)_{\Omega}\\&=(\overline{y}_h,\nabla_h v_h)_{\Omega}%\\&=( \Pi_h y_h,\nabla_h v_h)_{\Omega}
				\,.
			\end{aligned}
		\end{align}
		Using \eqref{lem:reconstruction.0.1} in \eqref{lem:reconstruction.1}, for every $v_h\in \mathcal{S}^{1,cr}_D(\mathcal{T}_h)$ with $\pi_hv_h=0$ a.e.\ on $\Gamma_N\cup\Gamma_C$, we arrive at
		\begin{align}\label{lem:reconstruction.2}
			(y_h-\widehat{y}_h,\nabla_h v_h)_{\Omega}=(\Pi_h y_h-\Pi_h \widehat{y}_h,\nabla_h v_h)_{\Omega}=0\,.
		\end{align}
		On the other hand, due to $\textup{div}(y_h-\widehat{y}_h)=0$ in $T$ for all $T\in \mathcal{T}_h$, we have that ${y_h-\widehat{y}_h\in(\mathcal{L}^0(\mathcal{T}_h))^d}$.
		%As a result,  
		By \eqref{lem:reconstruction.2} and the discrete Helmholtz--Weyl decomposition (\textit{cf}.\ \cite[Sec.\ 2.4]{BW21}), %\eqref{eq:decomposition},
		we conclude that 
		\begin{align*}
			y_h-\widehat{y}_h\in %(\nabla_h(\mathcal{S}^{1,cr}_0(\mathcal{T}_h))) ^\perp=
			\textup{ker}(\textup{div}|_{\mathcal{R}T^0(\mathcal{T}_h)})\,,
		\end{align*} 
		and, thus, $y_h\hspace*{-0.05em}\in\hspace*{-0.05em} \mathcal{R}T^0(\mathcal{T}_h)$ with \eqref{lem:reconstruction.0.2}.~In~addition, 
		for every $v_h\hspace*{-0.05em}\in\hspace*{-0.05em} \mathcal{S}^{1,cr}_D(\mathcal{T}_h)$~with~${\pi_h v_h\hspace*{-0.05em}=\hspace*{-0.05em}0}$~a.e.~on~$\Gamma_C$, the discrete integration-by-parts formula~\eqref{eq:pi} and \eqref{lem:reconstruction.0.0} yield that
		\begin{align}\label{lem:reconstruction.3}
			\begin{aligned} 
				(y_h\cdot n, \pi_h v_h)_{\Gamma_N}&=(\Pi_h y_h,\nabla_h v_h)_{\Omega}+(\textup{div}\,y_h,\Pi_h v_h)_{\Omega}
				\\&=(\overline{y}_h,\nabla_h v_h)_{\Omega}-(f_h,\Pi_h v_h)_{\Omega}\\&= 
				(g_h, \pi_h v_h)_{\Gamma_N}\,.
			\end{aligned}
		\end{align}
		Thus, choosing $v_h\hspace*{-0.1em}=\hspace*{-0.1em}\varphi_S$ for all $S\hspace*{-0.1em}\in\hspace*{-0.1em} \mathcal{S}_h^{\Gamma_N}$ in \eqref{lem:reconstruction.3} and exploiting that $\pi_h\varphi_S\hspace*{-0.1em}=\hspace*{-0.1em}\chi_S$, for~every~$S\hspace*{-0.1em}\in\hspace*{-0.1em} \mathcal{S}_h^{\Gamma_N}$, we find that \eqref{lem:reconstruction.0.3} is satisfied. 
	\end{proof} 
	
	{\small\setlength{\bibsep}{0pt plus 0.0ex}
%	\bibliographystyle{aomplain}
%	\bibliography{references}
		\providecommand{\bysame}{\leavevmode\hbox to3em{\hrulefill}\thinspace}
		\providecommand{\noopsort}[1]{}
		\providecommand{\mr}[1]{\href{http://www.ams.org/mathscinet-getitem?mr=#1}{MR~#1}}
		\providecommand{\zbl}[1]{\href{http://www.zentralblatt-math.org/zmath/en/search/?q=an:#1}{Zbl~#1}}
		\providecommand{\jfm}[1]{\href{http://www.emis.de/cgi-bin/JFM-item?#1}{JFM~#1}}
		\providecommand{\arxiv}[1]{\href{http://www.arxiv.org/abs/#1}{arXiv~#1}}
		\providecommand{\doi}[1]{\url{https://doi.org/#1}}
		\providecommand{\MR}{\relax\ifhmode\unskip\space\fi MR }
		% \MRhref is called by the amsart/book/proc definition of \MR.
		\providecommand{\MRhref}[2]{%
			\href{http://www.ams.org/mathscinet-getitem?mr=#1}{#2}
		}
		\providecommand{\href}[2]{#2}

	}
\end{document}